\documentclass[12pt,reqno]{amsart}
\pdfoutput=1
\usepackage[table]{xcolor}
\usepackage{colortbl}
\usepackage{booktabs}
\usepackage{multirow}
\usepackage{subfig}
\usepackage[headings]{fullpage}
\usepackage{amssymb,bbm}
\usepackage{epstopdf}
\usepackage{graphicx}
\usepackage{texdraw}
\usepackage{url}
\usepackage[all]{xy}
\usepackage[T1]{fontenc}
\usepackage{mathtools}
\usepackage{float}   
\usepackage[bookmarks=true,%
    colorlinks=true,%
    linkcolor=blue,%
    citecolor=blue,%
    filecolor=blue,%
    menucolor=blue,%
    urlcolor=blue,%
    breaklinks=true]{hyperref}
\usepackage{stmaryrd}

\newtheorem{theorem}{Theorem}

\theoremstyle{definition}

\newtheorem{remark}[theorem]{Remark}
\newtheorem{corollary}[theorem]{Corollary}
\newtheorem{conjecture}[theorem]{Conjecture}

\def\BN{\mathbbm N}
\def\BZ{\mathbbm Z}
\def\BQ{\mathbbm Q}
\def\BR{\mathbbm R}
\def\BC{\mathbbm C}

\def\BP{\mathbbm P}

\def\IZ{\mathbb{Z}}
\def\calA{\mathcal A}
\def\calB{\mathcal B}
\def\calC{\mathcal C}
\def\calD{\mathcal D}

\def\calJ{\mathcal J}

\def\calL{\mathcal L}

\def\s{\sigma}

\def\SL{\mathrm{SL}}

\def\tq{\tilde{q}}
\def\tx{\tilde{x}}

\def\={\;=\;}
\def\+{\,+\,}
\def\-{\,-\,}

\def\be{\begin{equation}}
\def\ee{\end{equation}}

\def\Li{\mathrm{Li}}
\def\e{\bold e}

\def\Z{\Bbb Z}

\def\th{\theta}

\def\sma#1#2#3#4{\bigl(\smallmatrix#1&#2\\#3&#4\endsmallmatrix\bigr)}

\def\ve{\varepsilon}
\def\bb{\mathsf{b}}
\def\ri{\mathrm{i}}
\def\rd{\mathrm{d}}

\def\tx{\tilde{x}}

\def\re{{\rm e}}
\def\ri{{\rm i}}
\def\rd{{\rm d}}
\newcommand{\figref}[1]{Fig.~\protect\ref{#1}}
\def\SU{\mathrm{SU}}

\def\Im{\mathrm{Im}}

\def\diag{\mathrm{diag}}
\def\ub{u_{\mathsf{b}}}
\def\bJ{\mathbf J}
\def\DS{Q} 
\def\RED{\mathrm{red}}
\def\sma#1#2#3#4{\bigl(\smallmatrix#1&#2\\#3&#4\endsmallmatrix\bigr)}  

\def\DJ{\mathrm{DJ}}
\def\e{\bold e}

\def\ts{\tilde{s}}
\def\mc{\mathcal}
\def\md{\mathbf}
\def\mf{\mathfrak}

\def\ms{\mathsf}

\def\cO{O}
\def\IC{\mathbb{C}}

\def\IQ{\mathbb{Q}}
\def\IR{\mathbb{R}}
\def\IZ{\mathbb{Z}}

\newcommand{\knot}[1]{\md{#1}}
\newcommand{\fad}{\operatorname{\Phi}_{\mathsf{b}}}
\newcommand{\mb}{{\mathsf{b}}}

\DeclareMathOperator{\real}{Re}
\DeclareMathOperator{\imag}{Im}

\newcommand{\nn}{\nonumber \\}

\graphicspath{{figs/}}

\begin{document}
\title[Resurgence of Chern--Simons theory at the trivial flat
connection]{Resurgence of Chern--Simons theory at the trivial flat
  connection}%
\author{Stavros Garoufalidis}%
\address{
  International Center for Mathematics, Department of Mathematics \\
  Southern University of Science and Technology \\
  Shenzhen, China \newline
  {\tt \url{http://people.mpim-bonn.mpg.de/stavros}}}
\email{stavros@mpim-bonn.mpg.de}
\author{Jie Gu}
\address{D\'epartement de Physique Th\'eorique, Universit\'e de Gen\`eve \\
  Universit\'e de Gen\`eve, 1211 Gen\`eve 4, Switzerland}
\address{Shing-Tung Yau Center and School of Physics\\
Southeast University, Nanjing 210096, China}
\email{jie.gu@unige.ch}
\author{Marcos Mari\~no}
\address{Section de Math\'ematiques et D\'epartement de Physique Th\'eorique\\
Universit\'e de Gen\`eve, 1211 Gen\`eve 4, Switzerland  \newline
  {\tt \url{http://www.marcosmarino.net}}}
\email{marcos.marino@unige.ch}
\author{Campbell Wheeler}
\address{Max Planck Institute for Mathematics \\
         Vivatsgasse 7, 53111 Bonn, Germany \newline
  {\tt \url{https://guests.mpim-bonn.mpg.de/cjwh/}}}
\email{cjwh@mpim-bonn.mpg.de}
       


\date{5 November 2021}

\begin{abstract}
  Some years ago, it was conjectured by the first author that the
  Chern--Simons perturbation theory of a 3-manifold at the trivial
  flat connection is a resurgent power series. We describe completely
  the resurgent structure of the above series (including the location
  of the singularities and their Stokes constants) in the case of a
  hyperbolic knot complement in terms of an extended square matrix of
  $(x,q)$-series whose rows are indexed by the boundary parabolic
  $\SL_2(\BC)$-flat connections, including the trivial one. We use our
  extended matrix to describe the Stokes constants of the above
  series, to define explicitly their Borel transform and to identify
  it with state--integrals.  Along the way, we use our matrix to give
  an analytic extension of the Kashaev invariant and of the colored
  Jones polynomial and to complete the matrix valued holomorphic
  quantum modular forms as well as to give an exact version of the
  refined quantum modularity conjecture of Zagier and the first
  author.  Finally, our matrix provides an extension of the 3D-index
  in a sector of the trivial flat connection. We illustrate our
  definitions, theorems, numerical calculations and conjectures with
  the two simplest hyperbolic knots.
\end{abstract}

\maketitle

{\footnotesize
\tableofcontents
}


\section{Introduction}
\label{sec.intro}

\subsection{Resurgence of Chern--Simons perturbation theory}
\label{sub.intro}

Quantum Topology originated by Jones's discovery of the famous
polynomial invariant of a knot~\cite{Jones}, followed by Witten's
3-dimensional interpretation of the Jones polynomial by means of a
gauge theory with a topological (i.e., metric independent)
Chern--Simons action~\cite{witten-jones}.  The connection between this
topological quantum field theory and the Jones polynomial appears both
on the level of the exact partition function and its perturbative
expansion which both determine, and are determined by, the (colored)
Jones polynomial. Indeed, the exact partition function on the
complement of a knot colored by the defining representation of the
gauge group $\SU(2)$ at level $k$ coincides with the value of the
Jones polynomial at the complex root of unity $e^{2\pi\ri/(k+2)}$. On
the other hand, the perturbative expansion along the trivial flat
connection $\s_0$ is a formal power series
$\Phi^{(\s_0)}(h) \in \BQ[[h]]$ whose coefficients are Vassiliev knot
invariants which are determined by the colored Jones polynomial of a
knot expanded as a power series in $h$ where
$q=e^h$~\cite{Bar-Natan}. More generally, the loop expansion of the
colored Jones polynomial is a formal power series
$\Phi^{(\s_0)}(x,h) \in \BQ(x)[[h]]$ introduced by
Rozansky~\cite{Rozansky:bureau} and further studied by
Kricker~\cite{Kricker, GK:noncommutative}, where $x=q^n$ plays the role of the
monodromy of the meridian. An important feature of the power series
$\Phi^{(\s_0)}(x,h)$ is that it is determined by (but also uniquely
determines) the colored Jones polynomial.  Likewise, the power series
$\Phi^{(\s_0)}(h)$ is determined by (and determines) the Kashaev
invariant of a knot~\cite{K95}, interpreted as an element of the
Habiro ring~\cite{Habiro:WRT}.

In~\cite{Ga:resurgence} the first author conjectured that the factorially divergent
formal power series $\Phi^{(\s_0)}(h)$ is resurgent, whose Borel transform has
singularities arranged in a peacock pattern, and can be re-expanded
in terms of the perturbative series $\Phi^{(\s)}(h)$ corresponding
to the remaining non-trivial flat connections of the Chern-Simons action. Although this
is a well-defined statement, resurgence was a bit of the surprise and a mystery.
We should point out that the above series are well-defined (for $\s\neq \s_0$ via
formal Gaussian integration using as input an ideal triangulation of a
3-manifold~\cite{DG}, and for $\s=\s_0$ using the Kashaev invariant itself) and their
coefficients are (up to multiplication by a power of $2\pi\ri$) algebraic numbers.
However a numerical computation of their coefficients is difficult (about 280
coefficients can be obtained for the simplest hyperbolic knot), hence it is difficult
to numerically study them beyond the nearest to the origin singularity of their Borel
transform.

The resurgence question has attracted a lot of attention in mathematics and
mathematical physics and some aspects of it were discussed by Jones~\cite{Jones:origin},
Witten~\cite{Witten:2010cx}, Gukov, Putrov and the third author~\cite{gmp},
Costin and the first author~\cite{CG} and Sauzin~\cite{Sauzin:resurgent}. Further
aspects of resurgence in Chern--Simons
theory were studied in~\cite{mmreview,gmp,gh-res,GZ:qseries,GZ:kashaev}.

When $s\neq \s_0$, the resurgence structure of the series
$\Phi^{(\s)}(h)$ was given explicitly in~\cite{GGM:resurgent}, where it was
found that the location of the singularities was arranged in a peacock pattern, and
the Stokes constants were integers. The latter were fully described by an $r \times r$
matrix $\bJ^\RED(q)$.
The passage from a vector $(\Phi^{(\s)}(h))_\s$ of power series to a matrix $\bJ^\RED(q)$
is inevitable, and points out to the possibility that the non-perturbative partition
function of a theory yet-to-be defined and its corresponding perturbative expansion is
matrix-valued and \emph{not} vector-valued, as was discussed in detail
in~\cite{GZ:kashaev} and~\cite{GZ:qseries}.
Let us summarise some key properties of the matrix $\bJ^\RED(q)$. 

\noindent
{\bf Linear $q$-difference equation.} 
The entries of $\bJ^\RED(q)$ are $q$-series with integer coefficients defined for
$|q| \neq 1$. The matrix $\bJ^\RED(q)$ is a fundamental solution of a linear
$q$-difference equation of order $r$, and its rows are labeled by the set of
nontrivial $\s$.

\noindent
{\bf Asymptotics in sectors: $q$-Stokes phenomenon.} 
The function $\bJ^\RED(e^{2\pi\ri\tau})$ as $\tau$ approaches zero in a fixed cone,
has a full asymptotic expansion as a sum of power series in $\tau$, times power
series in $e^{-2\pi\ri/\tau}$. However, passing from one cone to an adjacent one changes
the $e^{-2\pi\ri/\tau}$-series. The dependence of the asymptotics on the cone is the
$q$-Stokes phenomenon, analogous to the well-studied Stokes phenomenon in the theory
of linear differential equations with polynomial coefficients 
(see, e.g.,~\cite{Sibuya}). In our case, the $q$-Stokes phenomenon is a consequence 
of the fact that $\bJ^\RED(q)$ is a fundamental matrix solution to a linear 
$q$-difference equation. 

\noindent
{\bf Analyticity.} 
The product
$W(\tau)$ of $\bJ^\RED(\tq)$ with a diagonal automorphy factor and with $\bJ^\RED(q)$,
when $q=e^{2\pi\ri \tau}$ and $\tq=e^{-2\pi\ri/\tau}$, although defined for
$\tau \in \BC\setminus\BR$, equals to a matrix of state-integrals and hence it
analytically extends to $\tau$ in the cut plane $\BC'=\BC\setminus (-\infty, 0]$.
A distinguished $(\s_1,\s_1)$
entry of $W(\tau)$, where $\s_1$ is the geometric representation of a
hyperbolic 3-manifold, is the Andersen--Kashaev state-integral~\cite{AK}. The
latter is often identified with the unknown partition function of complex Chern--Simons
theory. Thus, analyticity of $W$ is interpreted as a factorisation property of
state-integrals, or as a matrix-valued holomorphic quantum modular
form~\cite{GZ:kashaev,Z:HQMF}.
 
\noindent
{\bf Borel resummation.} 
The matrix $W(\tau)$ coincides (in a suitable ray) to the Borel resummation of the
matrix of perturbative series. In particular, the Borel resummation of the perturbative
series is \emph{not} a $q$-series 
as has been claimed repeatedly in some physics literature, but rather
a bilinear combination of $q$-series and $\tq$-series.\footnote{A similar phenomenon
  was observed by Hatsuda--Okuyama~\cite{Hatsuda-Okuyama}.}

\noindent
{\bf Relation with the 3D-index.}
The 3D-index of Dimofte--Gaiotto--Gukov can be expressed bilinearly in terms of
$\bJ^\RED(q)$ and $\bJ^\RED(q^{-1})$. A detailed conjecture is given in
see~\cite[Conj.4]{GGM:peacock}.

\noindent
{\bf $x$-extension.} 
There is an extension of the above invariants by a nonzero complex number $x$,
which measures the monodromy of the meridian in the case of a knot complement, and
extends the $q$-series to functions of $(x,q)$, where $x$ behaves like a Jacobi
variable. This results in a matrix $\bJ^\RED(x,q)$ whose properties extend those of
the matrix $\bJ^\RED(q)$ and were studied in detail in~\cite{GGM:peacock}.

\subsection{A summary of our results}
\label{sub.summary}

Our goal is to describe the Stokes constants and the resurgent structure of the missing
asymptotic series $\Phi^{(\s_0)}(h)$ in terms of completing the
matrix $\bJ^\RED(x,q)$ to a square matrix with one extra row (namely
$(1,0,\dots,0)^T$) and column, whose distinguished $(\s_0,\s_1)$ entry
is conjecturally the Gukov--Manolescu series~\cite{Gukov-Manolescu}
(evaluated at $x=1$), and the remaining series in the top row are the
descendants of the Gukov-Manolescu series.

Along the way of solving the resurgence problem for the $\Phi^{(\s_0)}(h)$ series,
we solve several related problems, which we now discuss.

\noindent $\bullet$ {\bf A $q$-series that sees $\Phi^{(\s_0)}(h)$.}
This is a problem raised by Gukov and his collaborators
(see e.g.~\cite{Gukov:BPS,Gukov-Manolescu}). More precisely, our Resurgence
Conjecture~\ref{conj-exact-ra} implies that the asymptotics as $q=e^{2\pi\ri\tau}$
and $\tau \to 0$ in a sector of each of the $q$-series of the top row of the
matrix $\bJ(q)$ is a linear combination of the $\Phi^{(\s)}(h)$ series which includes
the $\Phi^{(\s_0)}(h)$ series. 

\noindent $\bullet$ {\bf A matrix-valued holomorphic quantum modular
  form.}  In~\cite{GZ:qseries} the first author and Don Zagier studied
a matrix $\bJ^\RED(q)$ of $q$-series with rows indexed by nontrivial
flat connections, and conjectured that the corresponding value of the
cocycle $\bJ(\tq)^{-1} \Delta(\tau) \bJ(q)$\footnote{For a suitable
  diagonal matrix $\Delta(\tau)$ of weights.} at
$S=\sma 0{-1}10 \in \SL_2(\BZ)$, which a priori is an analytic
function on $\BC\setminus\BR$, actually extends to the cut plane
$\BC'$. A problem posed was to find an extension of the matrix
$\bJ^\RED(q)$ which includes the trivial flat connection. We do so in
Sections~\ref{sub.41q3}, \ref{sub.41qx} and \ref{sub.52q3} for the
$\knot{4}_1$ and $\knot{5}_2$ knots.

\noindent $\bullet$ {\bf An exact form of the Refined Quantum Modularity Conjecture.}
In~\cite{GZ:kashaev} a Refined Quantum Modularity Conjecture was formulated. The
conjecture was numerically motivated by a smoothed optimal summation of the
divergent series $\Phi^{\s)}(\tau)$, and the final result was a matrix-valued 
periodic function defined at the rational numbers. We conjecture that if we replace
the smoothed optimal truncation by the median Borel resummation, all asymptotic
statements in~\cite{GZ:kashaev} become exact equalities, valid for finite (and not
necessarily large) range of the parameters. 
  
\noindent $\bullet$ {\bf An analytic extension of the Kashaev invariant and of
the colored Jones polynomial.} 
A consequence of the above conjecture is an exact formula for the Kashaev
  invariant at rational points as a linear combinations of three
  smooth functions, multiplied by the top row of $\bJ$.

\begin{conjecture}
  \label{conj.ekashaev}
  For every knot $K$ and every natural number $N$ we have:
  \be
  \label{ekashaev}
  \langle K \rangle_N = \sum_{\s} c^K_\s N^{\delta_\s} 
  s_{\rm med}( \Phi^{(K,\s)})(\tfrac{1}{N})  
  \ee
  where $\delta_\s=3/2$ for $\s\neq \s_0$ and $\delta_{\s_0}=1$ and
  $(c^K_\s)$ is a vector of elements of the Habiro ring
  (tensor $\BQ$) evaluated at $q=1$, with $c^K_{\s_1}=c^K_{\s_0}=1$. 
\end{conjecture}
The vector $(c_\s)$ for the $\knot{4}_1$ knot appears in Sec.4.2 of~\cite{GZ:kashaev}
and also as the top row of the matrix of Eqn.(92), and for the $\knot{5}_2$ knot it
appears in Section 4.3 as well as the top row of the matrix of Eqn.(104) of ibid.

For the $\knot{4}_1$ and the $\knot{5}_2$ knots,
we find numerically that $c^{\knot{4}_1}_{\s_2}=0$, $c^{\knot{5}_2}_{\s_2}=0$ and
$c^{\knot{5}_2}_{\s_3}=-2$ in complete agreement with the results of~\cite{GZ:kashaev}.
A corollary of~\eqref{ekashaev} is the Volume Conjecture
$\langle K \rangle_N \sim N^{3/2} \Phi^{(K,\s)}(\tfrac{1}{N})$ to all
orders in $1/N$ as $N \to \infty$.

\begin{conjecture}
  \label{conj.ejones}
  For every knot $K$, there is a neighborhood $U^K$ of $0$ in the complex plane, such
  that for every natural number $N$ and for $u \in U^K$, we have 
  \begin{equation}
    \label{ejones}
  J^K_N(e^{\frac{2\pi\ri}{N}+\frac{u}{N}})=  \sum_{\s} \delta_\s(u,N) c^K_\s(\tx)
  s_{\rm med}( \Phi^{(K,\s)})(e^u;\tau)  
\end{equation}
  where $\delta_{\s}(u,N)=\tau^{-1/2} \frac{\tx^{1/2}-\tx^{-1/2}}{x^{1/2}-x^{-1/2}}$
  for $\s \neq \s_0$ and $\delta_{\s_0}(x,\tau)=1$, where $x=e^u$, $\tx=e^{u/x}$,
  $\tau=\tfrac{u}{2\pi\ri N}+\tfrac{1}{N}$, and $c^K_\s(\tx) \in \BQ[\tx^{\pm 1}]$
  with $c^K_{\s_1}(\tx)=c^K_{\s_0}(\tx)=1$. 
\end{conjecture}
For the $\knot{4}_1$ and the $\knot{5}_2$ knots, we find numerically that
$c^{\knot{4}_1}_{\s_2}(\tx)=-\frac{\tx-\tx^{-1}}{2}$,
$c^{\knot{5}_2}_{\s_2}(\tx)=-\frac{\tx-\tx^{-1}}{2}$ and
$c^{\knot{5}_2}_{\s_3}(\tx)=-1-\tx$.

Since $\lim_{u\to 0} \delta_\s(N,u)=N^{\delta_s}$, the above conjecture specialises 
to Conjecture~\ref{conj.ekashaev} when $u \to 0$. Note also that the above conjecture 
implies the  Generalised Volume Conjecture when $u \not\in \pi \ri \BQ$ is fixed and 
$N \to \infty$. Indeed, $\delta(N,u)$ is nonzero and $J^{K}_N(e^{(2\pi\ri +u)/N}) 
\sim \delta(N,u)\Phi^{(\s_1)}(e^u;\tau)$. Note finally that the above conjecture 
explains the failure of exponential growth when $u$ is a rational multiple of $2\pi\ri$, 
known for all knots from theorems 1.10 and 1.11 of~\cite{GL:asy},
and theorem 5.3 of~\cite{Murakami2011:ivc} valid for the $\knot{4}_1$ knot.
Indeed, when $u=2\pi\ri r/s$ for integers $r$ and $s$ with $r/s$ near zero,
then $J^{K}_N(e^{(2\pi\ri +u)/N})$ is a periodic function of $N$ (see~\cite{Ha}), 
and so is $\delta(N,u)$ since $e^{u/\tau}=e^{2\pi\ri N r/(r+s)}$.
Moreover, $\delta(N,u)=0$ when $N$ is a multiple of $r+s$ which explains why in that
case the colored Jones polynomial does not grow exponentially.

\noindent $\bullet$ {\bf An extension of the 3D-index.}
Our completed matrix proposes a computable extension of the
  3D-index in the sector of the trivial connection $\s_0$, whose
  mathematical or physical definition is yet-to-be given.

\subsection{Challenges}
\label{sub.challenge}

Our solution to the above problems brings a new challenge: namely, the new square
matrix is actually a submatrix of a larger matrix $\bJ(x,q)$, one with block
triangular form which is a fundamental solution to the linear $q$-difference
equation satisfied by the descendants of the colored Jones
polynomials~\cite{GK:desc}. Already for the case of the $\knot{5}_2$ knot,
one obtains a $6 \times 6$ matrix instead of the original $3 \times 3$
matrix $\bJ^\RED(x,q)$, or of the completed $4 \times 4$ matrix.

A second challenge is to interpret the integers appearing in the new Stokes constants
associated to the trivial flat connection as BPS indices in the dual 3d super
conformal field theory. Incorporating the trivial connection in the
3d/3d correspondence of~\cite{DGG1} is subtle, but we expect our explicit results
to give hints on this problem.   

We should point out that although a proof of resurgence of the
asymptotic series $\Phi^{(\s)}(h)$ is \emph{still} missing, the
current paper (as well as the prior ones~\cite{GGM:resurgent,GGM:peacock}) provide
a complete description of
their resurgent structure (namely the location of the singularities
and a calculation of the Stokes constants) with precise statements,
complemented by extensive numerical computations (including a
numerical computation of the Stokes constants).  In addition, we
provide proofs of the algebraic properties of the matrices of
$q$-series and $(x,q)$-series.

\subsection{Illustration with the two simplest hyperbolic knots}
\label{sub.4152}

We will illustrate our ideas by giving a detailed description of these matrices
and of their algebraic, analytic and asymptotic properties for the case of the two
simplest hyperbolic knots, the $\knot{4}_1$ and the $\knot{5}_2$ knots. Let us summarise our findings
for the $\knot{4}_1$ knot.

\begin{itemize}
\item We complete the $2 \times 2$ matrix $\bJ^\RED(x,q)$ of
  $(x,q)$-series to the $3 \times 3$ matrix $\bJ(x,q)$ by adding the
  trivial flat connection. Our completed matrix is a fundamental
  solution of a third order linear $q$-difference equation.
\item A distinguished entry of $\bJ(x,q)$ is the Gukov--Manolescu
  series.
\item The matrix $\bJ(x,q)$ determines explicitly (but conjecturally)
  the Stokes constants and hence the resurgence structure of the three
  perturbative formal power series.
\item The matrix $\bJ(x,q)$ conjecturally computes an extension of the
  3D-index in a sector of the trivial flat connection.
\item We complete the $2 \times 2$ matrix of descendant
  Andersen--Kashaev state-integrals to a $3 \times 3$ matrix by adding
  new state-integrals which are implicit in work of Kashaev and show
  their bilinear factorisation property.
\end{itemize}

As a second example, we present our results for the $\knot{5}_2$ knot. In this case,
we complete the $3 \times 3$ matrix $\bJ^\RED(q)$ to a $4 \times 4$ one, and use it
to describe explicitly the Stokes constants of the $4$ asymptotic series in
half of the complex plane, thus completing the resurgence question of those
asymptotic series. However, the $\knot{5}_2$ knot reveals a new puzzle: the $4 \times 4$
matrix is a block of a $6 \times 6$ matrix whose rows are a fundamental solution
to a sixth order linear $q$-difference equation, namely the one satisfied by the
descendant colored Jones polynomial of the $\knot{5}_2$ knot~\cite[Eqn.14]{GK:desc}.
Although the homogeneous linear $q$-difference equation for the colored Jones polynomial
is fourth order, the one for the descendant colored Jones polynomial is sixth order,
and both equations are knot invariants. In the case of the $\knot{5}_2$ knot, the
extra $2 \times 2$ block is a matrix of modular functions, in fact of the famous
Rogers--Ramanujan modular $q$-hypergeometric series. We do not understand the
labeling of the two excess rows and columns (e.g., in terms of $\SL_2(\BC)$-flat
connections). Since the formulas for the $6 \times 6$ matrix appear rather complicated,
we will not give the $x$-deformation here, and postpone to a future publication
a systematic definition of the matrix of $(x,q)$-series for all knots.

We should point out that the definition of the top row of the $3 \times 3$ matrices
for the $\knot{4}_1$ knot, and the $6 \times 6$ matrix for the $\knot{5}_2$ knot, as well as an
extension of the above results to the case of closed hyperbolic 3-manifolds have
been taken from the forthcoming thesis of the last author.

\subsection*{Acknowledgements} 
The authors would like to thank Jorgen Andersen, Sergei Gukov, Rinat Kashaev, 
Maxim Kontsevich, Pavel Putrov and Matthias Storzer for enlightening conversations.
S.G. wishes to thank the University of Geneva for their hospitality during his visit
in the summer of 2021. The work of J.G. has been supported in part by the
NCCR 51NF40-182902 ``The Mathematics of Physics'' (SwissMAP). The work of M.M. has
been supported in part by the ERC-SyG project ``Recursive and Exact New
Quantum Theory" (ReNewQuantum), which received funding from the
European Research Council (ERC) under the European Union's Horizon
2020 research and innovation program, grant agreement No. 810573. The work of C.W. has been supported by the Max-Planck-Gesellschaft.


\section{The $\knot{4}_1$ knot}
\label{sec.41}

\subsection{A $2 \times 2$ matrix of $q$-series}
\label{sub.41q2}

In this section we recall in detail what is known about the resurgence
of the two asymptotic series of the $\knot{4}_1$ knot, labeled by the
geometric and the complex-conjugate flat connections. As explained in
the introduction, the answer is determined by a $2 \times 2$ matrix of
$q$-series which was discovered in a long story and in several stages
in a series of papers~\cite{GZ:qseries,
  GK:qseries,GGM:resurgent,GGM:peacock}. A detailed description of the
numerical discoveries and coincidences is given in~\cite{GZ:qseries}
and will not be repeated here. In that paper, the following pair of
$q$-series $G^{(j)}(q)$ for $j=0,1$ was introduced and studied by the
first author and Zagier~\cite{GZ:qseries}
\begin{equation}
\label{G01}
\begin{aligned}
G^{(0)}(q) &=
\sum_{n\geq0}(-1)^{n}\frac{q^{n(n+1)/2}}{(q;q)_{n}^{2}} \\
G^{(1)}(q) &=
\sum_{n\geq0}\left(n+\frac{1}{2}-2E_{1}^{(n)}(q)\right)(-1)^{n}
\frac{q^{n(n+1)/2}}{(q;q)_{n}^{2}} 
\end{aligned}
\end{equation}
where
\be
\label{Ek}
E_{k}^{(n)}(q) = \sum_{s=1}^{\infty}s^{k-1}\frac{q^{s(n+1)}}{1-q^{s}} \,.
\ee
These series were found to be connected to the $\knot{4}_1$ knot in at least two ways,
discussed in detail in~\cite{GZ:qseries}. On the one hand, they express
bilinearly the Andersen-Kashaev state-integral~\cite{GK:qseries} and the total
3D-index of Dimofte-Gaiotto-Gukov \cite{DGG2}. On the other hand, their radial
asymptotics as $q=e^{2\pi \ri \tau} \to 1$ (where $\tau$ is in a ray in the upper
half-plane)
is a linear combination of the two asymptotic series $\Phi^{(\s_1)}(\tau)$ and
$\Phi^{(\s_2)}(\tau)$ of the Kashaev invariant, where $\s_1$ is the geometric
representation of the fundamental group of the knot complement and $\s_2$ is
the complex conjugate. The resurgence of the factorially divergent
asymptotic series $\Phi^{(\s_1)}(\tau)$ and $\Phi^{(\s_2)}(\tau)$, including
a complete description of the Stokes automorphism and the Borel
resummation was given by the first three authors in~\cite{GGM:resurgent}. Surprisingly,
the Stokes matrices were expressed bilinearly in terms of a $2 \times 2$ matrix
of explicit descendant $q$-series whose definition we now give.
Consider the linear $q$-difference equation
\be
\label{41rec}
f_m(q) + (q^{m+1}-2) f_{m+1}(q) + f_{m+2}(q) = 0 \qquad (m \in \BZ) \,.
\ee
In~\cite{GGM:resurgent} it was shown that it has a
basis of solutions $G^{(j)}_m(q)$ for $j=1,2$ given by 
\footnote{$G_m^{(1)}(q)$ defined here is one half of $G^1_m(q)$ in
  \cite{GGM:resurgent}.}
\be
\label{Gm01}
\begin{aligned}
G_{m}^{(0)}(q) &=
\sum_{n\geq0}(-1)^{n}\frac{q^{n(n+1)/2}}{(q;q)_{n}^{2}}q^{mn}\\
G_{m}^{(1)}(q) &=
\sum_{n\geq0}\left(n+m+\frac{1}{2}-2E_{1}^{(n)}(q)\right)(-1)^{n}
\frac{q^{n(n+1)/2}}{(q;q)_{n}^{2}}q^{mn}
\end{aligned}
\ee
where $E_{k}^{(n)}(q)$ are as in Equation~\eqref{Ek}. Note that
$G_0^{(j)}(q)=G^{(j)}(q)$, and that $G^{(j)}_m(q) \in \BZ((q))$
are Laurent series in $q$ (with finitely many negative powers of $q$), meromorphic on
$|q|<1$ with only possible pole at $q=0$. We will extend them to analytic functions on
$|q| \neq 1$ by
\begin{equation}
  G_m^{(j)}(q^{-1}) = (-1)^{i}G^{(j)}_{-m}(q),\qquad j=0,1.
  \label{eq:Gmj-cn}
\end{equation}
The $2 \times 2$ matrix is given by 
$\bJ^\RED(q) = \bJ^\RED_{-1}(q) \sma 0{-2}1{-1}$ where
\begin{equation}
\label{Jqred}
\bJ^\RED_{m}(q) =
\begin{pmatrix}
G_{m}^{(1)}(q) & G_{m+1}^{(1)}(q)\\
G_{m}^{(0)}(q) & G_{m+1}^{(0)}(q)\\
\end{pmatrix} 
\end{equation}
coincides with the transpose of the matrix
$W_m(q)$ of~\cite[Eqn.(48)]{GGM:peacock} after interchanging of the
two rows. A complete description of the resurgent structure of the
series $\Phi^{(\s_j)}(\tau)$ for $j=0,1$, of their Borel resummation
and their expression in terms of a $2 \times 2$ matrix of
state-integrals (with one distinguished entry being the
Andersen--Kashaev state-integral~\cite{AK}) was given
in~\cite{GGM:resurgent,GGM:peacock}.

\subsection{A $3 \times 3$ matrix of $q$-series}
\label{sub.41q3}

In this section we define the promised $3 \times 3$ matrix of $q$-series
$\bJ^\RED_{m}(q)$ and give several algebraic properties thereof. 
In his forthcoming thesis, the fourth author introduced the series $G^{(2)}(q)$
\be
\label{G2}
G^{(2)}(q) =
\sum_{n\geq0}\left(\frac{1}{2}\left(n+\frac{1}{2}-2E_{1}^{(n)}(q)\right)^2
  -E_{2}^{(n)}(q)
  -\frac{1}{24}E_{2}(q)\right)(-1)^{n}\frac{q^{n(n+1)/2}}{(q;q)_{n}^{2}}
\ee
which is the coefficient of $\ve^2$ in the $\ve$-deformed $q$-series
\be
\label{G41epsilon}
G(q,\ve)= e^{-\ve^{2}\frac{E_{2}(q)}{24}}
\sum_{n=0}^{\infty}(-1)^{n}\frac{q^{n(n+1)/2}e^{(n+1/2)\ve}}{(qe^{\ve};q)_{n}^{2}}
	= 
	\sum_{k=0}^{\infty}G^{(k)}(q)\ve^{k}
\ee
which appears in~\cite{GZ:qseries}. Here, $E_2(q)=1-24E_2^{(0)}(q)$.
Adding the descendant variable $m \in \BZ$, leads to the $q$-series
\be
\label{Gm2}
G_{m}^{(2)}(q) =
\sum_{n\geq0}\left(\frac{1}{2}\left(n+m+\frac{1}{2}-2E_{1}^{(n)}(q)\right)^2
  -E_{2}^{(n)}(q)
  -\frac{1}{24}E_{2}(q)\right)(-1)^{n}\frac{q^{n(n+1)/2}}{(q;q)_{n}^{2}}q^{mn}
\ee
As in the case of $G_{m}^{(j)}(q)$ for $j=0,1$, it is a meromorphic function on
$|q|<1$ with only possible pole at $q=0$, and extends to an analytic function on
$|q|>1$ satisfying~\eqref{eq:Gmj-cn} with $j=2$.

The sequence $G_{m}^{(2)}(q)$ is a solution of the inhomogenous
equation obtained by replacing the right hand side of~\eqref{41rec} by
$1$. This follows easily by using creative telescoping of the theory
of $q$-holonomic functions implemented by
Koutschan~\cite{Koutschan:holofunctions}.

We can assemble the three sequences of $q$-series into a matrix
\be
\label{Jq}
\bJ_{m}(q) =
\begin{pmatrix}
1 & G_{m}^{(2)}(q) & G_{m+1}^{(2)}(q)\\
0 & G_{m}^{(1)}(q) & G_{m+1}^{(1)}(q)\\
0 & G_{m}^{(0)}(q) & G_{m+1}^{(0)}(q)\\
\end{pmatrix} 
\ee
whose bottom-right $2 \times 2$ matrix is $\bJ^\RED_{m}(q)$. The next theorem summarises
the properties of $\bJ_{m}(q)$.

\begin{theorem}
\label{thm.41bJ1}
  The matrix $\bJ_{m}(q)$ is a fundamental solution to the linear $q$-difference
  equation
\be
\label{41Jqrec}
\bJ_{m+1}(q) = \bJ_{m}(q) A(q^m,q), \qquad A(q^m,q)=
\begin{pmatrix}
1 & 0 & 1\\
0 & 0 & -1\\
0 & 1 & 2-q^{m+1}\\
\end{pmatrix} \,,
\ee
has $\det(\bJ_m(q))=-1$ and satisfies the analytic extension
\begin{equation}
\label{Jmqq}  
  \bJ_m(q^{-1}) =
  \begin{pmatrix}
    1&0&0\\
    0&-1&0\\
    0&0&1
  \end{pmatrix} \bJ_{-m-1}(q)
  \begin{pmatrix}
    1&0&0\\
    0&0&1\\
    0&1&0
  \end{pmatrix} \,.
\end{equation}
\end{theorem}

\begin{proof}
Equation~\eqref{41Jqrec} follows from the fact the last two rows of $\bJ_m(q)$ are   
solutions of the $q$-difference equation~\eqref{41rec} and the first is a solution
of the corresponding inhomogenous equation. Moreover, the block form of $\bJ_m(q)$
implies that $\det(\bJ_m(q))=\det(\bJ^\RED_m(q))=-1$ where the last equality follows
from~\cite[eq.~(14)]{GGM:resurgent}. Equation~\eqref{Jmqq} follows from the
fact that all three sequences of $q$-series satisfy~\eqref{eq:Gmj-cn}.
\end{proof}

We now give the inverse matrix of $\bJ_m(q)$ in terms of Appell-Lerch like sums.
The latter appear curiously in the mock modular forms and the meromorphic Jacobi
forms of Zwegers~\cite{Zwegers:thesis}, and in~\cite{DMZ:blackholes}.

\begin{theorem}
\label{thm.41bJ2}  
We have
\be
\label{Jqinv}
\bJ_{m}(q)^{-1} =
\begin{pmatrix}
1 & L_{m}^{(0)}(q) & -L_{m}^{(1)}(q)     \\
0 & -G_{m+1}^{(0)}(q) & G_{m+1}^{(1)}(q) \\
0 & G_{m}^{(0)}(q) & -G_{m}^{(1)}(q)
\end{pmatrix} 
\ee
for the $q$-series $L_m^{(j)}(q)$ ($j=0,1$) defined by
\be
\label{L12}
\begin{aligned}
L^{(0)}_{m}(q) &=
G_{m+1}^{(0)}(q)G_{m}^{(2)}(q)-G_{m}^{(0)}(q)G_{m+1}^{(2)}(q)
\\
L_{m}^{(1)}(q) &=
G_{m+1}^{(1)}(q)G_{m}^{(2)}(q)-G_{m}^{(1)}(q)G^{(2)}_{m+1}(q) \,.
\end{aligned}
\ee
The $q$-series $L_m^{(j)}(q)$ are expressed in terms of Appell-Lerch type sums:
  \be
  \label{Lm}
  \begin{aligned}
L^{(0)}_{m}(q) &=
2E_{1}^{(0)}(q)-1-m+
\sum_{n=1}^{\infty}(-1)^{n}\frac{q^{n(n+1)/2}}{(q;q)_{n}^{2}}\frac{q^{mn+n}}{1-q^{n}}
\\
L_{m}^{(1)}(q) &=
-\frac{3}{8}-2E_{1}^{(0)}(q)^{2}+2E_{1}^{(0)}(q)-E_{2}^{(0)}(q)
-\frac{1}{24}E_{2}(q)+2mE_{1}^{(0)}(q)-m-\frac{m^{2}}{2}
\\
&+\sum_{n=1}^{\infty}(-1)^{n}\frac{q^{n(n+1)/2}}{(q)_{n}^{2}}
\frac{q^{mn+n}}{1-q^{n}}\left(n+m+\frac{1}{2}-2E^{(n)}_{1}(q)+\frac{1}{1-q^{n}}\right)
\,.
\end{aligned}
\ee
\end{theorem}
\begin{proof}
Since $\bJ^\RED_{m}(q)$ is a $2 \times 2$ matrix with determinant $-1$, it follows
that the inverse matrix $\bJ_{m}(q)^{-1}$ is given by~\eqref{Jqinv} for the 
$q$-series $L_m^{(j)}(q)$ ($j=0,1$) given by~\eqref{L12}.

Observe that $A(q^m,q)$ has first column $(1,0,0)^t$, first row $(1,0,1)$, and
the remaining part is a companion matrix. It follows that its inverse matrix
has first column $(1,0,0)^t$ and first row $(1,1,0)$. This, together
with~\eqref{41Jqrec} implies that

\begin{equation}
\bJ_{m+1}(q)^{-1} = A(q^m,q)^{-1} \bJ_{m}(q)^{-1} =
\begin{pmatrix}
1 & 1 & 0\\
0 & 2-q^{m+1} & 1\\
0 & -1 & 0\\
\end{pmatrix}
\bJ_{m}(q)^{-1} \,.
\end{equation}
It follows that $L_m^{(j)}(q)$ satisfy the first order inhomogeneous linear
$q$-difference equation

\begin{equation}
L^{(j)}_{m-1}(q)-L^{(j)}_{m}(q) =
G_{m}^{(j)}(q)  \qquad (j=0,1) \,.
\end{equation}
Let $\calL^{(0)}_m(q)$ denote the right hand side of the top Equation~\eqref{Lm}.  
Then we have
\[
\calL^{(0)}_{m-1}(q)-\calL^{(0)}_{m}(q)
=
1+\sum_{n=1}^{\infty}(-1)^{n}\frac{q^{n(n+1)/2}}{(q)_{n}^{2}}\frac{q^{mn}-q^{mn+n}}{1-q^{n}}
=
G_{m}^{(0)}(q).
\]
Therefore $\calL^{(0)}_{m}(q)-L_{m}^{(0)}(q)$ is independent of $m$. Moreover,
$\lim_{m \to \infty} \calL^{(0)}_{m}(q)-L_{m}^{(0)}(q)=0$. The top part of
Equation~\eqref{Lm} follows.

Likewise, let $\calL^{(1)}_m(q)$ denote the right hand side of the bottom part of
Equation~\eqref{Lm}. Then we have

\begin{align*}
\calL^{(1)}_{m-1}(q)-\calL^{(1)}_{m}(q)
=&
\sum_{n=1}^{\infty}(-1)^{n}\frac{q^{n(n+1)/2}}{(q)_{n}^{2}}
\frac{q^{mn}-q^{mn+n}}{1-q^{n}}\left(n+m+\frac{1}{2}
  -2E^{(n)}_{1}(q)+\frac{1}{1-q^{n}}\right)\\
&-\sum_{n=1}^{\infty}(-1)^{n}\frac{q^{n(n+1)/2}}{(q)_{n}^{2}}
\frac{q^{mn}}{1-q^{n}}+m+\frac{1}{2}-2E^{(0)}_{1}(q)\\
=&\:
G^{(1)}_{m}(q) \,.
\end{align*}
Therefore $\calL^{(1)}_{m}(q)-L^{(1)}_{m}(q)$ is independent of $m$. Moreover,
$\lim_{m \to \infty} \calL^{(1)}_{m}(q)-L_{m}^{(1)}(q)=0$. Equation~\eqref{Lm} follows.
\end{proof}

\subsection{The $\Phi^{(\s_0)}(\tau)$ asymptotic series}
\label{sub.413Phi}

The $\knot{4}_1$ knot has three asymptotic series $\Phi^{(\s_j)}(\tau)$ for $j=0,1,2$
corresponding to the trivial flat connection $\s_0$, the geometric flat connection
$\s_1$ and its complex conjugate $\s_2$. The asymptotic series $\Phi^{(\s_j)}(\tau)$
for $j=1,2$ are defined in terms of perturbation theory of a state-integral~\cite{DG}
and can be computed via formal Gaussian integration in a way that was explained in
detail in~\cite{GGM:resurgent} and in~\cite{GZ:kashaev} and will not be repeated here.
They have the form
\begin{equation}
  \Phi^{(\s_j)}(\tau) = e^{\tfrac{V(\s_j)}{2\pi\ri\tau}}
  \varphi^{(\s_j)}(\tau)
  ,\quad j=1,2,
\end{equation}
where
\begin{equation}
  V(\s_1) = -V(\s_2) = \ri \text{Vol}(\knot{4}_1)
  = \ri 2\text{Im} \Li_2(\re^{\ri \pi/3}) = \ri 2.029883\ldots,
\end{equation}
with $\text{Vol}(\knot{4}_1)$ being the hyperbolic volume of
$S^3\backslash \knot{4}_1$, and $\varphi^{(\s_1)}(\tau/(2\pi\ri))$ is
a power series with algebraic coefficients with first few terms
\begin{equation}
  \varphi^{(\s_1)}(\tfrac{\tau}{2\pi\ri})
  = 3^{-1/4}\left(
    1+\frac{11\tau}{72\sqrt{-3}}
    +\frac{697\tau^2}{2(72\sqrt{-3})^2}+\ldots\right)
\end{equation}
(a total of 280 terms have been computed), while
$\varphi^{(\s_2)}(\tau) = \ri\varphi^{(\s_1)}(-\tau)$.

We now discuss the new series $\varphi^{(\s_0)}(\tau) \in \BQ[[\tau]]$ corresponding
to the zero volume $V(\s_0)=0$ trivial flat connection. This series can be defined
and computed (for any knot) using either the colored Jones polynomial or the Kashaev
invariant. Let us recall how this works.

Let $J_n(q) \in \BZ[q^{\pm 1}]$ denotes the Jones polynomial colored by the
$n$-dimensional irreducible representation of $\mathfrak{sl}_2$, and normalised to $1$
at the unknot. Setting $q=e^h$, one obtains a power series in $h$, whose coefficient
of $h^k$ is a polynomial in $n$ of degree at most $k$. In other words, we have
\be
\label{Jnq}
J_n(e^h) = \sum_{i=0}^\infty \sum_{j=0}^i a_{i,j} n^j h^i \in \BQ[[n,h]]
\ee
where $a_{i,j}$ depends on the knot and, as the knot varies, defines a Vassiliev
invariant of type (i.e., degree) $i$~\cite{Bar-Natan}. Then, the perturbative series
$\varphi^{(\s_0)}(\tau)$ is given by 
\be
\label{phi0h}
\varphi^{(\s_0)}(\tau) = \sum_{i=0}^\infty a_{i,0} \tau^i \,.
\ee
With this definition, to compute the coefficient of $\tau^k$ in $\varphi^{(\s_0)}(\tau)$,
one needs to compute the first $k$ colored Jones polynomials $J_n(e^h)$ for
$k =1, \dots, n$ up to $O(h^{k+1})$, polynomially interpolate and extract the
coefficient $a_{k,0}$. An efficient computation of the colored Jones polynomial
is possible if one knows a recursion relation with respect to $n$ (such a relation
always exists~\cite{GL}) together with some initial conditions. This gives a polynomial
time algorithm to compute $J_n(e^h)+O(h^{k+1})$.

An alternative method is the so-called
loop expansion of the colored Jones polynomial
\be
\label{Jnqloop}
J_n(e^h) = \sum_{\ell=0}^\infty \frac{P_\ell(x)}{\Delta(x)^{2\ell+1}} h^\ell \in
\BZ[x^{\pm 1}, \Delta(x)^{-1}][[h]]
\ee
where $x=q^n=e^{n h}$ and $\Delta(x) \in \BZ[x^{\pm 1}]$ is the Alexander polynomial
of the knot. This expansion was introduced by Rozansky~\cite{Rozansky:bureau} (see also
Kricker~\cite{Kricker} and~\cite{GK:noncommutative}) and it is related to the
Vassiliev power series expansion~\eqref{Jnq} by 
\be
\label{Jnqresum}
\sum_{k=0}^\infty a_{\ell+k,k} h^k= \frac{P_\ell(e^h)}{\Delta(e^h)^{2\ell+1}} \,.
\ee
Then the perturbative series $\varphi^{(\s_0)}(\tau)$ is given in terms of the
loop expansion by
\be
\label{phi0hloop}
\varphi^{(\s_0)}(\tau) = \sum_{\ell=0}^\infty P_\ell(1) h^\ell 
\ee
as follows from the above equations together with the fact that $\Delta(1)=1$.

A third method uses a theorem of Habiro~\cite{Habiro:sl2, Habiro:WRT} which lifts
the Kashaev invariant of a knot to an element of the Habiro ring
$\widehat{\BZ[q]} = \varprojlim \;\Z[q]/((q;q)_n)$. There is a canonical
ring homomorphism $\widehat{\BZ[q]} \to \BZ[[h]]$ defined by $q \mapsto e^h$,
which sends $(q;q)_n$ to $(-1)^n h^n + O(h^{n+1})$ and the image of the lifted
element of the Habiro ring under this homomorphism equals to the series
$\varphi^{(\s_0)}(h)$. For the case of the $\knot{4}_1$ knot, the corresponding
element of the Habiro ring is given by
\be
\label{hab41}
\sum_{n=0}^\infty (q;q)_n (q^{-1};q^{-1})_n 
\ee
and its expansion when $q=e^h$ gives the power series with first few terms 
\begin{equation}
  \varphi^{(\s_0)}
  (\tfrac{\tau}{2\pi\ri})=
  1-\tau^2+\frac{47}{12}\tau^4+\ldots.
\end{equation}

We end this section with a comment. 
Going back to the case of a general knot, it was shown in~\cite{GK:desc} that
the colored Jones polynomial is equivalent (in the sense of knot invariants) to
a descendant sequence of colored Jones polynomials and of Kashaev invariants
(indexed by the integers) which is $q$-holonomic. These descendant Kashaev invariants
play a key role in extending matrices of periodic functions whose rows and columns are
indexed by nonrtivial flat connections to a matrix that includes the trivial flat
connection. This is explained in detail in~\cite{GZ:kashaev}.
  
\subsection{Borel resummation and Stokes constants}
\label{sub.41borel}

In this section we discuss the asymptotic expansion as
$q = \re^{2\pi\ri\tau}\rightarrow 1$ of the vector $G(q)$ of $q$-series
and relate it to the vector $\Phi(\tau)$ of the asymptotic series,
where 
\begin{equation}
  G(q) =
  \begin{pmatrix}
    G^{(2)}(q) \\
    G^{(1)}(q) \\
    G^{(0)}(q)
  \end{pmatrix}, \qquad
  \Phi(\tau) =
  \begin{pmatrix}
    \Phi^{(\s_0)}(\tau)\\
    \Phi^{(\s_1)}(\tau)\\
    \Phi^{(\s_2)}(\tau)
  \end{pmatrix}
  \label{eq:Gq}
\end{equation}
with $G^{(0)}(q), G^{(1)}(q)$ given in~\eqref{G01}, and the
additional series $G^{(2)}(q)$ given in~\eqref{G2}.

The three power series $\Phi^{(\s_j)}(\tau)$, $j=0,1,2$ can be
resummed by Borel resummation.  On the other hand, according to the
resurgence theory, the value of the Borel resummation of an asymptotic
power series depends crucially on the argument of the expansion
variable.  If the Borel transform of the power series has a singular
point located at $\iota$, the values of the Borel resummation of the
power series whose expansion variable has an argument slightly greater
and less than the angle $\theta = \arg \iota$ differ by an
exponentially small quantity, called the Stokes discontinuity. Usually
the difference is identical with the Borel resummation of another
power series in the theory, a phenomenon called the Stokes
automorphism.

\begin{figure}
  \centering
  \includegraphics[height=6cm]{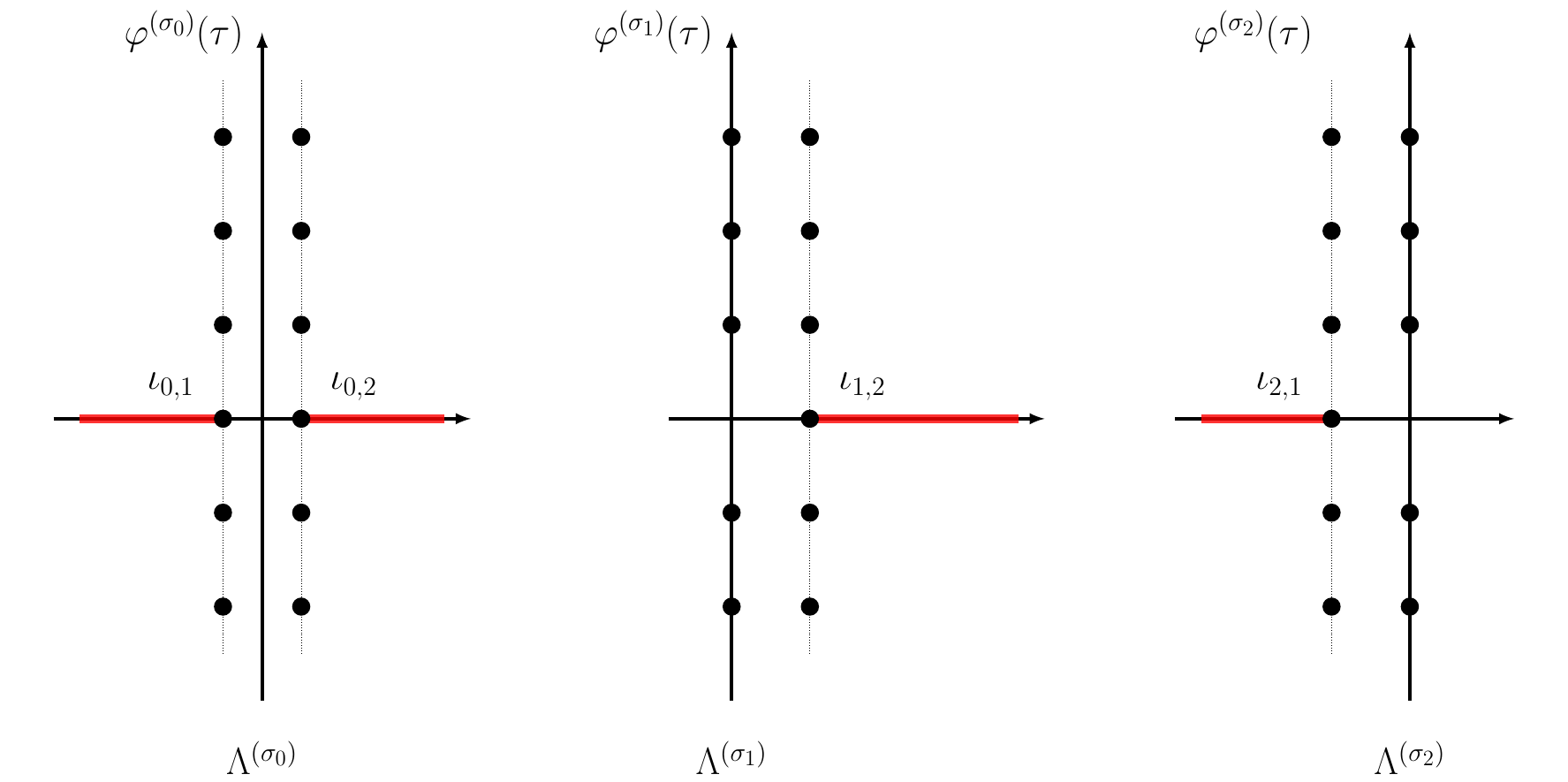}
  \caption{Singularities of the Borel transforms of
    $\varphi^{(\s_j)}(\tau)$ for $j=0,1,2$ of the knot $\knot{4}_1$.}
  \label{fg:41_brplane}
\end{figure}

In the case of the power series $\Phi^{(\s_j)}(\tau)$, $j=0,1,2$, the
singularities of the Borel transforms of $\Phi^{(\s_j)}(\tau)$,
$j=1,2$ were already studied in~\cite{GGM:resurgent,GGM:peacock}, and
they are located at
\begin{equation}
  \Lambda^{(\s_j)} = \{ \iota_{j,i} +2\pi\ri k\;|\; i=1,2,i\neq j,\,
  k\in\IZ\}\cup
  \{2\pi\ri k\;|\; k\in\IZ_{\neq 0}\},\quad j=1,2,
  \label{eq:Lmb12}
\end{equation}
as shown in the middle and the right panels of
Fig.~\ref{fg:41_brplane}, while the singularities of the Borel
transform of $\Phi^{(\s_0)}(\tau)$ are located at 
(see also \cite[Conj.~4]{Ga:resurgence})
\begin{equation}
  \Lambda^{(\s_0)} = \{\iota_{0,i} +2\pi\ri k\;|\; i=1,2,\,k\in\IZ\},
  \label{eq:Lmb0}
\end{equation}
as shown in the left panel of Fig.~\ref{fg:41_brplane}, where
\begin{equation}
  \iota_{j,i} = \frac{V(\s_j) - V(\s_i)}{2\pi\ri},\quad
  i,j=0,1,2.
\end{equation}
All the rays $\rho_\theta$ (Stokes rays) passing through the
singularities in the union
\begin{equation}
  \Lambda = \cup_{j=0,1,2} \Lambda^{(\s_j)},
  \label{eq:Lmb}
\end{equation}
form a peacock pattern, cf.~Fig.~\ref{fg:41_rays}, and they divide the
complex plane of Borel transform into infinitely many cones. The Borel
resummation of the vector $\Phi(\tau)$ is only well-defined within one
of these cones.  

Recall that the Borel transform $\widehat{\varphi}(\zeta)$ of a 
Gevrey-1 power series $\varphi(\tau)$ 
\begin{equation}
  \varphi(\tau) = \sum_{n=0}^\infty a_n\tau^n,\quad a_n = O(C^n n!),
\end{equation}
is defined by
\begin{equation}
  \widehat{\varphi}(\zeta) =\sum_{n=0}^\infty \frac{a_n}{n!}\zeta^n.
\end{equation}
If it analytically continues to an $L^1$-analytic function along the
ray $\rho_\theta:=\re^{\ri\theta}\IR_+$ where $\theta = \arg\tau$, we
define the Borel resummation by the Laplace integral
\begin{equation}
  s_\theta(\varphi)(\tau) =
  \int_0^\infty \widehat{\varphi}(\tau\zeta)
  \re^{-\zeta}\rd \zeta =
  \frac{1}{\tau}\int_{\rho_\theta}
  \widehat{\varphi}(\zeta)\re^{-\zeta/\tau}\rd
  \zeta.
  \label{eq:svarphi}
\end{equation}
The Borel resummation of the trans-series $\Phi(\tau) =
\re^{\frac{V}{2\pi\ri\tau}}\varphi(\tau)$ is defined to be
\begin{equation}
  s_\theta(\Phi)(\tau) =
  \re^{\frac{V}{2\pi\ri\tau}}s_\theta(\varphi)(\tau).
  \label{eq:sPhi}
\end{equation}
In the following we will also use the notation $s_R(\Phi)(\tau)$ when
the argument of $\tau$ is in the cone $R$ and it is a continuous
function of $\tau$.



Coming back to the vector of $q$-series $G(q)$, we find that the
asymptotic expansion of $G(q)$ when $q=e^{2\pi\ri\tau}$ and
$\tau \to 0$ in a cone $R$ can be expressed in terms of
$\Phi(\tau)$. Moreover, this asymptotic expansion can be lifted to an
exact identity between $q$-series $G^{(j)}(q)$ and linear combinations
of Borel resummation of $\Phi^{(\s_j)}(\tau)$ multiplied by power
series in $\tq = \re^{-2\pi\ri\tau^{-1}}$ (thought of as exponentially
small corrections) with integer coefficients. This is the content of
the following conjecture.

\begin{conjecture}
  \label{conj-exact-ra}
  For every cone $R \subset \BC\setminus \Lambda$ and every
  $\tau \in R$, we have
  \begin{equation}
    \Delta'(\tau) G(q) = M_R(\tq) \Delta(\tau)
    s_R(\Phi)(\tau),
    \label{eq:GMsPhi}
  \end{equation}
  where
  \begin{equation}
    \Delta'(\tau) = \diag(\tau^{3/2},\tau^{1/2},\tau^{-1/2}),\quad
    \Delta(\tau) = \diag(\tau^{3/2},1,1),
  \end{equation}
  and $M_R(\tq)$ is a $3\times 3$ matrix of $\tq$ (resp., $\tq^{-1}$)-series 
  if $\Im\tau>0$ (resp., $\Im\tau<0$) with integer coefficients that depend on
  $R$.
\end{conjecture}

\begin{figure}
  \centering
  \includegraphics[height=8cm]{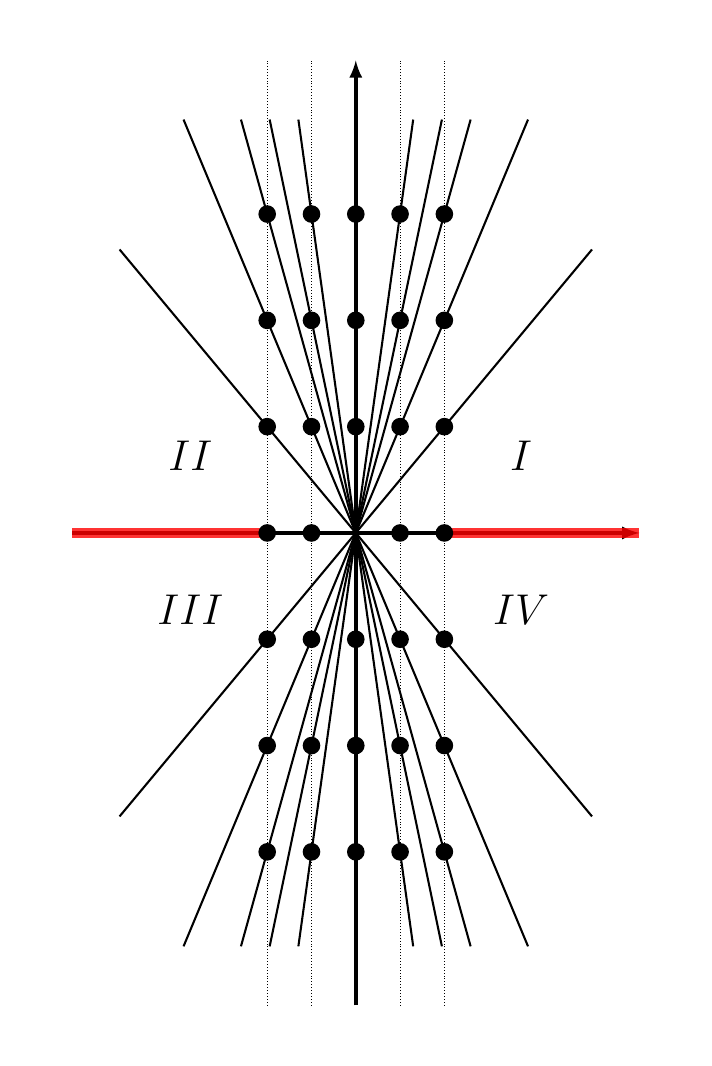}
  \caption{Stokes rays and cones in the $\tau$-plane for the 3-vector
    $\Phi(\tau)$ of asymptotic series of the knot $\knot{4}_1$.}
  \label{fg:41_rays}
\end{figure}

As in~\cite{GGM:resurgent,GGM:peacock}, we pick out in particular four
of these cones, located slightly above and below the positive or the
negative real axis, labeled in clockwise direction by $I,II,III,IV$ as
indicated in Fig.~\ref{fg:41_rays}.  We work out the exact matrices
$M_R(\tq)$ in these four cones.

\begin{conjecture}
\label{conj-exact-regions}
  Equation \eqref{eq:GMsPhi} holds in the cones $R=I,II,III,IV$ where
  the matrices $M_R(\tq)$ are given in terms of $\bJ_{-1}(\tq)$ as
  follows 
  \begin{subequations}
    \begin{align}
      M_I(\tq) &=
      \bJ_{-1}(\tq)
        \begin{pmatrix}
          1 & 0 & 0\\
          0 & 0 & -1\\
          0 & 1 & -1
        \end{pmatrix}, & |\tq|<1,
                  \label{eq:M1}\\
      M_{II}(\tq) &=
      \begin{pmatrix}
        1 & 0 & 0\\
        0 & -1& 0\\
        0 & 0 & 1
      \end{pmatrix}
                \bJ_{-1}(\tq)
                \begin{pmatrix}
                  1&0&0\\
                  0&1&0\\
                  0&1&-1
                \end{pmatrix}, & |\tq|<1,
                         \label{eq:M2}\\
      M_{III}(\tq) &=
                     \begin{pmatrix}
        1 & 0 & 0\\
        0 & -1& 0\\
        0 & 0 & 1
      \end{pmatrix}
      \bJ_{-1}(\tq)
        \begin{pmatrix}
          1&1&0\\
          0&-1&0\\
          0&2&1
        \end{pmatrix}, & |\tq|>1,
                         \label{eq:M3}\\      
      M_{IV}(\tq) &=
                    \bJ_{-1}(\tq)
    \begin{pmatrix}
      1&0&1\\
      0&0&-1\\
      0&1&2
    \end{pmatrix}, & |\tq| >1.
           \label{eq:M4}
    \end{align}
  \end{subequations}  
\end{conjecture}

We now discuss the Stokes automorphism.  To any singularity in the
Borel plane located at $\iota_{i,j}^{(k)}:= \iota_{i,j}+2\pi\ri k$, we
can associate a local Stokes matrix 
\begin{equation}
  \mf{S}_{\iota_{i,j}^{(k)}}(\tq) = I + \mc{S}_{i,j}^{(k)}\tq^k
  E_{i,j},\quad  \mc{S}_{i,j}^{(k)}\in\IZ,
\end{equation}
where $E_{i,j}$ is the elementary matrix with $(i,j)$-entry 1
($i,j=0,1,2$) and all other entries zero, and $\mc{S}_{i,j}^{(k)}$ is
the Stokes constant.  Let us assume the \emph{locality} condition that
no two Borel singularities share the same argument, or if there are,
their Stokes matrices commute.  This is indeed the case in our
example.  Then for any ray of angle $\theta$, the Borel resummations
of $\Phi(\tau)$ with $\tau$ whose argument is raised slight above
($\theta_+$) or lowered sightly below ($\theta_-$) $\theta$ are related
by the following formula of Stokes automorphism 
\begin{equation}
  \Delta(\tau) s_{\theta_+}(\Phi)(\tau) = \mf{S}_{\theta}(\tq)
  \Delta(\tau) s_{\theta_-}(\Phi)(\tau),\quad
  \mf{S}_\theta(\tq) = \prod_{\arg\iota=\theta}\mf{S}_\iota(\tq).
\end{equation}
Because of the locality condition, we don't have to worry about the
order of the product of local Stokes matrices.

More generally, consider two rays $\rho_{\theta^+}$ and
$\rho_{\theta^-}$ whose arguments satsify
$0<\theta^+-\theta^-\leq \pi$, we define the global Stokes
automorphism
\begin{equation}
  \Delta(\tau)s_{\theta^+}(\Phi)(\tau) =
  \mf{S}_{\theta^-\rightarrow \theta^+}(\tq)
  \Delta(\tau)s_{\theta^-}(\Phi)(\tau),
  \label{eq:S-glob}
\end{equation}
where both sides are analytically continued smoothly to the same value
of $\tau$.  The global Stokes matrix
$\mf{S}_{\theta^-\rightarrow\theta^+}(\tq)$ satisfies the
factorisation property~\cite{GGM:resurgent,GGM:peacock}
\begin{equation}
  \mf{S}_{\theta^-\rightarrow\theta^+}(\tq) =
  \prod_{\theta^-<\theta<\theta^+}^{\leftarrow}
  \mf{S}_{\theta}(\tq),
  \label{eq:S-fac}
\end{equation}
where the ordered product is taken over all the local Stokes matrices
whose arguments are sandwiched between $\theta^-,\theta^+$ and they
are ordered with rising arguments from right to left.

Given \eqref{eq:GMsPhi} with explicit values of $M_R(\tq)$ for
$R=I,II,III,IV$, in general we can calculate the global Stokes matrix via
\begin{equation}
  \mf{S}_{R\rightarrow R'}(\tq) = M_{R'}(\tq)^{-1}\cdot M_{R}(\tq).
  \label{eq:SMM}
\end{equation}
Here in the subscript of the global Stokes matrix on the left hand
side, $R$ stands for any ray in the cone.
For instance, we find that the global Stokes matrix from cone $I$
anti-clockwise to cone $II$ is 
\begin{equation}
  \mf{S}_{I\rightarrow II}(\tq) =
  \begin{pmatrix}
    1&0&0\\
    0&1&0\\
    0&1&-1
  \end{pmatrix}
  \bJ_{-1}(\tq)^{-1}
  \begin{pmatrix}
    1&0&0\\
    0&-1&0\\
    0&0&1
  \end{pmatrix}
  \bJ_{-1}(\tq)
  \begin{pmatrix}
    1&0&0\\
    0&0&-1\\
    0&1&-1
  \end{pmatrix},\quad |\tq|<1.\label{eq:S12}
\end{equation}
This Stokes matrix has the block upper triangular form
\begin{equation}
    \begin{pmatrix}
    1&*&*\\
    0&*&*\\
    0&*&*
  \end{pmatrix}.\label{eq:1bd}
\end{equation}
Let us note that this form implies that
$\Phi^{(\s_j)}(\tau)$ ($j=1,2$) form a closed subset under Stokes
automorphisms (this was called in \cite{GM:top-sting} a ``minimal
resurgent structure'').  They are controled by the $2\times 2$
submatrix of $\mf{S}_{I\rightarrow II}(\tq)$ in the bottom right and
one can verity that it is indeed the Stokes matrix
in~\cite{GGM:resurgent}.  In addition we can also extract Stokes
constants $\mc{S}_{0,j}^{(k)}$ ($j=1,2$, $k=1,2,\ldots$) responsible
for Stokes automorphisms into $\Phi^{(\s_0)}(\tau)$ from Borel
singularities in the upper half plane, and collect them in the
generating series
\begin{equation}
  \ms{S}^+_{0,j}(\tq) = \sum_{k=1}^\infty
  \mc{S}_{0,j}^{(k)}\tq^k,\quad j=1,2.
\end{equation}
We find
\begin{align}
  \ms{S}_{0,1}^+(\tq) = \ms{S}_{0,2}^+(\tq) =
  &-G_0^{(2)}(\tq) - G_1^{(2)}(\tq)
    +  \left(G_0^{(0)}(\tq)+G_1^{(0)}(\tq)\right)G_0^{(2)}(\tq)
    /G_0^{(0)}(\tq)\nn=
  &-\tq-2\tq^2-3\tq^3-7\tq^4-14\tq^5-34\tq^6+\ldots.
\end{align}

Similarly, we find that the global Stokes matrix from cone $III$
anti-clockwise to cone $IV$ is 
\begin{equation}
  \mf{S}_{III\rightarrow IV}(\tq) =
  \begin{pmatrix}
    1&0&0\\
    0&-1&1\\
    0&1&0
  \end{pmatrix}\cdot \bJ_{-1}(\tq^{-1})^{-1}\cdot
  \begin{pmatrix}
    1&0&0\\
    0&-1&0\\
    0&0&1
  \end{pmatrix}
  \cdot \bJ_{-1}(\tq^{-1})\cdot
  \begin{pmatrix}
    1&0&0\\
    0&1&0\\
    0&1&1
  \end{pmatrix},\quad |\tq|>1.\label{eq:S34}
\end{equation}
It also has the form as \eqref{eq:1bd}, and the $2\times 2$ submatrix of
$\mf{S}_{III\rightarrow IV}(\tq)$ in the bottom right is the Stokes matrix
given in~\cite{GGM:resurgent}. We also extract
Stokes constants $\mc{S}_{0,j}^{(k)}$ ($j=1,2$, $k=-1,-2,\ldots$)
responsible for Stokes automorphisms into $\Phi^{(\s_0)}(\tau)$ from
Borel singularities in the lower half plane, and collect them in the
generating series
\begin{equation}
  \ms{S}^-_{0,j}(\tq) = \sum_{k=-1}^{-\infty}
  \mc{S}_{0,j}^{(k)}\tq^k,\quad j=1,2.
\end{equation}
We find
\begin{align}
  \ms{S}_{0,2}^-(\tq) = -\ms{S}_{0,1}^-(\tq) = \ms{S}_{0,1}^+(\tq^{-1}).
\end{align}
We can also use \eqref{eq:SMM} to compute the global Stokes matrix
$\mf{S}_{IV\rightarrow I}(\tq)$ and we find
\begin{equation}
  \mf{S}_{IV\rightarrow I} = 
  \begin{pmatrix}
    1&0&1\\
    0&1&3\\
    0&0&1
  \end{pmatrix}.
\end{equation}
Note that this can be identified as $\mf{S}_{0}$ associated to the ray
$\rho_0$ and it can be factorised as
\begin{equation}
  \mf{S}_{0} = \mf{S}_{\iota_{0,2}}\mf{S}_{\iota_{1,2}},\quad
  \mf{S}_{\iota_{0,2}} =
  \begin{pmatrix}
    1&0&1\\
    0&1&0\\
    0&0&1
  \end{pmatrix},\quad
  \mf{S}_{\iota_{1,2}} =
  \begin{pmatrix}
    1&0&0\\
    0&1&3\\
    0&0&1
  \end{pmatrix}.
\end{equation}
Since the local Stokes matrices $\mf{S}_{\iota_{0,2}}$ and
$\mf{S}_{\iota_{1,2}}$ commute, the locality condition is satisfied.
We read off the Stoke discontinuity formulas
\begin{equation}
\label{disc-pra}
\begin{aligned}
  &\text{disc}_0 \Phi^{(0)}(\tau) = \tau^{-3/2}s(\Phi^{(s_2)})(\tau),\\
  &\text{disc}_0 \Phi^{(1)}(\tau) = 3s(\Phi^{(s_2)})(\tau),
\end{aligned}
\end{equation}
with 
\begin{equation}
  \text{disc}_\theta\Phi(\tau) = s_{\theta_+}(\Phi)(\tau) - s_{\theta_-}(\Phi)(\tau),
\end{equation}
and the second identity has already appeared in
\cite{gh-res,GGM:resurgent}.

Finally, in order to compute the global Stokes matrix
$\mf{S}_{II\rightarrow III}(\tq)$, we need to take into account that
the odd powers of $\tau^{1/2}$ on both sides of \eqref{eq:GMsPhi} give
rise to additional $-1$ factors when one crosses the branch cut at the
negative real axis, and \eqref{eq:SMM} should be modified by
\begin{equation}
  \mf{S}_{II\rightarrow III}(\tq) = \diag(1,-1,-1)
  M_{III}(\tq)^{-1}\cdot M_{II}(\tq),
\end{equation}
and we find
\begin{equation}
  \mf{S}_{II\rightarrow III} = 
  \begin{pmatrix}
    1&1&0\\
    0&1&0\\
    0&-3&1
  \end{pmatrix}.
\end{equation}
Similarly this can be identified as $\mf{S}_{\pi}$ associated to the
ray $\rho_{\pi}$ and it can be factorised as
\begin{equation}
  \mf{S}_{\pi} = \mf{S}_{\iota_{0,1}}\mf{S}_{\iota_{2,1}},\quad
  \mf{S}_{\iota_{0,1}} =
  \begin{pmatrix}
    1&1&0\\
    0&1&0\\
    0&0&1
  \end{pmatrix},\quad
  \mf{S}_{\iota_{2,1}} =
  \begin{pmatrix}
    1&0&0\\
    0&1&0\\
    0&-3&1
  \end{pmatrix}.
\end{equation}
Note that the local Stokes matrices $\mf{S}_{\iota_{0,1}}$ and
$\mf{S}_{\iota_{2,1}}$ also commute.  We read off the Stokes
discontinuity formulas
\begin{equation}
  \label{disc-prb}
\begin{aligned}
  &\text{disc}_\pi \Phi^{(0)}(\tau) =
  \tau^{-3/2}s(\Phi^{(s_1)})(\tau),\\
  &\text{disc}_\pi \Phi^{(2)}(\tau) = -3s(\Phi^{(s_1)})(\tau),
\end{aligned}
\end{equation}
where the second identity has already appeared in
\cite{GGM:resurgent}.

\subsection{The Andersen--Kashaev state-integral}
\label{sub.AK41}

In this section we briefly recall the properties of the state-integral of
Andersen--Kashaev for the $\knot{4}_1$ knot~\cite[Sec.11.4]{AK}, defined by
\begin{equation}
  \label{Z410}
  Z_{\knot{4}_1}(\tau) = 
  \int_{\BR+\ri 0} \Phi_\bb(v)^2 \, \re^{-\pi \ri v^2} \rd v, \qquad (\tau=\sqrt{\bb}) \,.
\end{equation}
Here, $\fad (z)$ is Faddeev's quantum dilogarithm~\cite{Faddeev}, in the conventions of
e.g.~\cite[Appendix A]{AK}.
With this choice of contour, the integrand is exponentially decaying at $\pm \infty$
hence the integral is absolutely convergent. State-integrals have several key features:
\begin{itemize}
\item
  They are analytic functions in $\BC'$.
\item
  Their restriction to $\BC\setminus\BR$ factorises bilinearly as finite sum of a
  product of a $q$-series and a $\tilde q$-series, where $q=\e(\tau)$ and
  $\tilde q=\e(-1/\tau)$; see~\cite{Beem,GK:qseries}.
\item
  Their evaluation at positive rational numbers also factorises bilinearly as a finite
  sum of a product of a periodic function of $\tau$ and a periodic function of
  $-1/\tau$; see~\cite{GK:evaluation}.
\item
  State-integrals are equal to linear combinations of the median Borel summation
  of asymptotic series.
\item
  State-integrals come with a descendant version which satisfies a linear
  $q$-difference equation.
\end{itemize}
Let us explain these properties for the state-integral~\eqref{Z410}. The integrand
is a quasi-periodic meromorphic function with explicit poles and residues. Moving the
contour of integration above, summing up the residue contributions, and using the fact
that there are no contributions from infinity, one finds
that~\cite[Cor.1.7]{GK:qseries} 
\be
\label{zlmuv0}
 Z(\tau)    =
    -\ri 
    \left(\frac{q}{\tq}\right)^{\frac{1}{24}}
    \left(\tau^{1/2}G^{(1)}(q)G^{(0)}(\tq) -
     \tau^{-1/2} G^{(0)}(q)G^{(1)}(\tq)\right), \qquad (\tau \in \BC\setminus\BR). 
\ee
When $\tau$ is a positive rational number, the quasi-periodicity of the integrand,
together with a residue calculation leads to a formula for $Z(\tau)$ given
in~\cite{GK:evaluation}. More generally, in~\cite{GGM:resurgent} we considered the descendant
integral
\begin{equation}
\label{eq.statelm}
 Z_{\lambda,\mu}(\tau) =
    \int_{\calD} \Phi_{\bb}(v)^2 \re^{-\pi\ri v^2 +
      2\pi(\lambda \bb - \mu \bb^{-1})v}\rd v,
\end{equation}
where $\lambda, \mu \in \BZ$ and the contour $\calD$ is asymptotic at
infinity to the horizontal line $\Im(v)=v_0$ where
$v_0 > |\real(\lambda\bb-\mu\bb^{-1})|$ but is deformed near the
origin so that all the poles of the quantum dilogarithm located at
\begin{equation}
  c_\bb + \ri \bb m +\ri \bb^{-1}n,\quad m,n\in\IZ_{\geq 0},
\end{equation}
are above the contour. 
These integrals factorise as follows:
\be
\label{zlmu}
 Z_{\lambda,\mu}(\tau)    =
    (-1)^{\lambda-\mu+1}\ri q^{\frac{\lambda}{2}}\tq^{\frac{\mu}{2}}
    \left(\frac{q}{\tq}\right)^{\frac{1}{24}}
    \left(\tau^{1/2}G^{(1)}_\lambda(q)G^{(0)}_\mu(\tq) -
     \tau^{-1/2} G^{(0)}_\lambda(q)G^{(1)}_\mu(\tq)\right). 
\ee
The above factorisation can be expressed neatly in matrix form. Indeed, let
us define 
\be
\label{Wgtred}
W^\RED_{S,\lambda,\mu}(\tau) = \bJ^\RED_\lambda(\tq)^{-1}
\diag(\tau^{3/2}, \tau^{1/2}, \tau^{-1/2}) \, \bJ^\RED_\mu(q) \,.
\ee
Using the $q$-difference equation~\eqref{41Jqrec}, it is easy to see that
$W^\RED_{S,\lambda+1,\mu}(\tau)=A^{-1}(-1/\tau) W^\RED_{S,\lambda,\mu}(\tau)$
and $W^\RED_{S,\lambda,\mu+1}(\tau)=W^\RED_{S,\lambda,\mu}(\tau) A(\tau)$ hence
the domain of $W^\RED_{S,\lambda,\mu}$ is independent of the integers $\lambda$
and $\mu$. Equation~\eqref{zlmu} implies that 
$W^\RED_{S,\lambda,\mu}(\tau)$ are given by the matrix $(Z_{\lambda+i,\mu+j}(\tau))$
(for $i,j =0,1$), up to left-multiplication by a matrix of automorphy factors.

Finally we discuss the relation between the Borel summation of the two
asymptotic series $\Phi^{(\sigma_j)}(h)$ for $j=1,2$ and the
descendant state-integrals. Since the Borel transform of those series
may have singularities at the positive real axis, we denote by
$s_{\rm med}$ their \emph{median resummation} given by the average of
the two Laplace transforms to the left and to the right of the
positive real axis. Then, we have
\begin{equation}
  \label{smed12}
  \begin{aligned}
    s_{\rm med} ( \Phi^{(\sigma_1)} ) (\tau) = &\ri(\tq/q)^{1/24}
    \left(-\frac{1}{2}Z_{0,0}(\tau)-\tq^{1/2}
      Z_{0,-1}(\tau)\right), \\
    s_{\rm med} ( \Phi^{(\sigma_2)} ) (\tau) = &\ri(\tq/q)^{1/24}
    Z_{0,0}(\tau) \,.
  \end{aligned}
\end{equation}

\subsection{A new state-integral}
\label{sub.AK41new}

In the previous section, we saw how the matrix $W^\RED(\tau)$ of
products of $q$-series and $\tq$-series~\eqref{Gm01} coincides with a
matrix of state-integrals.  Having found the $q$-series~\eqref{G2}
which complement the series~\eqref{Gm01}, it is natural to search for
a new state-integral which factorises in terms of all three $q$-series
$G_{m}^{(j)}(q)$ for $j=0,1,2$ and their $\tq$-versions.  Upon looking
carefully, the series $G_{m}^{(j)}(q)$ for $j=0,1$ were produced from
the Andersen--Kashaev state-integral because its integrand had a
double pole, hence the contributions came from
expanding~\eqref{G41epsilon} up to $O(\ve^2)$.  If we expanded up to
$O(\ve^3)$, we would capture the new series $G^{(2)}(q)$.  Hence the
problem is to find a state-integral of the $\knot{4}_1$ whose integrand has
poles of order $3$. After doing so, one needs to understand how this
story, which seems a bit ad hoc and coincidental to the $\knot{4}_1$ knot,
can generalise to all knots.  It turns out that such a state-integral
existed in the literature for many years, and in fact was devised by
Kashaev~\cite{kas-volume} as a method to convert the state-sums of the
Kashaev invariants into state-integrals, using as a building block the
Faddeev quantum dilogarithm function at rational numbers, multiplied
by $1/\sinh x$.  Incidentally, similar integrals have appeared in
\cite{kas-mar} and more recently in the work of two of the authors on
the topological string on local $\BP^2$;
see~\cite[Eqn.3.141]{GM:top-sting}. The integrand of such
state-integrals are meromorphic functions with the usual pole
structure coming from the Faddeev quantum dilogarithm function,
together with the extra poles coming from $1/\sinh x$. The residues of
the former give rise to products of $q$-series times $\tq$-series, but
the presence of of $1/\sinh x$ has two effects.  On the one hand, it
produces, in an asymmetric fashion, poles of the integrand of one
order higher, contributing to sums of $q$-series or $\tq$-series. On
the other hand, the produced $q$ and $\tq$-series look like
multidimensional Appell-Lerch sums.  An original motivation for
converting state-sum formulas for the Kashaev invariants into
state-integral formulas was to use such an integral expression for a
proof of the Volume Conjecture.

\begin{figure}[htpb!]
\leavevmode
\begin{center}
\includegraphics[height=5cm]{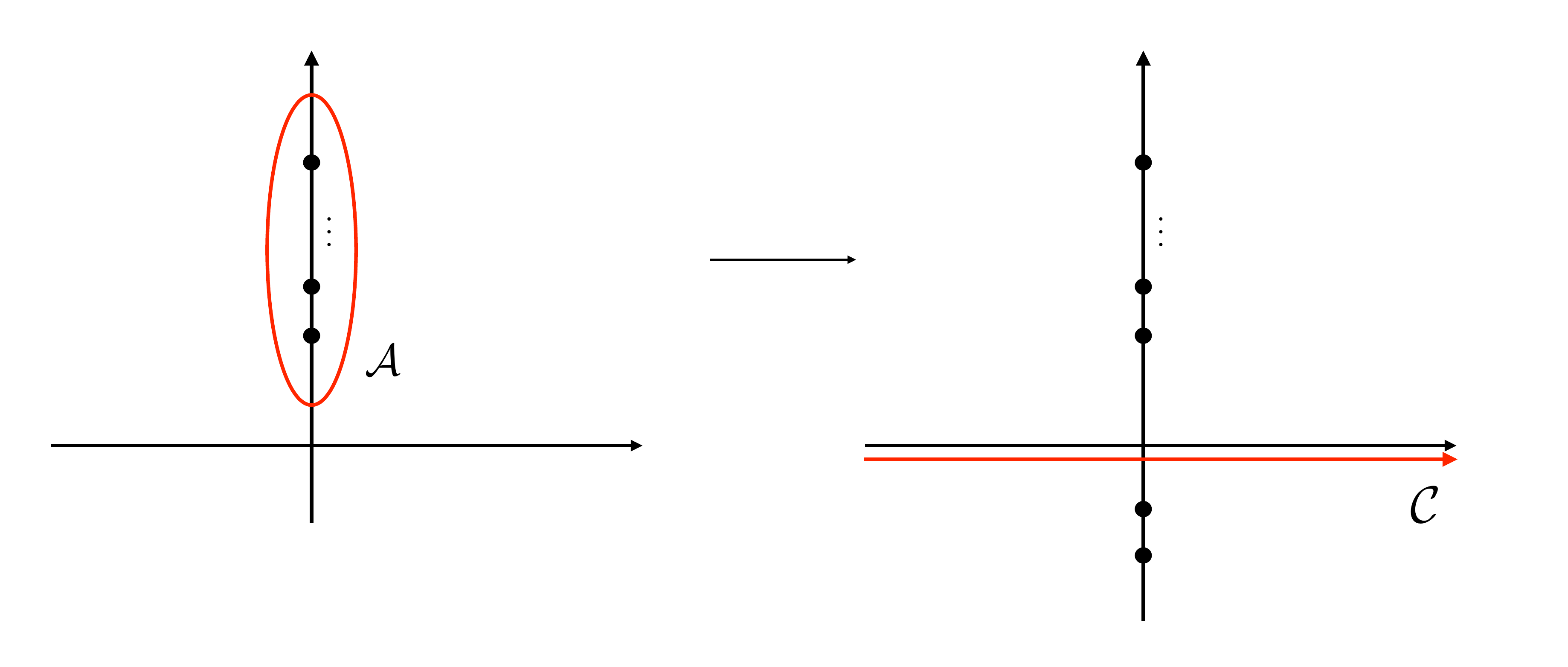}
\end{center}
\caption{The contour ${\mathcal A}_N$ appears in the integral formula \eqref{car}
  for the Kashaev invariant of the $\knot{4}_1$ knot, and it encircles the $N$
  poles~\eqref{jones-poles}. By doing the integral along the contour
  $\calC$ and picking the poles in the lower half plane, one obtains a new
  state-integral with information about the trivial connection.}
\label{contours-fig}
\end{figure} 
There are two examples that convert state-sums into state-integrals, one given
by Kashaev in~\cite{kas-volume} for the $\knot{4}_1$ knot and further studied by
Andersen--Hansen~\cite{AH}, and one in Kashaev--Yokota~\cite{KY:52} for the $\knot{5}_2$ knot. 
In the case of the $\knot{4}_1$ knot, the integral considered in~\cite{kas-volume, AH} is 
\begin{equation}
\label{car}
\langle \knot{4}_1 \rangle_N
=-{\ri \over 2 \mb^3} \int_{{\mathcal A}_N}
\tanh\left( {\pi y \over \mb} \right) {\fad \left( -y +{\ri \over 2 \mb} \right)
  \over \fad \left(y -{\ri \over 2 \mb} \right)} \rd y. 
\end{equation}
For generic $\bb^2\in\IC'$ so that $\real \bb >0$, the
integrand has the following poles and zeros, all in the imaginary axis: 
\begin{equation}
\begin{aligned}
  &\text{simple poles}: \left\{\ri\bb\left(\frac{1}{2}+m\right)
    \,\big|\,    m=0,1,2,\ldots\right\},\\
  &\text{double poles}: \left\{
    -\ri\bb\left(\frac{1}{2}+m\right)-\ri\bb^{-1}(1+n)
    \,\big|\, m,n=0,1,2,\ldots \right\},  \\
  &\text{triple poles}: \left\{-\ri\bb\left(\frac{1}{2}+m\right)
    \,\big|\, m=0,1,2,\ldots \right\},\\
  &\text{double zeros}: \left\{ \ri\bb\left(\frac{1}{2}+m\right)
    +\ri\bb^{-1}(1+n)\,\big|\,m,n=0,1,2,\ldots\right\} \,.
\end{aligned}
\end{equation}
In the special case where $\mb^2=N^{-1}$ where $N \in \BZ_{>0}$, which
is the case where \eqref{car} is well-defined, the poles and zeros
in the upper half plane conspire so that there are only finite many
simple poles located at
\begin{equation}
\label{jones-poles}
y_m=\ri \mb \left(m+{1\over 2} \right), \qquad m=0, \cdots, N-1, 
\end{equation}
and we can define the contour ${\mathcal A}_N$ encircling these points as
in \figref{contours-fig} (left). An application of the residue theorem
gives that this integral calculates the Kashaev invariant of the $\knot{4}_1$
knot,
\be
\label{kas-ex}
\langle \knot{4}_1 \rangle_N=\sum_{m=0}^{N-1} (-1)^m \xi^{-m(m+1)/2}
\prod_{\ell=1}^m (1- \xi^\ell)^2, \qquad \xi = \re^{2\pi \ri \over N}. 
\ee
Now, we can define a new analytic function by changing the contour of
integration from $\calA_N$ to the horizontal contour $\mathcal C$
slightly below the horizontal line $\imag(y) = \real(\bb^{-1})/2$,
\be
\label{itau-def}
{ \mc{Z}}(\tau)= -\frac{\ri}{2\bb^3}\int_{\mathcal C} 
\tanh\left( {\pi y \over \mb} \right) {\fad \left( -y +{\ri \over 2 \mb} \right)
  \over \fad \left(y -{\ri \over 2 \mb} \right)} \rd y.
\ee
This is now defined for $\tau= \mb^2 \in \BC'$. Although
both \eqref{car} and \eqref{itau-def} share the same integrand, it has
significant contributions from infinity in the upper half plane, so
that we cannot deform the contour $\mc{A}_N$ smoothly to the contour
$\mc{C}$, and \eqref{car} and \eqref{itau-def} are really different.
On the other hand, the integrand does have vanishing contributions
from infinity in the lower half plane. Consequently we can smoothly
deform the new controur $\mc{C}$ downwards, and collect the residues
of the integrand on the lower half-plane, as shown in
\figref{contours-fig} (right). This integral, in contrast to the
Andersen--Kashaev state-integral, contains information about the
trivial connection. In particular, we conjecture that, in the region
of the complex plane slightly above the positive real axis, the
all-orders asymptotic of $\mc{Z}(\tau)$ at $\tau=0$ is given by
\be
{ \mc{Z}}(\tau)\sim \Phi^{(\sigma_0)}(\tau). 
\ee
Moreover, this can be upgraded to an exact asymptotic formula by using Borel
resummation in the same region, and one has 
\be
\label{itau-borel}
{ \mc{Z}}(\tau)=  s( \Phi^{(\sigma_0)} )(\tau)
-{\ri \over 2}\tau^{-3/2}  s(\Phi^{(\sigma_2)} )(\tau) . 
\ee
It turns out that the change of contour in \figref{contours-fig} implements the
inversion of the Habiro series recently studied in~\cite{park-inverted}: the
integral over the contour ${\mathcal A}_N$ leads to the Habiro series, while the
integral over $\mathcal{C}$ gives the ``inverted'' Habiro series, see
also Section~\ref{sub.41statex}. This inversion between $q$-series and elements of
the Habiro ring was observed 10 years ago by the first author
in his joint work with  Zagier~\cite{GZ:qseries}, under the informal name
``upside-down cake''. 

\subsection{A $3 \times 3$ matrix of state-integrals}
\label{sub.3x3state}

Having found a new state-integral whose asymptotics sees the asymptotic series
$\Phi^{(\s_0)}(\tau)$, we now consider its descendants, and their factorisations
to complete the story. 
The new state-integrals $\mc{Z}_{\lambda,\mu}(\tau)$ are defined
as follows: 
\be
\label{inst-def}
{\mc{Z}}_{\lambda, \mu}(\tau) = -\frac{\ri}{2\bb^3}\int_{\mathcal C}
\tanh\left( {\pi y \over \mb} \right)
{\fad \left( -y +{\ri \over 2 \mb} \right) \over \fad
  \left(y -{\ri \over 2 \mb} \right)} \re^{-2 \pi (\lambda \mb- \mu \mb^{-1})} \rd y,  
\ee
where $\mb$ is related to $\tau$ by $\tau= \mb^2$ and $\lambda,\mu\in \BZ$. The
integration contour $ {\mathcal C}$ is chosen so that, at infinity, it is asymptotic
to the line ${\rm Im}(y)=y_2$, where $y_2$ satisfies 
\begin{equation}
  y_2 < \tfrac{1}{2}\real \bb^{-1} - |\real(\lambda\bb-\mu\bb^{-1})|.
\end{equation}
%
This guarantees convergence of the integral. We choose $\calC$ so that
all poles of the integrand in the lower half plane are below $\calC$. Note that
$\mc{Z}_{0,0}(\tau)=\mc{Z}(\tau)$ is the integral introduced in
\eqref{itau-def}, so that the state-integrals with general $\lambda,\mu$ are
descendants of $\mc{Z}(\tau)$.  

\begin{theorem}
  \label{thm-newst}
  The descendant state-integral \eqref{inst-def} can be expressed in terms of
  the series (\ref{Gm01}), (\ref{Gm2}) as follows: 
  \begin{equation}
    \label{Zlmfac}
  \begin{aligned}
    {\mc{Z}}_{\lambda, \mu}(\tau) &= q^{\lambda/2} (-1)^{\mu}
    \left(  G_\lambda^{(2)}(q)+\tau^{-1}
      G_\lambda^{(1)}(q) L_\mu^{(0)}(\tq)
      -\tau^{-2}  G_\lambda^{(0)}(q) L_\mu^{(1)}(\tq) \right) \\
    & +\frac{1}{2}   q^{\lambda/2} (-1)^{\mu} \left( \tau^{-1}
      G_\lambda^{(1)}(q) G_\mu^{(0)}(\tq)
      - \tau^{-2}  G_\lambda^{(0)}(q) G_\mu^{(1)}(\tq) \right) 
\end{aligned}
\end{equation}
\end{theorem}

\begin{proof}
  This follows by applying the residue theorem to the
  state-integral~\eqref{inst-def}, along the lines of the proof of
  Theorem 1.1 in~\cite{GK:qseries}. One closes the contour to encircle the poles in
  the lower half-plane, located at 
  \be
  y_{m,n}= -{\ri \mb \over 2} -\ri m \mb -\ri n \mb^{-1}, \qquad m,n\ge 0. 
  \ee
  The poles of the integrand come the poles and the zeros of the quantum dilogarithm
  as well as from the $\tanh$ function. When $n=0$ they are triple (a double pole
  comes from the quantum dilogarithm and a simple pole from $\tanh$), while those
  with $n>0$ are double, coming only from the quantum dilogarithm. The triple poles
  lead to the series $G_\lambda^{(2)}(q)$. In order to obtain the final result, one
  also has to use the properties of $E_2(q)$ under modular transformations, i.e. 
  \be
  E_2(\tq)= \tau^2 \left( E_2 (q) + {12 \over 2 \pi \ri \tau} \right). 
  \ee
\end{proof}
 
\begin{remark}
  The state-integral \eqref{inst-def} can be evaluated for arbitrary rational values
  of $\tau$ by using the techniques of~\cite{GK:evaluation}. One finds for example,
  for $\mb^2=1$, 
 \be
 \mc{Z} (1)= -2 \sinh^2 \left( {V \over 4 \pi} \right), 
 \ee
 where $V$ is the hyperbolic volume of $\knot{4}_1$. 
\end{remark} 

\begin{remark} 
Equation~\eqref{itau-borel} can be written as 
\be
\label{nsi-median}
\mc{Z}(\tau)= \, s_{\rm med} ( \Phi^{(\sigma_0)} ) (\tau),
\qquad \tau >0. 
\ee
\end{remark}
 
We now discuss an important analytic extension of the matrix $\bJ_\mu(q)$
defined for $|q| \neq 1$. We define 
\be
\label{Wgt}
W_{S,\lambda,\mu}(\tau) = \bJ_\lambda(\tq)^{-1}
\diag(\tau^{3/2}, \tau^{1/2}, \tau^{-1/2}) \, \bJ_\mu(q)
\qquad (\tau \in \BC\setminus\BR)\,.
\ee
As in Section~\ref{sub.AK41}, we find that the domain of $W_{S,\lambda,\mu}$ is
independent of the integers $\lambda$ and $\mu$. 

\begin{theorem}
  \label{thm.2}
  $W_{S,\lambda,\mu}(\tau)$ extends to a holomorphic function on
  $\BC'$ and equals to the matrix $(Z_{\lambda+i,\mu+j}(\tau))$
(for $i,j =0,1,2$), up to left-multiplication by a matrix of automorphy factors.
\end{theorem}

\begin{proof}
  For the bottom block of four entries, this result is already known
  from~\cite{GGM:resurgent,GGM:peacock}, and it follows from \eqref{zlmu} as
  was discussed in Section~\ref{sub.AK41}. The top
  two non-trivial entries $(\s_0,\s_j)$ of $W_{S,\lambda,\mu}(\tau)$ for $j=1,2$
  are given by  
\begin{equation}
  \tau^{3/2}\left(G_{\mu-1+j}^{(2)}(q)
    +\tau^{-1} G_{\mu-1+j}^{(1)}(q)L_{\lambda}^{(0)}(\tq) - 
    \tau^{-2} G_{\mu-1+j}^{(0)}(q)L_{\lambda}^{(1)}(\tq) \right)\,.
\end{equation}
In view of Theorem \ref{thm-newst} and~\eqref{zlmu} they can be written as a sum of
state-integrals ${\mc{Z}}_{\lambda, \mu}(\tau)$ and ${\mc{Z}}_{\lambda, \mu+1}(\tau)$,
multiplied by holomorphic factors. This proves the theorem.
\end{proof} 


\section{The $x$-variable}
\label{sub.xvariable}

In this section we discuss an extension of the results of Section~\ref{sec.41}
adding an $x$-variable. In the context of the $n$th colored Jones polynomial,
$x=q^n$ corresponds to an eigenvalue of the meridian in the asymptotic expansion of
the Chern--Simons path integral around an abelian representation of a knot complement.
In the context of the state-integral of Andersen-Kashaev~\cite{AK}, the $x$-variable
is the monodromy of a peripheral curve. The corresponding state-integral factorises
bilinearly into holomorphic blocks, which are functions of $(x,q)$ and $(\tx,\tq)$ \cite{Beem}.
In the context of quantum modular forms, $x$ plays the role of a Jacobi variable.

The corresponding perturbative series are now $x$-deformed
(see~\cite[Sec.5.1]{GGM:peacock}), but there are some tricky aspects of this
deformation that we now discuss. The critical points of the
action, after exponentiation, lie in a plane curve $S$ in $(\BC^*)^2$ (the so-called
spectral curve) defined over the rational numbers, where $(\BC^*)^2$ is equipped with
coordinate functions $x$ and $y$. The field $\BQ(S)$ of rational functions of $S$
(assuming $S$ is irreducible, or working with one component of $S$ at a time) can be
identified with $\BQ(x)[y]/(p(x,y))$ where $p(x,y)=0$ is the defining polynomial of $S$. 
The coefficients of the perturbative series are elements of $(\BQ(S)^*)^{-1/2} \BQ(S)$
and the perturbative series are labeled by the branches of the projection $S \to \BC^*$
corresponding to $(x,y) \mapsto y$ (with discriminant $\delta$, a rational function
on $S$). Each such branch $\s$ defines locally an algebraic function
$y=y^\s=y^\s(x) \in \overline{\BQ}(x)$ satisfying the equation $p(x,y^\s(x))=0$,
which gives rise to  an embedding of the field of $\BQ(S)$ to the field
$\overline{\BQ}(x)$ of algebraic functions obtained by replacing $y$ by $y^\s(x)$. 
For each such branch $\s$, the perturbative series has the form 
\be
\label{Phix}
\Phi^{(\s)}(x,\tau) = e^{\tfrac{V^{\s}(x)}{2\pi\ri\tau}} \varphi^{(\s)}(x,\tau)
\ee
where $\varphi^{(\s)}(x,\tau) \in \frac{1}{\sqrt{i\delta_\s(x)}}
\overline{\BQ}(x)[[2\pi\ri\tau]]$. The volume $V^{\s}(x)$ is also a function of $x$
given explicitly as a sum of dilogarithms and products of logarithms.

In the above discussion it is important to keep in mind that the asymptotic
series~\eqref{Phix} are labeled by branches of the finite ramified covering
$S \to \BC^*$. Going around a loop in $x$-space that avoids the finitely many
ramified points \emph{will} change the labeling of the $y=y(x)$ branches, and
correspondingly of the asymptotic series. In the present paper (as well as
in~\cite{GGM:peacock}), we define the asymptotic series in a neighborhood
of $x \sim 1$ of the geometric representation, and we do not discuss the
$x$-monodromy question. 

In the case of the $\knot{4}_1$ knot, the asymptotic series associated to the
geometric, and the conjugate flat connections are given by
\begin{equation}
\label{2phi41}
\begin{aligned}
  \varphi^{(\s_1)}(x;\tfrac{\tau}{2\pi\ri}) &=
  \frac{1}{\sqrt{\delta(x)}} \left(1-\frac{\ri
      (x^{-3}-x^{-2}-2x^{-1}+15-2x-x^2+x^3)} {24\delta(x)^3}\tau
    +\ldots \right)\\
  \varphi^{(\s_2)}(x;\tfrac{\tau}{2\pi\ri}) &= \frac{\ri}
  {\sqrt{\delta(x)}}
  \left(1+\frac{\ri(x^{-3}-x^{-2}-2x^{-1}+15-2x-x^2+x^3)}
    {24\delta(x)^3}\tau+\ldots \right)
\end{aligned}
\end{equation}
where
\begin{equation}
  \delta(x) = \sqrt{-x^{-2}+2x^{-1}+1+2x-x^2}.
\end{equation}
The corresponding perturbative series are defined by
\begin{equation}
\label{eq:Phi1}
  \begin{aligned}
    \Phi^{(\s_1)}(x;\tau) =
    & \re^{\tfrac{A(x)}{2\pi\ri\tau}}
      \varphi^{(\s_1)}(x;\tau),\\
    \Phi^{(\s_2)}(x;\tau) =
    & \re^{-\tfrac{A(x)}{2\pi\ri\tau}}
      \varphi^{(\s_2)}(x;\tau),
\end{aligned}
\end{equation}
where
\begin{equation}
  \label{eq:Axr}
  A(x) = \frac{1}{2} \log(t)^2 + 2 \log(t) \log(x) + \log(x)^2
    + \Li_2(-t x) + 
 \Li_2(-t)+ \frac{\pi^2}{6} + \pi\ri\log(x),
\end{equation}
with $t(x) = \frac{-1-x^{-1}+x -\ri \delta(x)}{2}$ being a solution to
the equation $(t+x^{-1})+(t+x^{-1})^{-1}=x+x^{-1}-1$. Note that when
$x=1$, $\delta(1)=\sqrt{3}$, $t(1)=-\frac{1+\ri \sqrt{3}}{2}$ and
$\Phi^{(\s_j)}(1;\tau)=\Phi^{(\s_j)}(\tau)$, the latter defined in
Section~\ref{sub.413Phi}.

\subsection{The $\Phi^{(\s_0)}(x,\tau)$ series}
\label{sub.41Phi0xq}

We begin by discussing the perturbative series
$\varphi^{(\s_0)}(x,\tau)$ which is a formal power series in
$2\pi\ri\tau$ whose coefficients are rational functions of $x$ with
rational coefficients. The series is defined by the right hand side of
Equation~\eqref{Jnqloop} after setting $h=2\pi\ri\tau$.
One way to compute the $\ell$-th coefficient of that series
is by computing the colored Jones polynomial, expanding in $n$ and $h$
as in~\eqref{Jnq} and then resumming as in~\eqref{Jnqresum}, taking
into account the fact that the latter sum is a rational function. An
alternative way is by using Habiro's expansion of the colored Jones
polynomials~\cite{Habiro:sl2} (see also~\cite{Habiro:WRT})
\begin{equation}
\label{JHx}
J^K(x,q) = \sum_{k=0}^\infty c_{k}(x,q) H^K_k(q), \qquad
c_{k}(x,q) = x^{-k}(q x;q)_k (q^{-1}x;q^{-1})_k 
\end{equation}
where $H^K_k(q) \in \BZ[q^\pm]$ are the Habiro polynomials of the knot $K$
and $J^K(q^n,q)$ is the $n$th colored Jones polynomial. The latter
can be efficiently computed using a recursion (which always exists~\cite{GL}) together
with initial conditions. This is analogous to applying the WKB method to
a corresponding linear $q$-difference equation~\cite{DGLZ,Ga:diff}.
We comment that the colored Jones polynomials of a knot $K$ have a descendant
version defined by~\cite{GK:desc}
\be
\label{DJdef}
DJ^{K,(m)}(x,q) = \sum_{k=0}^\infty c_{k}(x,q) H^K_k(q) \, q^{km}, \qquad (m \in \BZ)\,.
\ee
Correspondingly, the Kashaev invariant has a descendant version $DJ^{K,(m)}(1,q)$
(an element of the Habiro ring) and the asymptotic series 
$\Phi^{(\s_0)}(x,h)$ have a descendant version
$\Phi^{(\s_0)}_m(x,h)$ defined for all integers $m$ in~\cite{GK:desc}, which we will
not use in the present paper. 

Going back to the case of the $\knot{4}_1$ knot, we have
\begin{align}
\label{Phi410x}
\varphi^{(\s_0)}(x;\tfrac{\tau}{2\pi\ri}) &=
    -\frac{1}{x^{-1}-3+x}
    -\frac{x^{-1}-1+x}{(x^{-1}-3+x)^4}\tau^2 \\ &
    -\frac{x^{-4}+14x^{-3}+64x^{-2}-156x^{-1}
    +201-156x+64x^2+14x^3+x^4}
    {12(x^{-1}-3+x)^7}\tau^4
    +\ldots
\end{align}
and the corresponding perturbative series is given by $\Phi^{(\s_0)}(x;\tau) =
\varphi^{(\s_0)}(x;\tau)$.

\subsection{A $3\times 3$ matrix of $(x,q)$-series}
\label{sub.41qx}


We now extend the results of Section~\ref{sub.41q3} by including the Jacobi
variable $x$ which, on the representation side, determines the monodromy of the
meridian of an $\SL_2(\BC)$ representation $\s$. 

Our first task is to define the $3 \times 3$ matrix $\bJ_{m}(x,q)$.
For $|q| \neq 1$, we define
\be
\label{ABC}
\begin{aligned}
  C_{m}(x,q) &=
  \sum_{k=0}^{\infty}(-1)^{k}\frac{q^{k(k+1)/2+km}}{(x^{-1};q)_{k+1}(x;q)_{k+1}}
\\
A_{m}(x,q) &=
\sum_{k=0}^{\infty}(-1)^{k}\frac{q^{k(k+1)/2+km}x^{k+m}}{(q;q)_{k}(x^{2}q;q)_{k}}
\\
B_{m}(x,q) &=A_{m}(x^{-1},q).
\end{aligned}
\ee
Our series $C_m(x,q)$ contain as a special case the series
$F_{\knot{4}_1}(x,q)$ in \cite{Gukov-Manolescu,Park:2020edg,park-inverted}
\begin{equation}
  F_{\knot{4}_1}(x,q) = (x^{1/2}-x^{-1/2}) C_0(x,q).
  \label{eq:FK}
\end{equation}
We assemble these $(x,q)$-series into a matrix
\begin{equation}
\bJ_{m}(x,q)
=
\begin{pmatrix}
1 & C_{m}(x,q) & C_{m+1}(x,q)\\
0 & A_{m}(x,q) & A_{m+1}(x,q)\\
0 & B_{m}(x,q) & B_{m+1}(x,q)\\
\end{pmatrix}
\end{equation}
whose bottom-right $2 \times 2$ matrix is $\bJ^\RED_{m}(x,q)$. The properties of
$\bJ_{m}(x,q)$ are summarised in the next theorem.
\begin{theorem}
\label{thm.41bJx1}
  The matrix $\bJ_{m}(x,q)$ is a fundamental solution to the linear $q$-difference
  equation
\be
\label{41Jxqrec}
\bJ_{m+1}(x,q) = \bJ_{m}(x,q) A(x,q^m,q), \qquad A(x,q^m,q)=
\begin{pmatrix}
1 & 0 & 1\\
0 & 0 & -1\\
0 & 1 & x^{-1}+x-q^{m+1}\\
\end{pmatrix} \,,
\ee
has $\det(\bJ_m(x,q))=x^{-1}-x$ and satisfies the analytic extension 
\begin{equation}
\label{Jmxqq}  
  \bJ_m(x,q^{-1}) =
  \begin{pmatrix}
    1&0&0\\
    0&0&1\\
    0&1&0
  \end{pmatrix} \bJ_{-m-1}(x,q)
  \begin{pmatrix}
    1&0&0\\
    0&0&1\\
    0&1&0
  \end{pmatrix}  \,.
\end{equation}
\end{theorem}
\begin{proof}
The proof is analogous to the proof of Theorem~\ref{thm.41bJ1}.
Equation~\eqref{41Jxqrec} follows quickly using the $q$-hypergeometric expressions and
noting that $C_{m}(x,q)$ has a boundary term so satisfies an inhomogenous version.
The block form again reduces the calculation of the determinant of $\bJ_{m}(x,q)$ to a
calculation of the determinant of $\bJ^\RED_{m}(x,q)$ given in \cite{GGM:peacock}.
Equation~\eqref{Jmxqq} follows from the symmetry of the $q$-hypergeometric functions
\be
\label{ABCqinv}
\begin{aligned}
C_{m}(x,q^{-1}) &=
  C_{-m}(x,q)
\\
A_{m}(x,q^{-1}) &=
B_{-m}(x,q)
\\
B_{m}(x,q^{-1}) &=
A_{-m}(x,q).
\end{aligned}
\ee
\end{proof}
The Appell-Lerch like sums again appear in the inverse of $\bJ_{m}(x,q)$. The proof
is again completely analogous to the proof of Theorem~\ref{thm.41bJ2}.

\begin{theorem}
\label{thm.41bJx2}  
We have
\be
\label{Jxqinv}
\bJ_{m}(x,q)^{-1}
=
\frac{1}{x^{-1}-x}\begin{pmatrix}
x^{-1}-x  & -LB_{m}(x,q) & LA_{m}(x,q)\\
0 & B_{m+1}(x,q) & -A_{m+1}(x,q)\\
0 & -B_{m}(x,q) & A_{m}(x,q)
\end{pmatrix} 
\ee
for the $q$-series $LA_m(x,q),LB_m(x,q)$ defined by
\be
\label{Lx12}
\begin{aligned}
LA_{m}(x,q) &=
A_{m+1}(x,q)C_{m}(x,q)-A_{m}(x,q)C_{m+1}(x,q)
\\
LB_{m}(x,q) &=
B_{m+1}(x,q)C_{m}(x,q)-B_{m}(x,q)C_{m+1}(x,q)
\end{aligned}
\ee
The $q$-series $LA_m(x,q),LB_m(x,q)$ are expressed in terms of Appell-Lerch type sums:
\be
\label{Lxm}
\begin{aligned} 
LA_{m}(x,q) &=
\sum_{k=0}^{\infty}(-1)^{k}
\frac{q^{k(k+1)/2+km+k}x^{k+m+1}}{(q;q)_{k}(x^{2}q;q)_{k}(1-xq^{k})} \\      
LB_{m}(x,q) &=
LA_{m}(x^{-1},q) \,.
\end{aligned}
\ee
\end{theorem}
\begin{proof}
  Given the block form of $\bJ_{m}(x,q)$ and the determinant calculated previously in
  Theorem~\ref{thm.41bJx1}, Equation~\eqref{Lx12} follows from taking the matrix inverse.
Observe
that again $A(x;q^m,q)$ has first column $(1,0,0)^t$ and
first row $(1,0,1)$. It follows that its inverse matrix has first column $(1,0,0)^t$
and first row $(1,1,0)$. This, together with \eqref{41Jxqrec}, implies that
\begin{equation}
\begin{split}
\bJ_{m+1}(x,q)^{-1}
&= A(x,q^m,q)^{-1}
\bJ_{m}(x,q)^{-1}\\
&=
\begin{pmatrix}
1 & 1 & 0\\
0 & x+x^{-1}-q^{m+1} & 1\\
0 & -1 & 0\\
\end{pmatrix}
\bJ_{m}(x,q)^{-1}
\end{split}
\end{equation}
which implies that $LA_m(x,q),LB_{m}(x,q)$ satisfy the first order inhomogeneous linear
$q$-difference equation
\begin{equation}
\begin{split}
LA_{m-1}(x,q)-LA_{m}(x,q) =
A_{m}(x,q),\\
LB_{m-1}(x,q)-LB_{m}(x,q) =
B_{m}(x,q)\,.
\end{split}
\end{equation}
Let $\calL A_m(x,q)$ denote the right-hand side of the first Equation~\eqref{Lxm}.  
Then we have
\[
\calL A_{m-1}(x,q)-\calL A_{m}(x,q)
=
\sum_{k=0}^{\infty}(-1)^{k}
\frac{q^{k(k+1)/2+km}x^{k+m}(1-xq^{k})}{(q;q)_{k}(x^{2}q;q)_{k}(1-xq^{k})}
=
A_{m}(x,q).
\]
Therefore $\calL A_{m}(x,q)-LA_{m}(x,q)$ is independent of $m$. Moreover,
$\lim_{m \to \infty} \calL A_{m}(x,q)-LA_{m}^{(0)}(x,q)=0$ for $|q|,|x|<1$ or
$\lim_{m \to -\infty} \calL A_{m}(x,q)-LA_{m}^{(0)}(x,q)=0$ for $|q|,|x|>1$.
Equations~\eqref{Lxm} follows from analytic continuation.
\end{proof}
Now if we take the inverse of $\bJ_{m}(x,q)^{-1}$ we can get similar identities for
$C_{m}(x,q)$.
\begin{corollary}
\begin{align}
\label{Cm}  
C_{m}(x,q)&=
\frac{1}{x^{-1}-x}\left(A_{m}(x,q)LB_{m}(x,q)-B_{m}(x,q)LA_{m}(x,q)\right)\\
\label{Cmp1}  
C_{m+1}(x,q)&=
\frac{1}{x^{-1}-x}\left(A_{m+1}(x,q)LB_{m}(x,q)-B_{m+1}(x,q)LA_{m}(x,q)\right)\,.
\end{align}
\end{corollary}

\subsection{Borel resummation and Stokes constants}
\label{sub.41borelx}

In this section we extend the discussion in Section~\ref{sub.41borel}
to include $x$-deformation. We analyse the asymptotic expansion as
$q = \re^{2\pi\ri\tau}$ and $\tau\rightarrow 0$ of the $(x,q)$-series
presented in Section~\ref{sub.41qx} and relate them to the
$(x,\tau)$-asymptotic series given in Section~\ref{sub.41Phi0xq}.  For
this purpose, it is more convenient to introduce the decorated
$(x,q)$-series
\begin{equation}
  \begin{aligned}
    \mc{C}_m(x,q) =
    &C_m(x,q),\\
    \mc{A}_m(x,q) =
    &\frac{(qx^2;q)_\infty}{\th(-q^{1/2}x,q)}A_m(x,q),\\
    \mc{B}_m(x,q) =
    &x\frac{(qx^{-2};q)_\infty}{\th(-q^{1/2}x^{-1},q)}B_m(x,q),
  \end{aligned}
\end{equation}
where
\begin{equation}
  \theta(x,q) = (-q^{1/2}x;q)_\infty(-q^{1/2}x^{-1};q)_\infty.
\end{equation}
The associated decorated matrix $\mc{J}(x,q)$ is given by
\begin{align}
  \calJ_{m}(x,q)
  =
  &\begin{pmatrix}
    1 & \mathcal{C}_{m}(x,q) & \mathcal{C}_{m+1}(x,q)\\
    0 & \mathcal{A}_{m}(x,q) & \mathcal{A}_{m+1}(x,q)\\
    0 & \mathcal{B}_{m}(x,q) & \mathcal{B}_{m+1}(x,q)
  \end{pmatrix}\nn
=
&\begin{pmatrix}
1 & 0 & 0\\
0 & \frac{(qx^{2};q)_{\infty}}{\theta(-q^{1/2}x;q)^{2}} & 0\\
0 & 0 & x\frac{(qx^{-2};q)_{\infty}}{\theta(-q^{1/2}x^{-1};q)^{2}}
\end{pmatrix}
\bJ_{m}(x,q)
\end{align}
and it has
\begin{equation}
  \det\mc{J}(x,q):= \det\mc{J}_m(x,q) = \theta(-q^{-1/2}x^2,q)
  \theta(-q^{1/2}x;q)^{-2}\theta(-q^{1/2}x^{-1},q)^{-2}.
\end{equation}
We will focus on the vector $\mc{B}(x,q)$
of $(x,q)$-series 
\begin{equation}
  \mc{B}(x,q) =
  \begin{pmatrix}
    \mc{C}_0(x,q)\\
    \mc{A}_0(x,q)\\
    \mc{B}_0(x,q)
  \end{pmatrix},
\end{equation}
which is defined for $|q| \neq 1$ and satisfies
by
\begin{equation}
  \mc{B}(x,q^{-1}) =
  \begin{pmatrix}
    1&0&0\\
    0&0&x \det\mc{J}(x,q)^{-1}\\
    0&-x\det\mc{J}(x,q)^{-1}&0
  \end{pmatrix} \mc{B}(x,q).
\end{equation}
We will write
\begin{equation}
\label{xq}
q=e^{2\pi\ri \tau}, \qquad x=e^u
\end{equation}
and we will show that the
asymptotic expansion of $\mc{B}(x,q)$ in the limit $\tau \to 0$ is
related to the vector $\Phi(x,\tau)$ of $(x,\tau)$ asymptotic series
\begin{equation}
    \Phi(x,\tau) =
  \begin{pmatrix}
    \Phi^{(\s_0)}(x,\tau)\\
    \Phi^{(\s_1)}(x,\tau)\\
    \Phi^{(\s_2)}(x,\tau)
  \end{pmatrix}
\end{equation}
with corrections given by $\mc{B}(\tx,\tq)$ where
\be
\label{xqt}
\tq=e^{-2\pi\ri/\tau}, \qquad \tx=e^{u/\tau} \,.
\ee

\begin{figure}[htpb!]
  \centering%
  \subfloat[$\Lambda^{\s_0}(x)$]
  {\includegraphics[height=6cm]{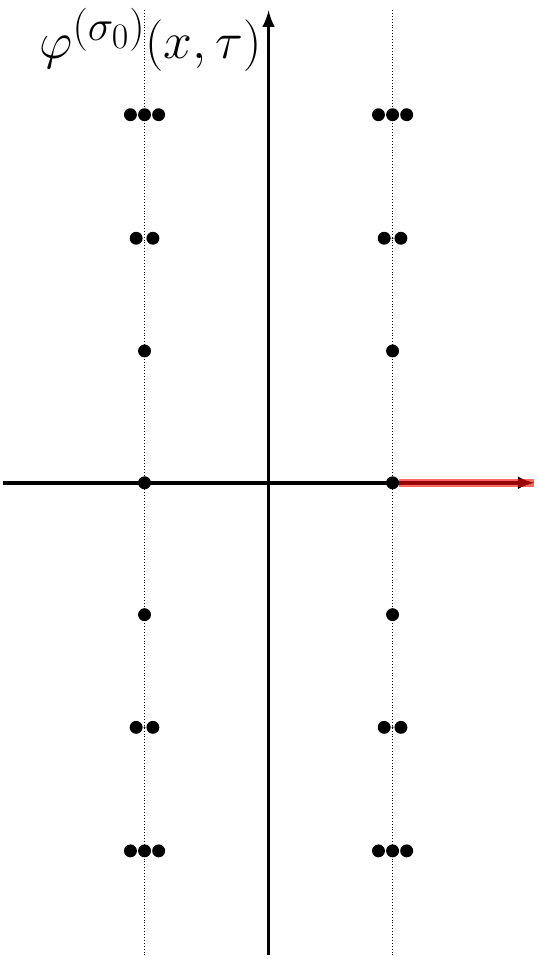}}%
  \hspace{4ex}
  \subfloat[$\Lambda^{\s_0}(x)$]
  {\includegraphics[height=6cm]{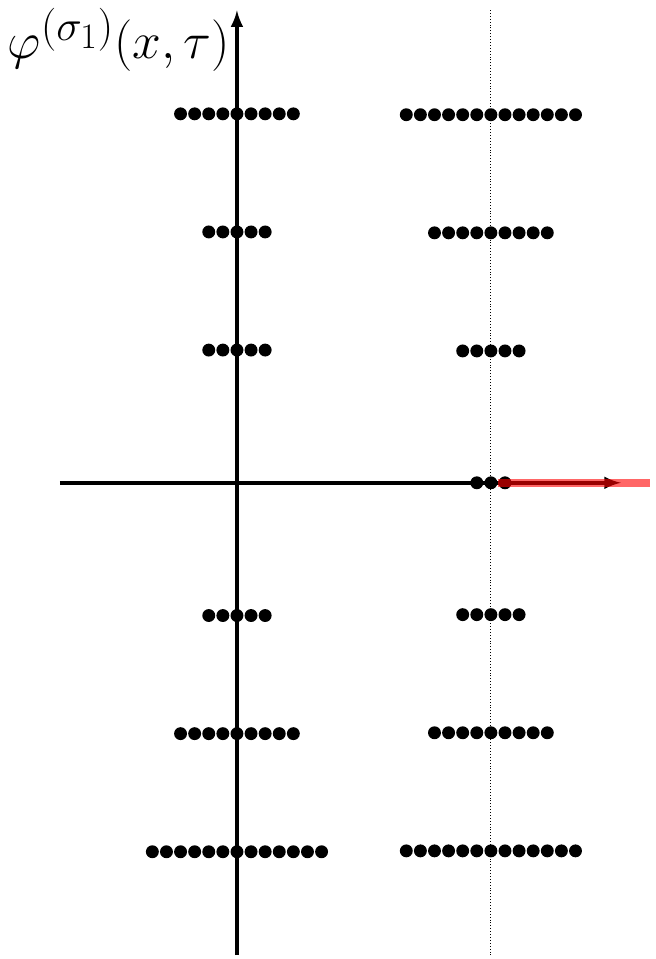}}%
  \hspace{4ex}
  \subfloat[$\Lambda^{\s_0}(x)$]
  {\includegraphics[height=6cm]{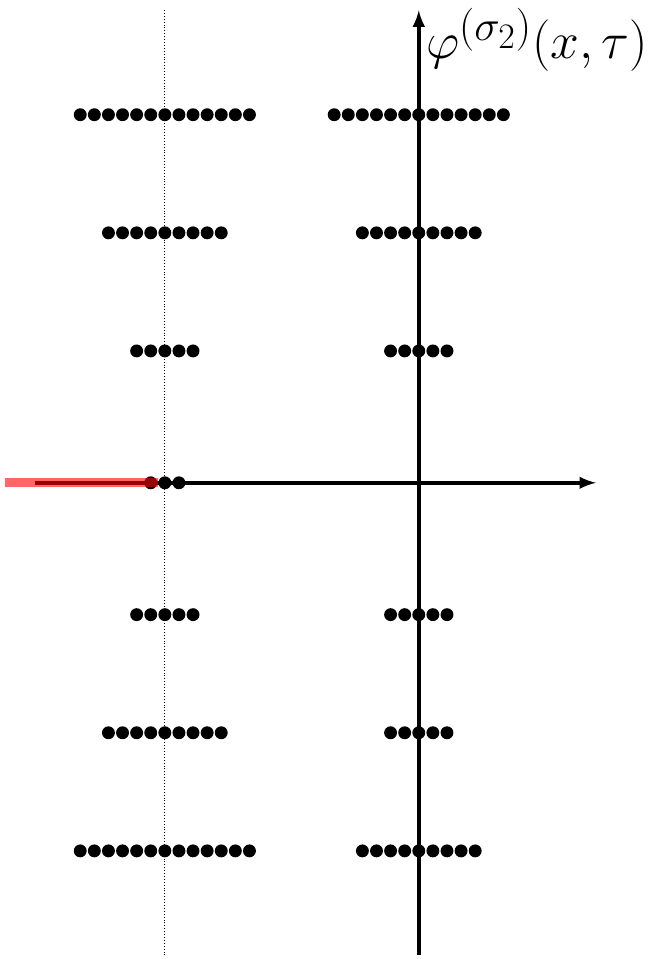}}%
  \caption{Singularities of the Borel transforms of
    $\varphi^{(\s_j)}(x,\tau)$ for $j=0,1,2$ of the knot
    $\knot{4}_1$.}
  \label{fg:41x_brplane}
\end{figure}

The asymptotic series $\Phi(x,\tau)$ can be resummed by Borel
resummation.  As we have explained in Section~\ref{sub.41borel} the
value of the Borel resummation depends on the singularities of the
Borel transform of $\Phi(x,\tau)$.  The positions of these singular
points, denoted collectively as $\Lambda(x)$, are smooth functions of
$x$, and in the limit $x=1$ they are equal to $\Lambda$ defined in
\eqref{eq:Lmb}.  When $x$ is near~1, which is the regime we will be
interested in, each singular point $\iota_{i,j}^{(k)}$ in $\Lambda$
splits to a finite set of points located at
$\iota_{i,j}^{(k,\ell)}:=\iota_{i,j}^{(k)}+\ell\log(x)$, $\ell\in\IZ$,
i.e.~they are aligned on
a line and are apart from each other by a distance $\log(x)$.  We
illustrate this schematically in Fig.~\ref{fg:41x_brplane}. The
complex plane of $\tau$ is divided to infinitely many cones by rays
passing through these singular points, and the Borel resummation of
$\Phi(x,\tau)$, denoted by $s_R(\Phi)(x,\tau)$, is only well-defined
within a cone $R$.

We conjecture that the asymptotic expansion in the limit $q\rightarrow
1$ of the vector of $(x,q)$-series $\mc{B}(x,q)$ can be expressed in
terms of $s_R(\Phi)(x,\tau)$. Furthermore, in each cone, the asymptotic
expansion can be upgraded to exact identities between $\mc{B}(x,q)$
and linear transformation of Borel resummation of $\Phi(x,\tau)$ up to
exponentially small corrections characterised by $\tq$ and $\tx =
\exp(\frac{\log x}{\tau})$.

\begin{conjecture}
  \label{conj-exact-ra-x}
  For every $x \sim 1$, every cone $R \subset \BC\setminus\Lambda(x)$
  and every $\tau \in R$ we have
  \begin{equation}
    \label{eq:GMsPhix}
    \Delta'(x,\tau) \mc{B}(x,q) = M_R(\tx,\tq)
    \Delta(x,\tau)s_R(\Phi)(x,\tau),
  \end{equation}
  where 
  \begin{equation}
    \begin{aligned}
      &\Delta'(x,\tau)
      =\diag(\tau^{1/2}\tfrac{x^{1/2}-x^{-1/2}}{\tx^{1/2}-\tx^{-1/2}},
      (\tx/x)^{1/2}\re^{\frac{3\pi\ri}{4}-\frac{\pi\ri}{4}(\tau+\tau^{-1})},
      (\tx/x)^{1/2}\re^{\frac{3\pi\ri}{4}-\frac{\pi\ri}{4}(\tau+\tau^{-1})}),\\
      &\Delta(x,\tau) = \diag(\tau^{1/2}
      \tfrac{x^{1/2}-x^{-1/2}}{\tx^{1/2}-\tx^{-1/2}},1,1),
      \end{aligned}
  \end{equation}
  and $M_R(\tx,\tq)$ is a $3\times 3$ matrix of $\tq$ (resp., $\tq^{-1}$)-series 
  if $\Im\tau>0$ (resp., $\Im\tau<0$) with coefficients in $\IZ[\tx^{\pm 1}]$
  that depend on $R$.
\end{conjecture}

\begin{figure}[htpb!]
  \centering
  \includegraphics[height=7cm]{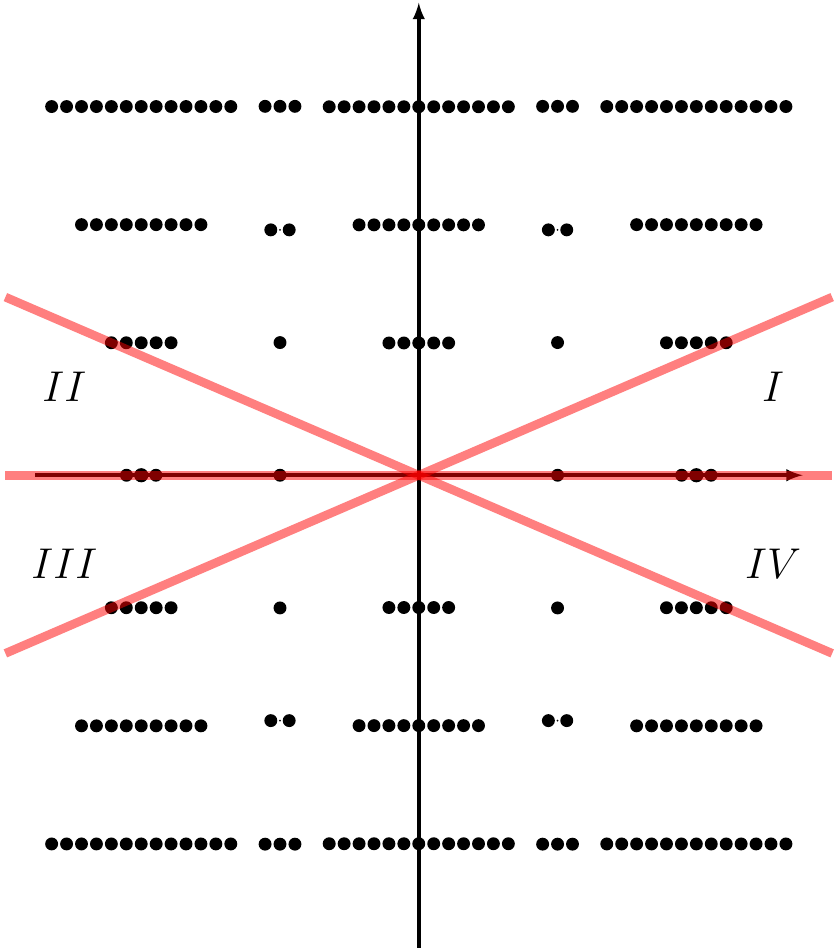}
  \caption{Stokes rays and cones in the $\tau$-plane for the 3-vector
    $\Phi(x,\tau)$ of asymptotic series of the knot $\knot{4}_1$.}
  \label{fg:41x_rays_m}
\end{figure}

To illustrate examples of $M_R(\tx,\tq)$, we pick four of these cones,
located slightly above and below the positive or negative real axis,
labeled in counterclockwise direction by $I,II,III,IV$, cf.~Fig.~\ref{fg:41x_rays_m}.
\begin{conjecture}
  Equation \eqref{eq:GMsPhix} holds in the cones $R=I,II,III,IV$ where
  the matrices $M_R(\tx,\tq)$ are given in terms of
  $\mc{J}_{-1}(\tx,\tq)$ as follows 
  \begin{subequations}
    \begin{align}
      M_I(\tx,\tq) &= \mc{J}_{-1}(\tx,\tq)
      \begin{pmatrix}
        1&0&0\\
        0&0&-1\\
        0&1&-1
      \end{pmatrix}, & |\tq|<1,\\
      M_{II}(\tx,\tq) &=
        \begin{pmatrix}
          1&0&0\\
          0&0&1\\
          0&1&0
        \end{pmatrix}
        \mc{J}_{-1}(\tx,\tq)
      \begin{pmatrix}
        1&0&0\\
        0&1&0\\
        0&1&-1
      \end{pmatrix}, & |\tq|<1,\\
      M_{III}(\tx,\tq) &=
        \begin{pmatrix}
          1&0&0\\
          0&0&-1\\
          0&-1&0
        \end{pmatrix}
        \mc{J}_{-1}(\tx,\tq)
      \begin{pmatrix}
        1&\tx^{-1}-1&0\\
        0&-1&0\\
        0&\tx+\tx^{-1}&1
      \end{pmatrix}, & |\tq|>1,\\
      M_{IV}(\tx,\tq) &= \mc{J}_{-1}(\tx,\tq)
      \begin{pmatrix}
        1&0&\tx^{-1}-1\\
        0&0&-1\\
        0&1&\tx+\tx^{-1}
      \end{pmatrix}, & |\tq|>1.
    \end{align}
  \end{subequations}
\end{conjecture}

\begin{remark}
  It is sometimes stated in the literature that the Gukov--Manolescu
  series is obtained by ``resumming'' the perturbative series
  $\Phi^{(\sigma_0)}(x, \tau) $ associated to the trivial connection, although it
  is not always clear what ``resumming'' means in that context. The above conjecture
  shows that, generically, $\mc{C}_0(x,q)$ involves the Borel resummation of all
  perturbative series $\Phi^{(\sigma_j)}(x, \tau)$, $j=0,1,2$, as well as
  non-perturbative corrections in $\tq, \tx$. 
\end{remark}

We now discuss the Stokes automorphism of the Borel resummation
$s_R(\Phi)(x,\tau)$.  The discussion is similar to the one in 
Section~\ref{sub.41borel}.  To any singular point of the Borel
transform of $\Phi(x,\tau)$ locatd at $\iota_{i,j}^{(k,\ell)}$, we can
associate a local Stokes matrix
\begin{equation}
  \mf{S}_{\iota_{i,j}^{(k,\ell)}} = I +
  \mc{S}_{i,j}^{(k,\ell)}\tq^k\tx^\ell E_{i,j}, \quad \mc{S}_{i,j}^{(k,\ell)}\in\IZ,
\end{equation}
where $E_{i,j}$ is the elementary matrix with $(i,j)$-entry 1
($i,j=0,1,2$) and all other entries zero, and
$\mc{S}_{i,j}^{(k,\ell)}$ is the Stokes constant.  Let us again assume
the locality condition.  Then for any ray of angle $\theta$, the Borel
resummations of $\Phi(x,\tau)$ with $\tau$ whose argument is raised
slightly above $\theta$ ($\theta_+$) or sightly below ($\theta_-$) 
are related by the following formula of Stokes automorphism
\begin{equation}
  \Delta(x,\tau)s_{\theta_+}(\Phi)(x,\tau) = \mf{S}_{\theta}(\tx,\tq)
  \Delta(x,\tau)s_{\theta_-}(\Phi)(x,\tau),\quad
  \mf{S}_\theta(\tx,\tq) = \prod_{\arg\iota =\theta}\mf{S}_\iota(\tx,\tq).
\end{equation}
Because of the locality condition, we don't have to worry about the
order of product of local Stokes matrices.

In addition, given two rays $\rho_{\theta^+}$ and $\rho_{\theta^-}$ whose
arguments satisfy $0<\theta^+-\theta^-\leq \pi$, we define the global
Stokes matrix $\mf{S}_{\theta^-\rightarrow \theta^+}(\tx,\tq)$ by
\begin{equation}
  \Delta(x,\tau)s_{\theta^+}(\Phi)(x,\tau) =
  \mf{S}_{\theta^-\rightarrow \theta^+}(\tx,\tq)
  \Delta(x,\tau)s_{\theta^-}(\Phi)(x,\tau),
\end{equation}
where both sides are analytically continued smoothly to the same value
of $\tau$.  The global Stokes matrix
$\mf{S}_{\theta^-\rightarrow\theta^+}(\tx,\tq)$ satisfies the
factorisation property~\cite{GGM:resurgent,GGM:peacock}
\begin{equation}
  \mf{S}_{\theta^-\rightarrow\theta^+}(\tx,\tq) =
  \prod_{\theta^-<\theta<\theta^+}^{\leftarrow}
  \mf{S}_{\theta}(\tx,\tq),
\end{equation}
where the ordered product is taken over all the local Stokes matrices
whose arguments are sandwiched between $\theta^-,\theta^+$ and they
are ordered with rising arguments from right to left.

Given \eqref{eq:GMsPhix} with explicit values of $M_R(\tx,\tq)$ for
$R=I,II,III,IV$, in general we can calculate the global Stokes matrix
via
\begin{equation}
  \mf{S}_{R\rightarrow R'}(\tx,\tq) = M_{R'}(\tx,\tq)^{-1} \cdot M_{R}(\tx,\tq).
  \label{eq:SMMx}
\end{equation}
For instance, we find the global Stokes matrix from cone $I$
anti-clockwise to cone $II$ is 
\begin{equation}
  \mf{S}_{I\rightarrow II}(\tx,\tq) =
  \begin{pmatrix}
    1&0&0\\
    0&1&0\\
    0&1&-1
  \end{pmatrix}
  \mc{J}_{-1}(\tx,\tq)^{-1}
  \begin{pmatrix}
    1&0&0\\
    0&0&1\\
    0&1&0
  \end{pmatrix}
  \mc{J}_{-1}(\tx,\tq)
  \begin{pmatrix}
    1&0&0\\
    0&0&-1\\
    0&1&-1
  \end{pmatrix},\quad |\tq|<1.\label{eq:S12x}
\end{equation}
This Stokes matrix has the block upper triangular form
\begin{equation}
    \begin{pmatrix}
    1&*&*\\
    0&*&*\\
    0&*&*
  \end{pmatrix}.\label{eq:1bdx}
\end{equation}
One can verify that the $2\times 2$
submatrix of $\mf{S}_{I\rightarrow II}(\tx,\tq)$ in the bottom right
is the Stokes matrix in~\cite{GGM:resurgent}. In addition we can also
extract Stokes constants $\mc{S}_{0,j}^{(k,\ell)}$
($j=1,2, k=1,2,\ldots$) responsible for Stokes automorphisms into
$\Phi^{(\s_0)}(x,\tau)$ from Borel singularities in the upper half
plane, and collect them in the generating series
\begin{equation}
  \ms{S}^+_{0,j}(\tx,\tq) = \sum_{k=1}^\infty\sum_{\ell}
  \mc{S}_{0,j}^{(k,\ell)}\tx^\ell\tq^k,\quad j=1,2.
\end{equation}
We find 
\begin{align}
  \ms{S}_{0,1}^+(\tx,\tq) = \ms{S}_{0,2}^+(\tx,\tq) =
  &\tx^{-1}\left(-\mc{C}_{-1}(\tx,\tq)+\mc{C}_0(\tx,\tq)
    \frac{\mc{A}_{-1}(\tx,\tq)+\mc{B}_{-1}(\tx,\tq)}
    {\mc{A}_0(\tx,\tq)+\mc{B}_0(\tx,\tq)}
    \right)
    \nn=
  &-\tq-(\tx+\tx^{-1})\tq^2
    -(\tx^2+1+\tx^{-2})\tq^3+\ldots.
\end{align}

Similarly, we find the global Stokes matrix from cone $III$
anti-clockwise to cone $IV$ is 
\begin{equation}
  \mf{S}_{III\rightarrow IV}(\tx,\tq) =
  \begin{pmatrix}
    1&0&0\\
    0&-1&1\\
    0&1&0
  \end{pmatrix}\cdot \mc{J}_{-1}(\tx,\tq^{-1})^{-1}\cdot
  \begin{pmatrix}
    1&0&0\\
    0&0&1\\
    0&1&0
  \end{pmatrix}
  \cdot \mc{J}_{-1}(\tx,\tq^{-1})\cdot
  \begin{pmatrix}
    1&0&0\\
    0&1&0\\
    0&1&1
  \end{pmatrix},\quad |\tq|>1.\label{eq:S34x}
\end{equation}
It also has the form as \eqref{eq:1bdx}. This, together with the same phenomenon in
the upper half plane, implies that $\Phi^{(s_j)}(x,\tau)$ ($j=1,2$) form a minimal
resurgent structure. The $2\times 2$ submatrix of
$\mf{S}_{III\rightarrow IV}(\tx,\tq)$ in the bottom right is identical
to the Stokes matrix given in~\cite{GGM:resurgent}.  We also extract
Stokes constants $\mc{S}_{0,j}^{(k,\ell)}$ ($j=1,2$, $k=-1,-2,\ldots$)
responsible for Stokes automorphisms into $\Phi^{(\s_0)}(x,\tau)$ from
Borel singularities in the lower half plane, and collect them in the
generating series
\begin{equation}
  \ms{S}^-_{0,j}(\tx,\tq) = \sum_{k=-1}^{-\infty}\sum_{\ell}
  \mc{S}_{0,j}^{(k,\ell)}\tx^\ell\tq^k,\quad j=1,2.
\end{equation}
And we find
\begin{align}
  \ms{S}_{0,2}^-(\tx,\tq) = -\ms{S}_{0,1}^-(\tx,\tq) = \ms{S}_{0,1}^+(\tx,\tq^{-1}).
\end{align}

We can also use \eqref{eq:SMMx} to compute the global Stokes matrix
$\mf{S}_{IV\rightarrow I}(\tq)$ and we find 
\begin{equation}
  \mf{S}_{IV\rightarrow I} = 
  \begin{pmatrix}
    1&0&1\\
    0&1&\tx+1+\tx^{-1}\\
    0&0&1
  \end{pmatrix}.
\end{equation}
Note that this can be identified as $\mf{S}_{0}$, associated to the ray
$\rho_0$, and it can be factorised as 
\begin{equation}
  \mf{S}_{0} = \mf{S}_{\iota_{0,2}}\mf{S}_{\iota_{1,2}},\quad
  \mf{S}_{\iota_{0,2}} =
  \begin{pmatrix}
    1&0&1\\
    0&1&0\\
    0&0&1
  \end{pmatrix},\quad
  \mf{S}_{\iota_{1,2}} =
  \begin{pmatrix}
    1&0&0\\
    0&1&\tx+1+\tx^{-1}\\
    0&0&1
  \end{pmatrix}.
\end{equation}
Since the local Stokes matrices $\mf{S}_{\iota_{0,2}}$ and
$\mf{S}_{\iota_{1,2}}$ commute, the locality condition is satisfied.
We read off the Stoke discontinuity formulas 
\begin{equation}
  \label{eq:dsPhi01x}
\begin{aligned}
  &\text{disc}_0 \Phi^{(0)}(x,\tau) =
  \frac{\tx^{1/2}-\tx^{-1/2}}{x^{1/2}-x^{-1/2}}
  \tau^{-1/2}s(\Phi^{(s_2)})(x,\tau), \\
  &\text{disc}_0 \Phi^{(1)}(x,\tau) =
  (\tx+1+\tx^{-1})s(\Phi^{(s_2)})(x,\tau) \,.
\end{aligned}
\end{equation}
They reduce properly to \eqref{disc-pra} in the $x\rightarrow 1$
limit, and the second identity has already appeared in
\cite{GGM:resurgent}.

Finally, in order to compute the global Stokes matrix
$\mf{S}_{II\rightarrow III}(\tq)$, we need to take into account that
the odd powers of $\tau^{1/2}$ on both sides of \eqref{eq:GMsPhix}
give rise to additional $-1$ factors when one crosses the branch cut at
the negative real axis, and \eqref{eq:SMMx} should be modified by
\begin{equation}
  \mf{S}_{II\rightarrow III}(\tq) = \diag(-1,1,1)
  M_{III}(\tq)^{-1}\diag(-1,1,1) M_{II}(\tq),
\end{equation}
and we find 
\begin{equation}
  \mf{S}_{II\rightarrow III} = 
  \begin{pmatrix}
    1&1&0\\
    0&1&0\\
    0&-\tx-1-\tx^{-1}&1
  \end{pmatrix}.
\end{equation}
Similarly this can be identified as $\mf{S}_{\pi}$ associated to the
ray $\rho_{\pi}$ and it can be factorised as 
\begin{equation}
  \mf{S}_{\pi} = \mf{S}_{\iota_{0,1}}\mf{S}_{\iota_{2,1}},\quad
  \mf{S}_{\iota_{0,1}} =
  \begin{pmatrix}
    1&1&0\\
    0&1&0\\
    0&0&1
  \end{pmatrix},\quad
  \mf{S}_{\iota_{2,1}} =
  \begin{pmatrix}
    1&0&0\\
    0&1&0\\
    0&-\tx-1-\tx^{-1}&1
  \end{pmatrix}.
\end{equation}
Note that the local Stokes matrices $\mf{S}_{\iota_{0,1}}$ and
$\mf{S}_{\iota_{2,1}}$ also commute.  We read off the Stokes
discontinuity formulas 
\begin{align}
  &\text{disc}_\pi \Phi^{(0)}(x,\tau) =
    \frac{1-\tx}{1-x}\tau^{-1/2}s(\Phi^{(s_1)})(x,\tau), \\
  &\text{disc}_\pi \Phi^{(2)}(x,\tau) = -(\tx+1+\tx^{-1})s(\Phi^{(s_1)})(x,\tau).
\end{align}
They reduce properly to \eqref{disc-prb} in the $x\rightarrow 1$
limit, and the second identity has already appeared in
\cite{GGM:resurgent}.

\subsection{$(u,\tau)$ state-integrals}
\label{sub.41statex}

In parallel to the discussion in Sections~\ref{sub.AK41new} and \ref{sub.3x3state}, we
now introduce a new state-integral which depends on $\tau$, but also on a variable $u$. 
Let us consider the state-integral 
\begin{equation}
  \mc{Z}_\calB(u,\tau) = -\frac{\ri}{2\bb}\frac{\sinh(\pi \bb^{-1} u)}
  {\sinh(\pi \bb u)}
  \int_{\calB} \tanh(\pi\bb^{-1}v)
  \frac{\Phi_\bb(-v+\frac{\ri}{2}\bb^{-1}+u)}
  {\Phi_\bb(v-\frac{\ri}{2}\bb^{-1}+u)}
  \re^{2\pi\ri u(v-\frac{\ri}{2}\bb^{-1})} \rd v,
\end{equation}
where the contour of integral $\calB$ is not specified yet. The integrand
reduces to that of \eqref{car} in the limit $u\rightarrow 0$.  For
generic $\bb^2\in\IC'$ so that $\real\bb>0$, the integrand has the following poles
and zeros
\begin{equation}
  \begin{aligned}
    &\text{Poles}:\; \left\{\pm \ri\bb\left(\frac{1}{2}+m\right),\;
      \pm u -\ri\bb\left(\frac{1}{2}+m\right)-\ri\bb^{-1}n\;\big|\;
      m,n=0,1,2,\ldots
    \right\}\\
    &\text{Zeros}:\; \left\{
      \pm u +\ri\bb\left(\frac{1}{2}+m \right)+\ri\bb^{-1}(1+n)\;\big|\;
      m,n=0,1,2,\ldots
    \right\} \,.
  \end{aligned}
  \label{eq:PZ41x}
\end{equation}
We can choose for the integral the contour $\mc{A}_N$ in the upper half
plane that wraps the following poles, as in the left panel of
Fig.~\ref{contours-fig}, 
\begin{equation}
\label{vm-poles}
  v_m = \ri\bb\left(\frac{1}{2}+m\right),\quad m=0,1,2,\ldots,N-1.
\end{equation}
By summing over the residues of these poles, the integral evaluates
as follows
\begin{equation}
  \mc{Z}_{\mc{A}_N}(\ub,\tau) = \sum_{n=0}^{N-1}(-1)^n
  q^{-n(n+1)/2}(qx;q)_n(qx^{-1};q)_n,\qquad
  x= \re^{u},\, q=\re^{2\pi\ri\tau} \,,
\end{equation}
where we defined $\ub=u/(2\pi \bb)$, as in~\cite[Eqn.(2)]{GGM:peacock}.
When $x = q^N$ this is none other
than the colored Jones polynomial of the knot $\knot{4}_1$
\begin{equation}
\label{aint-habiro}
  \mc{Z}_{\mc{A}_N}(\ri N\bb,\bb^2) = J_N^{\knot{4}_1}(q) =
  \sum_{n=0}^{N-1}(-1)^n
  q^{-n(n+1)/2}(q^{1+N};q)_n(q^{1-N};q)_n.
\end{equation}
Alternatively, we can choose for the integral the contour $\mc{C}$ as
in the right panel of Fig.~\ref{contours-fig}, which is asymptotic to
a horizontal line slightly below $\imag(v) = \real(\bb^{-1})$, but
deformed near the origin in such a way that all the poles
\begin{equation}
  v^{\pm}_{m,n} = \pm u -\ri\bb\left(\frac{1}{2}+m\right)-
  \ri\bb^{-1}n,\quad m,n=0,1,2,\ldots
  \label{eq:vmn}
\end{equation}
are below the contour $\mc{C}$.  Let
$ \mc{Z}(u,\tau):= \mc{Z}_{\mc{C}}(u,\tau)$ denote the
corresponding state-integral. Similar to the discussion in
Section~\ref{sub.AK41new}, as the integrand has non-trivial
contributions from infinity in the upper half plane, the two integrals
$\mc{Z}_{\mc{A}_N}(u,\tau)$ and $\mc{Z}(u,\tau)$ are different.  On
the other hand, since the integrand does have vanishing contributions
from infinity in the lower half plane, we can smoothly deform the
contour $\mc{C}$ downwards so that $\mc{Z}(u,\tau)$ can be evaluated
by summing over residues at the poles $v^{\pm}_{m,n}$, and we find
\begin{align}
  \mc{Z}(u,\tau)=
  &\mc{C}_0(x,q)
    +\frac{\tx^{-1}\re^{\frac{3\pi\ri}{4}-\frac{\pi\ri}{4}(\tau+\tau^{-1})}}
    {\tau^{1/2}}
    \mc{A}_0(x,q)\left(L\mc{A}_{0}(\tx,\tq^{-1})
    +\frac{1-\tx}{2}\mc{A}_{0}(\tx,\tq^{-1})\right)\nn
  &\phantom{==}
    +\frac{\tx^{-1}\re^{\frac{3\pi\ri}{4}-\frac{\pi\ri}{4}(\tau+\tau^{-1})}}
    {\tau^{1/2}}
    \mc{B}_0(x,q)\left(L\mc{B}_{0}(\tx,\tq^{-1})
    +\frac{1-\tx}{2}\mc{B}_{0}(\tx,\tq^{-1})
    \right),
    \label{eq:Ixq}
\end{align}
where $L\mc{A}_{\mu}(x,q),L\mc{B}_{\mu}(x,q)$ are defined as in
\eqref{Lx12} with Roman letters $A,B,C$ replaced by caligraphic
letters $\mc{A},\mc{B},\mc{C}$. As mentioned above, the change of integration
contour implements the Habiro inversion of \cite{park-inverted}: the integration
over $\mc{A}_N$ gives the Habiro series (\ref{aint-habiro}), while the integration
over $\mc{C}$ involves $\mc{C}_0(x,q)$, which was interpreted
in~\cite{park-inverted} as an inverted Habiro series. This contribution comes
from the poles $-v_m$ in the lower half-plane. 

The integral $\mc{Z}(u,\tau)$ can also be identified with the Borel
resummation of the perturbative series $\Phi^{(\s_j)}(x;\tau)$ for
$j=0,1,2$.  By inverting the matrix $M_R(\tx,\tq)$ in
\eqref{eq:GMsPhix}, we can also express the Borel resummation
$s_R(\Phi)(x,\tau)$ in any cone $R$ in terms of combinations of
$(x,q)$- and $(\tx,\tq)$-series, and they can be then compared with
the right hand side of \eqref{eq:Ixq}.  For instance, in the cones $I$
and $IV$ respectively, we find 
\begin{subequations}
\begin{align}
  \mc{Z}(u,\tau) =
  &s_I(\Phi^{(\s_0)})(x;\tau) -
    \frac{\tx^{1/2}-\tx^{-1/2}}
    {2(x^{1/2}-x^{-1/2})}\tau^{-1/2}s_I(\Phi^{(\s_2)})(x;\tau), \label{eq:Isphi-1}
  \\ =
  &s_{IV}(\Phi^{(\s_0)})(x;\tau) +
  \frac{\tx^{1/2}-\tx^{-1/2}}{2(x^{1/2}-x^{-1/2})}\tau^{-1/2}s_{IV}(\Phi^{(\s_2)})(x;\tau).
  \label{eq:Isphi-2}
\end{align}
\end{subequations}
This also implies that for positive real $\tau$,
\begin{equation}
  \label{eq:Isphi-3}
  \mc{Z}(u,\tau) =s_{\text{med}}(\Phi^{(\s_0)})(x;\tau) \,.
\end{equation}

Finally, we can introduce the descendants of the integral $\mc{Z}(u,\tau)$
as follows
\begin{equation}
  \mc{Z}_{\lambda,\mu}(u,\tau) =
  -\frac{\ri}{2\bb}\frac{\sinh(\pi \bb^{-1} u)}
  {\sinh(\pi \bb u)}
  \int_{\mc{C}} \tanh(\pi\bb^{-1}v)
  \frac{\Phi_\bb(-v+\frac{\ri}{2}\bb^{-1}+u)}
  {\Phi_\bb(v-\frac{\ri}{2}\bb^{-1}+u)}
  \re^{2\pi\ri u(v-\frac{\ri}{2}\bb^{-1})-2\pi(\lambda\bb-\mu\bb^{-1})v}
  \rd v.
\end{equation}
The integrand has the same poles and zeros as in \eqref{eq:PZ41x}.  To
ensure convergence, the contour $\mc{C}$ needs slight modification: it
is asymptotic to a horizontal line slightly below
$\imag(v) = \frac{1}{2}\real(\bb^{-1})
-|\real(\lambda\bb-\mu\bb^{-1})|$, and it is deformed near the origin
in such a way that all the poles \eqref{eq:vmn} are below the contour
$\mc{C}$.
Similarly, by smoothly deforming the contour downwards we can evaluate
this integral by summing up residues of all the poles in the lower
half plane, and we find 
\begin{align}
  \mc{Z}_{\lambda,\mu}(u,\tau) =
  &(-1)^\mu q^{\lambda/2}\left(
    \mc{C}_\lambda(x,q)
    +\frac{\tx^{-1}\re^{\frac{3\pi\ri}{4}-\frac{\pi\ri}{4}(\tau+\tau^{-1})}}
    {\tau^{1/2}}
    \mc{A}_\lambda(x,q)\left(L\mc{A}_{-\mu}(\tx,\tq^{-1})
    +\frac{1-\tx}{2}\mc{A}_{-\mu}(\tx,\tq^{-1})\right)\right.\nn
  &\left.\phantom{==}
    +\frac{\tx^{-1}\re^{\frac{3\pi\ri}{4}-\frac{\pi\ri}{4}(\tau+\tau^{-1})}}
    {\tau^{1/2}}
    \mc{B}_\lambda(x,q)\left(L\mc{B}_{-\mu}(\tx,\tq^{-1})
    +\frac{1-\tx}{2}\mc{B}_{-\mu}(\tx,\tq^{-1})\right)
    \right).
\end{align}

\subsection{An analytic extension of the colored Jones polynomial} 
\label{sub.CJ41analytic}

In this section we discuss a Borel resummation formula for the colored Jones
polynomial of the $\knot{4}_1$ knot. The latter is defined by
\begin{equation}
  J^{\knot{4}_1}_N(q) = \sum_{k=0}^{N-1}(-1)^k
  q^{-k(k+1)/2}(q^{1+N};q)_k(q^{1-N};q)_k \,.
\end{equation}
Let $u \sim 0$ be in a small neighborhood of the origin in the complex plane.
It is related to $x = q^N$ and $\tau$ by
\begin{equation}
  x = \re^{u},\quad \tau = \frac{u}{2\pi\ri N}+\frac{1}{N}.
  \label{eq:x-n-tau}
\end{equation}
Then $u$ is near $0$, then $x$ is close to $1$, which is the regime
that we studied in Section~\ref{sub.41borelx}, and $\tau$ is close to
$1/N$.  Note that $N\tau = 1+\frac{u}{2\pi\ri}$ is the analogue of
$n/k$ in~\cite{gukov}, and here we are considering a deformation from
the case of $n/k=1$.

Experimentally, we found that in cones $I$ and $IV$ respectively, we
have 
\begin{subequations}
  \begin{align}
    J^{\knot{4}_1}_N(q) =
    &s_{I}(\Phi^{(\s_0)})(x;\tau)
      +\frac{\tx^{1/2}-\tx^{-1/2}}{x^{1/2}-x^{-1/2}}
      \tau^{-1/2}s_{I}(\Phi^{(\s_1)})(x;\tau)
      \nn
    &-(1+\tx)\frac{\tx^{1/2}-\tx^{-1/2}}{x^{1/2}-x^{-1/2}}\tau^{-1/2}
      s_{I}(\Phi^{(\s_2)})(x;\tau)
      \label{eq:Jn-s1}
    \\=
    &s_{IV}(\Phi^{(\s_0)})(x;\tau)
      +\frac{\tx^{1/2}-\tx^{-1/2}}{x^{1/2}-x^{-1/2}}
      \tau^{-1/2}s_{IV}(\Phi^{(\s_1)})(x;\tau)\nn
    &+(1+\tx^{-1})\frac{\tx^{1/2}-\tx^{-1/2}}{x^{1/2}-x^{-1/2}}\tau^{-1/2}
      s_{IV}(\Phi^{(\s_2)})(x;\tau)
      \label{eq:Jn-s2}
  \end{align}
\end{subequations}
where $\tx=e^{u/\tau}=e^{2\pi\ri N u/(u+2\pi\ri)}$. This, together
with Conjecture~\ref{conj-exact-regions} implies 
\begin{align}
  J^{\knot{4}_1}_N(q) =
  &s_{\text{med}}(\Phi^{(\s_0)})(x;\tau)
    +\frac{\tx^{1/2}-\tx^{-1/2}}{x^{1/2}-x^{-1/2}}
    \tau^{-1/2}s_{\text{med}}(\Phi^{(\s_1)})(x;\tau)\nn
  &-\frac{\tx-\tx^{-1}}{2}\frac{\tx^{1/2}-\tx^{-1/2}}{x^{1/2}-x^{-1/2}}
    \tau^{-1/2}s_{\text{med}}(\Phi^{(\s_2)})(x;\tau) ,
  \label{eq:Jn-smed}
\end{align}
which is Conjecture~\ref{conj.ejones} for the $\knot{4}_1$ knot. 

We now make several consistency checks of the above conjecture. The first is
that equation~\eqref{eq:Jn-smed} is invariant
under complex conjugation which moves $\tau$ from cone $I$ to cone $IV$.
The second is that the conjecture implies the Generalised Volume Conjecture.
Indeed, in the limit
  \begin{equation}
    N\rightarrow \infty, \quad \tau \rightarrow 0, \quad \log(x) =
    2\pi\ri N\tau \text{ finite}
  \end{equation}
  the right hand side of~\eqref{eq:Jn-s1},\eqref{eq:Jn-s2} are dominated by the
  first term.  If we keep only the exponential, this is the
  generalised Volume Conjecture \cite{Murakami2011:ivc,gukov}.
  Recall from \cite{Murakami2011:ivc}, the generalised Volume
  Conjecture reads, for $u$ in a small neighborhood of origin such
  that $u\not\in \pi\ri\IQ$,
  \begin{equation}
    \lim_{N\rightarrow \infty}\frac{\log
      J_N^{K}(\exp((u+2\pi\ri)/N))}{N}
    = \frac{H(y,x)}{u+2\pi\ri},
    \label{eq:gVC}
  \end{equation}
  where $x = \exp(u+2\pi\ri)$ and $
    H(y,x) = \Li_2(1/(xy)) - \Li_2(y/x) + \log(x)\log(y)$, 
  with $y$ a solution to $y+y^{-1} = x+x^{-1}-1$.
  By the identification $u + 2\pi\ri = 2\pi\ri (N\tau) \sim 2\pi\ri$, 
  and since $A(x)$ is identical with $H(y,x)$ (up to $\pm 1$), one can
  check that \eqref{eq:Jn-s1},\eqref{eq:Jn-s2} imply \eqref{eq:gVC}.
  

\section{The $\knot{5}_2$-knot}
\label{sec.52}

\subsection{A $3 \times 3$ matrix of $q$-series}
\label{sub.52q3}

The trace field of the $\knot{5}_2$ knot is the cubic field of discriminant $-23$,
with a distinguished complex embedding $\s_1$ (corresponding to the geometric
representation of $\knot{5}_2$), its complex conjugate $\s_2$ and a real embedding $\s_3$.
The $\knot{5}_2$ knot has three boundary parabolic representations whose associated
asymptotic series $\varphi^{(\s_j)}(h)$ for $j=1,2,3$ correspond to the three
embeddings of the trace field. In~\cite{GGM:resurgent} these asymptotic series
were discussed, and a $3 \times 3$ matrix $\bJ^\RED_m(q)$ of $q$-series was constructed
to describe the resurgence properties of the asymptotic series. The matrix 
$\bJ^\RED_m(q)$ is a fundamental solution to the linear $q$-difference
equation~\cite[Eqn.(23)]{GGM:resurgent}
\begin{equation}
  \label{52qdiff}
  f_m(q)-3f_{m+1}(q)+(3-q^{2+m})f_{m+2}(q) - f_{m+3}(q) = 0
\end{equation}
and it is defined by\footnote{The matrices $\bJ^\RED_m(q)$ are related to the
  Wronskians $W_m(q)$ in \cite{GGM:resurgent,GGM:peacock} by
  \begin{equation}
    \bJ^\RED_m(q) =
    \begin{pmatrix}
      0&0&1\\0&1&0\\1&0&0
    \end{pmatrix} W_m(q)^T.
  \end{equation}}
\be
\label{Jqred52}
\bJ^\RED_{m}(q) =
\begin{pmatrix}
H_{m}^{(2)}(q) & H_{m+1}^{(2)}(q) & H_{m+2}^{(2)}(q) \\
H_{m}^{(1)}(q) & H_{m+1}^{(0)}(q) & H_{m+2}^{(1)}(q) \\
H_{m}^{(0)}(q) & H_{m+1}^{(0)}(q) & H_{m+2}^{(0)}(q) 
\end{pmatrix},
\qquad (|q| \neq 1)
\ee
where for $|q|<1$

\be
\label{Hpm}
\begin{aligned}
H^{(0)}_{m}(q)
  &=\sum_{n=0}^\infty \frac{q^{n(n+1)+nm}}{(q;q)_n^3}\,,
\\
  H^{(1)}_{m}(q)
  &=\sum_{n=0}^\infty \frac{q^{n(n+1)+nm}}{(q;q)_n^3}
    \left(1+2n+m-3E_1^{(n)}(q)\right) \,,
  \\
  H^{(2)}_{m}(q)
  &=\sum_{n=0}^\infty \frac{q^{n(n+1)+nm}}{(q;q)_n^3}
    \left((1+2n+m-3E_1^{(n)}(q))^2-3E_2^{(n)}(q)-\frac{1}{6}E_2(q)\right) \,,
\end{aligned}
\ee
and
\be
\label{Hnm}
\begin{aligned}
  H^{(0)}_{-m}(q^{-1})
  &=\sum_{n=0}^\infty (-1)^n\frac{q^{\frac{1}{2}n(n+1)+nm}}{(q;q)_n^3}\,,
\\
  H^{(1)}_{-m}(q^{-1})
  &=-\sum_{n=0}^\infty (-1)^n\frac{q^{\frac{1}{2}n(n+1)+nm}}{(q;q)_n^3}
    \left(\frac{1}{2}+n+m-3E_1^{(n)}(q)\right) \,,
  \\
  H^{(2)}_{-m}(q^{-1})
  &=\sum_{n=0}^\infty (-1)^n\frac{q^{\frac{1}{2}n(n+1)+nm}}{(q;q)_n^3}
  \left(\big(\frac{1}{2}+n+m-3E_1^{(n)}(q)\big)^2-3E_2^{(n)}(q)
    -\frac{1}{12}E_2(q)\right) \,.
\end{aligned}
\ee

\subsection{The Habiro polynomials and the descendant Kashaev invariants}
\label{sub.52hab}

The addition of the asymptotic series $\varphi^{(\s_0)}(h)$ corresponding to the
trivial flat connection
requires a $4 \times 4$ extension of the matrix $\bJ^\RED(q)$. This is consistent
with the fact that the colored Jones polynomial of $\knot{5}_2$ satisfies a third order
inhomogenous linear $q$-difference equation, and hence a $4$th order homogeneous
linear $q$-difference equation. However, the descendant colored Jones polynomials
of $\knot{5}_2$ satisfy a $5$th order inhomogeneous recursion~\cite[Eqn.(14)]{GK:desc},
hence a $6$th order homogeneous recursion. In view of this, we will give a
$6 \times 6$ matrix $\bJ(q)$ of $q$-series and we will use its $4 \times 4$ block
to describe the resurgent structure of the asymptotic series $\varphi^{(\s_0)}(h)$.

Let us recall the Habiro polynomials, the descendant colored Jones
polynomials, the descendant Kashaev invariants and their recursions.   
The Habiro polynomials $H_n^{\knot{5}_2}(q) \in \BZ[q^{\pm 1}]$ are given by terminating
$q$-hypergeometric sums 
\be
\label{Hk52}
H_n^{\knot{5}_2}(q) = (-1)^n q^{\frac{1}{2}n(n+3)}
\sum_{k=0}^n q^{k(k+1)} \binom{n}{k}_q
\ee
(see Habiro~\cite{Ha} and also Masbaum~\cite{Mas:HabPol}) 
where $\binom{a}{b}_q = (q;q)_a/((q;q)_b (q;q)_{b-a})$ is the $q$-binomial function.
In~\cite{GS:Cpoly}, it was shown that $H_n=H_n^{\knot{5}_2}(q)$ satisfies the linear
$q$-difference equation
\be
\label{recHk52}
H_{n+2}^{\knot{5}_2}(q) + q^{3 + n} (1 + q - q^{2 + n}
+ q^{4 + 2 n}) H_{n+1}^{\knot{5}_2}(q)
-q^{6 + 2 n} (-1 + q^{1 + n}) H_n^{\knot{5}_2}(q) =0, \qquad (n \geq 0)
\ee with initial conditions $H_n^{\knot{5}_2}(q)=0$ for $n<0$ and
$H_0^{\knot{5}_2}(q)=1$.  Actually, the above recursion is valid for all
integers if we replace the right hand side of it by $\delta_{n+2,0}$.
The recursion for the Habiro polynomials of $\knot{5}_2$, together with
Equation~\eqref{DJdef} and~\cite{Koutschan:holofunctions}, gives that
$\DJ^{(m)}=\DJ^{\knot{5}_2,(m)}(x,q)$, which is the descendant
colored Jones polynomial defined by~\eqref{DJdef}, satisfies the
linear $q$-difference equation
{\tiny
\begin{multline}
\label{recJm52x}
  (-1 + q^{1 + m}) (-1 + q^{2 + m}) x^2 \DJ^{(m)} - 
  q^{2 + m} (-1 + q^{2 + m}) x (1 + q + x + (1+ q) x^2) \DJ^{(1 + m)} \\
  + q^{3 + m} (q^{3 + m} +(- 1 + q^{2 + m} + q^{3 + m}) x +(- 2  - 
    q + q^{2 + m} + 2 q^{3 + m} + q^{4 + m}) x^2 +(-1 + 
    q^{2 + m} + q^{3 + m}) x^3 + q^{3 + m} x^4) \DJ^{(2 + m)} \\
    - 
 q^{4 + m} (q^{3 + m} +(-1 + q^{3 + m} + q^{4 + m}) x +(-1 + 
    q^{2 + m}  + 2 q^{3 + m}  + q^{4 + m} ) x^2 +(-1 + 
    q^{3 + m} + q^{4 + m}) x^3 + q^{3 + m} x^4) \DJ^{(3 + m)} \\
    + q^{5 + m} x (q^{3 + m} + q^{4 + m} +(-1 + q^{4 + m}) x + 
    (q^{3 + m} + q^{4 + m}) x^2) \DJ^{(4 + m)} - 
    q^{10 + 2 m} x^2 \DJ^{(5 + m)}
\\ = x (q^{2 + m} + q^{4 + m} + (1 - q^{1 + m} - 2 q^{3 + m} - 
    q^{5 + m}) x + (q^{2 + m} + q^{4 + m}) x^2) H_0(q) + 
 q^m x (1 - x q^{-1}) (1 - q x) H_1(q) \,.
\end{multline}
}

\noindent
Using the values $H^{\knot{5}_2}_0(q)=1$, $H^{\knot{5}_2}_1(q)=-q^2-q^4$, it follows that the
right hand side of the above recursion is $x^2$ for all $m$. Setting $x=1$, and
renaming $\DJ^{(m)}$ by $f_m(q)$, we arrive at the inhomogenous $5$-th order
$q$-difference equation satisfied by the descendant Kashaev
invariant~\cite[Eqn.(14)]{GK:desc}

{\tiny
\begin{multline}
\label{52rec}
  -q^{2m+10}f_{m+5}(q)
  + (3q^{2m+9} + 2q^{2m+8} - q^{m+5})f_{m+4}(q)
  + (-3q^{2m+8} - 6q^{2m+7} - q^{2m+6} + 3q^{m+4})f_{m+3}(q)\\
  + (q^{2m+7} + 6q^{2m+6} + 3q^{2m+5} - q^{m+4} - 4q^{m+3})f_{m+2}(q) 
  + (2q^{m+3} + 3q^{m+2})(1-q^{m+2})f_{m+1}(q)
  + (1-q^{m+1})(1-q^{m+2})f_{m}(q)
  =1
\end{multline}
}

\noindent
valid for all integers $m$. Our aim is to define an explicit fundamental matrix
solution to the corresponding sixth order homogenous linear $q$-difference
equation~\eqref{52rec}. To do so, we define a 2-parameter family of deformations
of the Habiro polynomials which satisfy a one-parameter deformation of the recursion
of the Habiro polynomials. Motivated by the $q$-hypergeometric expression~\eqref{Hk52}
for the Habiro polynomials, we define deformations of the Habiro polynomials,
for $|q|\neq 1$, with appropriate normalisations
\be\label{defHab51}
\begin{aligned}
  H_{n}(\ve,\delta;q)
  &=
  \frac{(qe^{\ve-\delta};q)_{\infty}(qe^{\delta};q)_{\infty}}{
    (qe^{\ve};q)_{\infty}(q;q)_{\infty}}
  \frac{(-1)^nq^{n(n+3)/2}e^{(n+1)\ve}}{
    e^{\frac{1}{12}\ve^2-\frac{1}{12}(\ve\delta-\delta^2)E_{2}(q)}}
  \sum_{k\in\BZ}\frac{q^{k(k+1)}e^{(2k+1)\delta}(qe^{\ve};q)_{n}}{
    (qe^{\delta};q)_{k}(qe^{\ve-\delta};q)_{n-k}}\\
  H_{n}(\ve,\delta;q^{-1})
  &=
  \frac{(qe^{\ve+\delta};q)_{\infty}}{(qe^{\delta};q)_{\infty}^{2}}
  \frac{q^{-n(n+3)/2}e^{(n+3/2)\ve}}{(-1)^{n}(e^{-\delta};q)_{\infty}(q;q)_{\infty}}
  \sum_{k\in\BZ}(-1)^{k}q^{k(k+1)/2}e^{\delta k}\frac{(qe^{\delta};q)_{k-1}}{
    (qe^{\ve+\delta};q)_{k-n-1}} 
\end{aligned}
\ee
where $n\in\BZ$ and $|q|<1$. These deformations satisfy the recursion
\begin{small}
\be
  H_{n+2}(\ve,\delta;q)
  +e^{\ve}q^{n+3}(1+q-e^{\ve}q^{n+2}+e^{2\ve}q^{2n+4})
  H_{n+1}(\ve,\delta;q)
  +e^{2\ve}q^{2n+6}(1-e^{\ve}q^{n+1})H_{n}(\ve,\delta;q)
  =0
\ee
\end{small}

\noindent
obtained from~\eqref{recHk52} by replacing $q^n$ to
$e^\ve q^n$.  Note that when $\ve=0$, we cannot solve for $H_{-1}$ in
terms of $H_{n}$ for $n\geq0$ as discussed
in~\cite{park-inverted}\footnote{Our $H_{-1}(q)$ agrees with the one
  defined in~\cite{park-inverted} when $|q|<1$, however differs when
  $|q|>1$.}. It follows that the function \be
\label{DS52}
\begin{aligned}
  \DS_{m}(\ve,\delta;q)
  &=
  -e^{-\ve}(1-e^{\ve})^{2}\sum_{n=-\infty}^{-1}q^{mn}e^{m\ve}
  H_{n}(\ve,\delta;q)(qe^{\ve};q)_{n}(q^{-1}e^{-\ve};q^{-1})_{n}\\
  &=
  \sum_{n=0}^{\infty}\frac{q^{-mn-m}e^{m\ve}
    H_{-1-n}(\ve,\delta;q)}{(q^{-1}e^{\ve};q^{-1})_{n}(qe^{-\ve};q)_{n}}
\end{aligned}
\ee
is an inhomogenous solution of Equation~\eqref{52rec}. 
In particular, for $|q|<1$ we have
\be
\begin{aligned}
  \DS_{m}(\ve,\delta;q)
  =&\:
  \frac{(qe^{\ve-\delta};q)_{\infty}(qe^{\delta};q)_{\infty}
    (1-e^{\ve-\delta})}{(qe^{\ve};q)_{\infty}(q;q)_{\infty}e^{\frac{1}{12}
      \ve^2-\frac{1}{12}(\ve\delta-\delta^2)E_{2}(q)}(1-e^{\ve})}\\
  &\times
  \sum_{n=0}^{\infty}\sum_{k\in\BZ}\frac{(-1)^nq^{(n+1)(n-2)/2-mn-m+k(k+1)}e^{(m-n)
      \ve+(2k+1)\delta}(q^{-1}e^{\ve-\delta};q^{-1})_{n+k}}{
    (q^{-1}e^{\ve};q^{-1})_{n}^{2}(qe^{\ve};q)_{n}(qe^{\delta};q)_{k}}\\
  \DS_{m}(\ve,\delta;q^{-1})
  =&\:
  \frac{(qe^{\ve+\delta};q)_{\infty}}{(qe^{\delta};q)_{\infty}^{2}(
    e^{-\delta};q)_{\infty}(q;q)_{\infty}}\\
  &\times\sum_{n=0}^{\infty}\sum_{k\in\BZ}(-1)^{n+k}
  \frac{q^{-(n+1)(n-2)/2+mn+m+k(k+1)/2}e^{(m-n+1/2)\ve+\delta k}
    (qe^{\delta};q)_{k-1}}{(qe^{\ve+\delta};q)_{k+n}
    (qe^{\ve};q)_{n}(q^{-1}e^{-\ve};q^{-1})_{n}}.
\end{aligned}
\ee
We see that $\DS_{m}(\ve,\delta;q)$ is convergent for $|q|<1$ and all $m\in\BZ$
and for
$|q|>1$ and all $m\in\BZ_{\geq0}$. Moreover,
$\ve\DS_{m}(\ve,\delta;q)\in\BZ((q))[[\ve,\delta]]$ for $m\in\BZ$ and
$\delta^{2}\DS_{m}(\ve,\delta;q^{-1})\in\BZ((q))[[\ve,\delta]]$ for
$m\in\BZ_{\geq0}$. Substituting $\DS$ for $f$ in the LHS of Equation~\eqref{52rec}
gives a RHS of
\be
\label{DSrec}
\begin{aligned}
  e^{(m-1)\ve}(1-e^{\ve})^{2}H_{0}(\ve,\delta;q)
  -q^{m+4}e^{(m+1)\ve}(1-q^{-1}e^{-\ve})(1-e^{\ve})^{3}
  (1-q^{-1}e^{\ve})H_{-1}(\ve,\delta;q).
\end{aligned}
\ee
In particular, for $|q|<1$ Equation~\eqref{DSrec} is
\be
\label{DSrec1}
\begin{aligned}
  \frac{(qe^{\ve-\delta};q)_{\infty}(qe^{\delta};q)_{\infty}}{
    (qe^{\ve};q)_{\infty}(q;q)_{\infty}e^{\frac{1}{12}\ve^2
      -\frac{1}{12}(\ve\delta-\delta^2)E_{2}(q)}}
  \bigg(e^{m\ve}(1-e^{\ve})^{2}
  \sum_{k\in\BZ}\frac{q^{k(k+1)}e^{(2k+1)\delta}}{(qe^{\delta};q)_{k}
    (qe^{\ve-\delta};q)_{-k}}\qquad
  \\
  +q^{m+3}e^{(m+1)\ve}(1-q^{-1}e^{-\ve})
  (1-e^{\ve})^{2}(1-q^{-1}e^{\ve})
  \sum_{k\in\BZ}\frac{q^{k(k+1)}e^{(2k+1)\delta}}{(qe^{\delta};q)_{k}
    (qe^{\ve-\delta};q)_{-1-k}}\bigg)\\
  =
  \ve^{2}(1+O(\delta))+O(\ve^{3})
\end{aligned}
\ee
and for $|q|>1$ Equation~\eqref{DSrec} is
\begin{small}
\be\label{DSrec2}
\begin{aligned}
  \frac{(q^{-1}e^{\ve+\delta};q^{-1})_{\infty}}{
    (q^{-1}e^{\delta};q^{-1})_{\infty}^{2}
    (e^{-\delta};q^{-1})_{\infty}(q^{-1};q^{-1})_{\infty}}
  \bigg(e^{(m+1/2)\ve}(1-e^{\ve})^{2}
  \sum_{k\in\BZ}(-1)^{k}q^{-k(k+1)/2}e^{\delta k}
  \frac{(q^{-1}e^{\delta};q^{-1})_{k-1}}{
    (q^{-1}e^{\ve+\delta};q^{-1})_{k-1}}\;\\
  +q^{m+3}e^{(m+3/2)\ve}(1-q^{-1}e^{-\ve})
  (1-e^{\ve})^{3}(1-q^{-1}e^{\ve})
  \sum_{k\in\BZ}(-1)^{k}q^{-k(k+1)/2}e^{\delta k}
  \frac{(q^{-1}e^{\delta};q^{-1})_{k-1}}{
    (q^{-1}e^{\ve+\delta};q^{-1})_{k}}\bigg)\\
  =
  \ve^{2}+O(\ve^{3}).
\end{aligned}
\ee
\end{small}

\subsection{A $6 \times 6$ matrix of $q$-series}
\label{sub.52q6}

We now have all the ingredients to define the promised $6 \times 6$ matrix
$\bJ_{m}(q)$ of $q$-series for $|q| \neq 1$. Let us denote by $\DS_{m}^{(a,b)}(q)$
the coefficient of $\ve^{a}\delta^{b}$ in the expansion of $\DS_{m}(q)$.
We now define
{\small
\be
\label{Jq6x6}
\begin{aligned}
  \bJ_{m}(q)=
  \begin{pmatrix}
    1 & \DS^{(2,0)}_{m}(q) & \DS^{(2,0)}_{m+1}(q) &
    \DS^{(2,0)}_{m+2}(q) & \DS^{(2,0)}_{m+3}(q) & \DS^{(2,0)}_{m+4}(q)\\
    0 & \DS^{(0,0)}_{m}(q) & \DS^{(0,0)}_{m+1}(q) & \DS^{(0,0)}_{m+2}(q)
    & \DS^{(0,0)}_{m+3}(q) & \DS^{(0,0)}_{m+4}(q)\\
    0 & \DS^{(-1,2)}_{m}(q) & \DS^{(-1,2)}_{m+1}(q) & \DS^{(-1,2)}_{m+2}(q)
    & \DS^{(-1,2)}_{m+3}(q) & \DS^{(-1,2)}_{m+4}(q)\\
    0 & \DS^{(0,2)}_{m}(q) & \DS^{(0,2)}_{m+1}(q) & \DS^{(0,2)}_{m+2}(q)
    & \DS^{(0,2)}_{m+3}(q) & \DS^{(0,2)}_{m+4}(q)\\
    0 & \DS^{(1,0)}_{m}(q) & \DS^{(1,0)}_{m+1}(q) & \DS^{(1,0)}_{m+2}(q)
    & \DS^{(1,0)}_{m+3}(q) & \DS^{(1,0)}_{m+4}(q)\\
    0 & \DS^{(1,2)}_{m}(q) & \DS^{(1,2)}_{m+1}(q) & \DS^{(1,2)}_{m+2}(q)
    & \DS^{(1,2)}_{m+3}(q) & \DS^{(1,2)}_{m+4}(q)
  \end{pmatrix}
  \qquad (|q|<1),\\
  \bJ_{m}(q)=
  \begin{pmatrix}
    1 & \DS^{(2,0)}_{m}(q) & \DS^{(2,0)}_{m+1}(q)
    & \DS^{(2,0)}_{m+2}(q) & \DS^{(2,0)}_{m+3}(q) & \DS^{(2,0)}_{m+4}(q)\\
    0 & \DS^{(1,-2)}_{m}(q) & \DS^{(1,-2)}_{m+1}(q) & \DS^{(1,-2)}_{m+2}(q)
    & \DS^{(1,-2)}_{m+3}(q) & \DS^{(1,-2)}_{m+4}(q)\\
    0 & \DS^{(2,-2)}_{m}(q) & \DS^{(2,-2)}_{m+1}(q) & \DS^{(2,-2)}_{m+2}(q)
    & \DS^{(2,-2)}_{m+3}(q) & \DS^{(2,-2)}_{m+4}(q)\\
    0 & \DS^{(1,0)}_{m}(q) & \DS^{(1,0)}_{m+1}(q) & \DS^{(1,0)}_{m+2}(q)
    & \DS^{(1,0)}_{m+3}(q) & \DS^{(1,0)}_{m+4}(q)\\
    0 & \DS^{(0,-2)}_{m}(q) & \DS^{(0,-2)}_{m+1}(q) & \DS^{(0,-2)}_{m+2}(q)
    & \DS^{(0,-2)}_{m+3}(q) & \DS^{(0,-2)}_{m+4}(q)\\
    0 & \DS^{(0,0)}_{m}(q) & \DS^{(0,0)}_{m+1}(q) & \DS^{(0,0)}_{m+2}(q)
    & \DS^{(0,0)}_{m+3}(q) & \DS^{(0,0)}_{m+4}(q)
  \end{pmatrix}
  \qquad (|q|>1)\,.
\end{aligned}
\ee
}
The next theorem relates the above matrix to the linear $q$-difference
equation~\eqref{52rec}.

\begin{theorem}
\label{thm.51bJ1}
  The matrix $\bJ_{m}(q)$ is a fundamental solution to the linear $q$-difference
  equation
\begin{small}
\be
\label{52Jqrec}
\bJ_{m+1}(q) = \bJ_{m}(q) A(q^m,q),\; A(q^m,q)=
\begin{pmatrix}
    1 & 0 & 0 & 0 & 0 & -q^{-2m-10}\\
    0 & 0 & 0 & 0 & 0 & (1-q^{m+1})(1-q^{m+2})q^{-2m-10}\\
    0 & 1 & 0 & 0 & 0 & (3 + 2q)(1-q^{m+1})q^{-m-8}\\
    0 & 0 & 1 & 0 & 0 & (q^{m+4} + 6q^{m+3} + 3q^{m+2} - q - 4)q^{-m-7}\\
    0 & 0 & 0 & 1 & 0 & (-3q^{m+4} - 6q^{m+3} - q^{m+2} + 3)q^{-m-6}\\
    0 & 0 & 0 & 0 & 1 & (3q^{m+4} + 2q^{m+3} - 1)q^{-m-5}
  \end{pmatrix} \,.
\ee
\end{small}
and has
\be
\begin{aligned}
  \det(\bJ_{m}(q))
  &=q^{-20-7m}(q;q)_{\infty}^{9}(q^{-m-1};q)_{\infty}(q^{-m};q)_{\infty}
  &\qquad (|q|<1),\\
  \det(\bJ_{m}(q))
  &=
  q^{-20-7m}(q^{-1};q^{-1})_{\infty}^{-9}(q^{-m-1};q^{-1})_{\infty}^{-1}
  (q^{-m-2};q^{-1})_{\infty}^{-1}
  &\qquad (|q|>1).
\end{aligned}
\ee
\end{theorem}
\begin{proof}
  Equation~\eqref{52Jqrec} follows from Equations~\eqref{DSrec1},~\eqref{DSrec2}. The
  determinant is calculated using the determinant of $A(q^{m},q)$ and by considering
  the limiting behavior in $m$.
\end{proof}
The construction of this matrix has used special $q$-hypergeometric formulae for the
Habiro polynomials. However, this construction can be carried out more generally and
will be developed in a later publication. 

There is a similar, however more complicated, relation between
$\bJ_{-m}(q^{-1})$ with the first row replaced by Appell-Lerch type
sums and $\bJ_{m}(q)^{-1}$ as in Theorem~\ref{thm.41bJ2}. This
indicates these matrices could come from the factorisation of a
state-integral. We will not give this relation, since we do not need
it for the purpose of resurgence. We will however, discuss an
important block property of the matrix $\bJ_{-2}(q)$, after a gauge transformation.
Namely, we define:
\be
\label{J52norm}
\bJ^{\mathrm{norm}}(q)
=
\bJ_{-2}(q)
\begin{pmatrix}
1 & 0 & 0 & 0 & 0 & 0\\
0 & 0 & 0 & 0 & 0 & q^{-1}-1\\
0 & 0 & 0 & 0 & 1 & -3\\
0 & -q & q & 3q^2 & 0 & 2q\\
0 & 0 & q^2 & q^2-3q^3 & 0 & -q^2\\
0 & 0 & 0 & q^4 & 0 & 0
\end{pmatrix} \,.
\ee

The first few terms of the matrix $\bJ^{\mathrm{norm}}(q)+Q(q^3)$ are given by
{\tiny
\be
\begin{pmatrix}
  1 & -\frac{1}{12} + \frac{25}{12}q + 4q^2
  & -\frac{5}{6} - \frac{19}{6}q - \frac{95}{12}q^2
  & \frac{1}{12} - 2q - \frac{83}{12}q^2 & -\frac{5}{12} + \frac{11}{12}q - 3q^2
  & \frac{5}{12} - \frac{1}{2}q + 2q^2\\
0 & 1 - q & -2 + 2q - q^2 & -1 - q^2 & -1 + q & 1\\
0 & -1 + 4q + q^2 & 1 - 7q + 2q^2 & -q + q^2 & 1 - 3q - q^2 & q^2\\
0 & \frac{5}{12} - \frac{35}{12}q + \frac{13}{2}q^2
& \frac{2}{3} + \frac{4}{3}q - \frac{263}{12}q^2
& \frac{1}{12} - \frac{5}{2}q - \frac{137}{12}q^2
& -\frac{17}{12} + \frac{53}{12}q - \frac{13}{2}q^2
& -\frac{1}{12} + 4q + \frac{11}{2}q^2\\
0 & 0 & 0 & 0 & 1 - 2q & -1 + q + 2q^2\\
0 & 0 & 0 & 0 & \frac{11}{12} - \frac{11}{6}q + 10q^2
& \frac{1}{12} - \frac{61}{12}q - \frac{1}{6}q^2
\end{pmatrix} \,.
\ee
}

We next discuss a block structure for the gauged-transform matrix~\eqref{J52norm}. 
\begin{conjecture}
\label{prop.52block}
When $|q|<1$, the matrix $\bJ^{\mathrm{norm}}(q)$ has a block form
\be
\begin{pmatrix}
1\times1 & 1\times3 & 1\times2\\
0 & 3\times 3 & 3\times 2\\
0 & 0 & 2\times2
\end{pmatrix} \,.
\ee
\end{conjecture}
Our next task is to identify the $3\times3$ and the $2\times2$ blocks of the
matrix $\bJ^{\mathrm{norm}}(q)$. The first observation is that the $3\times 3$ block
is related to the $3\times3$ matrix given in~\cite{GGM:resurgent}. The second is that
the $2\times 2$ block is related to modular forms. This is the content of the next
conjecture.

\begin{conjecture}
\label{DS2hid}
The $3\times3$ block for $|q|<1$ of $\bJ^{\mathrm{norm}}(q)$ of~\eqref{J52norm} has
the form
\be
(q;q)_{\infty}
\bJ_{-1}^\RED(q)
\begin{pmatrix}
  0&0&1\\
  -1&3&0\\
  0&-1&0
\end{pmatrix}
\ee
(where $\bJ_{m}^\RED(q)$ is the $3\times3$ matrix of~\cite{GGM:resurgent} reviewed
in Section~\ref{sub.52q3}) and the $2\times2$ block has the form
\be
(q;q)_{\infty}^{2}
\begin{pmatrix}
  H(q) & G(q)\\
  * & *
\end{pmatrix}
\ee
where
\be
  H(q)=\sum_{k=0}^{\infty}\frac{q^{k^{2}+k}}{(q;q)_{k}}
  \quad\text{and}\quad
  G(q)=\sum_{k=0}^{\infty}\frac{q^{k^{2}}}{(q;q)_{k}}
\ee
are the famous Rogers-Ramanujan functions. 
\end{conjecture}

The remaining two entries of the $2 \times 2$ block are higher weight vector-valued
modular forms associated to the same $\SL_{2}(\BZ)$-representation as the
Rogers-Ramanujan functions, discussed for example in~\cite{SW:NCJac}. Part of this
conjecture is proved in Appendix~\ref{Appqserids}.

This block decomposition fits nicely with the ``dream''
in~\cite{DZ:lec}. Here we do see the interesting property that the
$1\times 2$ and $3\times 2$ blocks contain some non-trivial gluing
information. This implies that the diagrammatic ``short exact
sequence'' will not always ``split''.  The block decomposition also
implies that the resurgent structure of the asymptotic series
associated to the $q$-series in the $4\times4$ block in the top left
does not depend on the other blocks. This block and in-particular the
second column of $\bJ^{\mathrm{norm}}$ will be the focus of
Section~\ref{sub.52borel}.

We now consider the analytic properties of the function
\be
\label{W52}
W(\tau)=\bJ^{\mathrm{norm}}(e(\tau))^{-1}
\begin{pmatrix}
\tau^2 & 0 & 0 & 0 & 0 & 0\\
0 & 1 & 0 & 0 & 0 & 0\\
0 & 0 & \tau & 0 & 0 & 0\\
0 & 0 & 0 & \tau^2 & 0 & 0\\
0 & 0 & 0 & 0 & \tau & 0\\
0 & 0 & 0 & 0 & 0 & \tau^3
\end{pmatrix}
\bJ^{\mathrm{norm}}(\e(-1/\tau)), \qquad (\tau \in \BC\setminus\BR) \,.
\ee
If the work~\cite{GZ:qseries} extended to the $6 \times 6$ matrix, it would imply
that the function $W$ extends to an analytic function on $\BC'$. This would follow
from an identification of $W$ with a matrix of state-integrals, as was done in
Section~\ref{sub.3x3state} for the $\knot{4}_1$ knot. Although we do not know of such
a matrix of state-integrals, we can numerically evaluate $W$ when $\tau$ is near
the positive real axis and test the extension hypothesis. Doing so for 
$\tau=1+\frac{\ri}{100}$ we have
\begin{tiny}
\be
\begin{split}
&\bJ^{\mathrm{norm}}(\e(-1/\tau))\\&=
\begin{pmatrix}
  1 & 1.9E^{9} + 3.8E^{8}\ri & -5.1E^{9} - 9.9E^{8}\ri & -4.5E^{9} - 8.8E^{8}\ri
  & -1.2E^{9} - 2.5E^{8}\ri & 2.9E^{9} + 5.7E^{8}\ri\\
  0 & 2.4E^{6} + 4.1E^{5}\ri & -6.1E^{6} - 1.0E^{6}\ri & -5.4E^{6} - 9.5E^{5}\ri
  & -1.5E^{6} - 2.7E^{5}\ri & 3.5E^{6} + 6.1E^{5}\ri\\
  0 & -1.3E^{-20} + 1.0E^{-20}\ri & 1.7E^{-20} - 2.6E^{-20}\ri & -6.2E^{-22} - 5.1E^{-21}\ri
  & 9.1E^{-21} - 2.5E^{-21}\ri & -4.0E^{-21} + 3.8E^{-21}\ri\\
  0 & 1.9E^{9} + 3.8E^{8}\ri & -5.1E^{9} - 9.9E^{8}\ri & -4.5E^{9} - 8.8E^{8}\ri
  & -1.2E^{9} - 2.5E^{8}\ri & 2.9E^{9} + 5.7E^{8}\ri\\
0 & 0 & 0 & 0 & 3.1E^{-17} - 1.3E^{-17}\ri & -5.0E^{-17} + 2.1E^{-17}\ri\\
0 & 0 & 0 & 0 & 2.6E^{-14} - 1.0E^{-14}\ri & -4.2E^{-14} + 1.7E^{-14}\ri
\end{pmatrix}
\end{split}
\ee
\end{tiny}
where $\e(x)=e^{2 \pi \ri x}$ whereas
\begin{tiny}
\be
W(\tau)=
\begin{pmatrix}
  0.99 - 0.019\ri & -0.10 - 0.028\ri & 0.24 - 0.25\ri & 0.060 - 0.43\ri & -0.064 + 0.059\ri
  & -0.18 - 0.094\ri\\
0 & 0.59 - 1.0\ri & 1.0 + 1.3\ri & 0.19 - 0.13\ri & -0.60 - 0.20\ri & -0.48 - 0.22\ri\\
0 & -0.17 - 0.17\ri & 1.2 - 0.30\ri & 0.024 - 0.31\ri & -0.14 - 0.0076\ri & -0.17 + 0.030\ri\\
0 & 0.028 - 0.31\ri & 0.097 + 1.1\ri & 1.0 + 0.46\ri & -0.17 + 0.030\ri & -0.12 - 0.53\ri\\
0 & 0 & 0 & 0 & 0.17 - 0.83\ri & -0.44 - 0.25\ri\\
0 & 0 & 0 & 0 & -0.46 - 0.26\ri & 0.63 - 0.56\ri
\end{pmatrix} \,.
\ee
\end{tiny}

\subsection{Borel resummation and Stokes constants}
\label{sub.52borel}

The $\knot{5}_2$ knot has four asymptotic series $\Phi^{(\s_j)}(\tau)$
for $j=0,1,2,3$ corresponding to the trivial, the geometric, the
conjugate, and the real flat connections respectively, denoted
respectively by $\s_j$ for $j=0,1,2,3$.  Similar to the $\knot{4}_1$
knot, the asymptotic series $\Phi^{(\s_j)}(\tau)$ for $j=1,2,3$ can be
defined in terms of a perturbation theory of a state-integral
\cite{KLV,AK} using the standard formal Gaussian integration as
explained in \cite{DGLZ,GGM:resurgent}, and they have been computed in
\cite{GGM:resurgent} with more than 200 terms.  Let $\xi_{j}$
($j=1,2,3$) be the roots to the algebraic equation
\begin{equation}
  (1-\xi)^3 = \xi^2
\end{equation}
with numerical values
\begin{equation}
  \xi_1 = 0.78492\ldots +1.30714\ldots\ri,\quad
  \xi_2 = 0.78492\ldots -1.30714\ldots\ri,\quad
  \xi_3 = 0.43016\ldots.
  \label{eq:xi}
\end{equation}
The asymptotic series $\Phi^{(\s_j)}(\tau)$  for $j=1,2,3$ have the
universal form\footnote{The series $\Phi^{(\s_j)}(\tau)$ ($j=1,2,3$) are related
  to the series in \cite{GGM:resurgent,GGM:peacock}, which we will
  denote by $\Phi^{(\s_j)}_{\text{GGM}}(\tau)$, by a common prefactor
  \begin{equation}
    \Phi^{(\s_j)}(\tau) = \ri
    \re^{-\frac{\pi\ri}{12}(\tau+\tau^{-1})-2\pi\ri\tau}
    \Phi^{(\s_j)}_{\text{GGM}}(\tau),\quad j=1,2,3.
  \end{equation}
  The Stokes constants associated to the Borel resummation of
  $\Phi^{(\s_j)}_{\text{GGM}}(\tau)$ are not changed. The additional
  prefactor is introduced so that the Stokes automorphism between
  $\Phi^{(\s_0)}(\tau)$ and $\Phi^{(\s_j)}(\tau)$ ($j=1,2,3$) can be
  presented in an elegant form, and is also dictated by positions of
  singularities of Borel transform of $\Phi^{(\s_0)}(\tau)$.}
\begin{equation}
  \Phi^{(\s_j)}(\tau) = \frac{\re^{\frac{3\pi\ri}{4}}}{\sqrt{\delta_j}}
  \re^{\frac{V_j}{2\pi\ri\tau}}  \varphi^{(\s_j)}(\tau),\quad j=1,2,3,
\end{equation}
where $\delta_j = 5-3\xi_j+3\xi_j^2$  and
\begin{equation}
\label{eq:V52}
  \begin{aligned}
    V_1 =
    &3\Li_2(\xi_1)+3/2\log(\xi_1)\log(1-\xi_1)
      -\pi\ri\log(\xi_1)-\frac{\pi^2}{3} 
    \\
    V_2 =
    &3\Li_2(\xi_2)+3/2\log(\xi_2)\log(1-\xi_2)
      +\pi\ri\log(\xi_2)-\frac{\pi^2}{3},
    \\
    V_3 =
    &3\Li_2(\xi_3)+3/2\log(\xi_3)\log(1-\xi_3)
      -\frac{\pi^2}{3}.
  \end{aligned}  
\end{equation}
Their numerical values are given by 
\begin{equation}
  V_1= 3.0241\ldots + 2.8281\ldots\ri,\quad
  V_2= 3.0241\ldots - 2.8281\ldots\ri,\quad
  V_3= -1.1134\ldots.
  \label{eq:V-52}
\end{equation}
where the common absolute value of the imaginary parts of
$V_1,V_2$ is the $\text{Vol}(S^3\backslash \knot{5}_2)$.
Finally the power series $\varphi^{(\s_j)}(\tau/(2\pi\ri)$
have coefficients in the number field $\BQ(\xi_j)$ and their first few
coefficients are given by
\begin{equation}
  \varphi^{(\s_j)}\left(\frac{\tau}{2\pi\ri}\right) =
  1+\frac{1452 \xi_j^2-1254 \xi_j +15949}{2^3\cdot 3\cdot 23^2}\tau +
  \frac{2124948 \xi_j ^2-2258148 \xi_j +11651375}
  {2^7\cdot 3^2\cdot 23^3}\tau^2 +\ldots
\end{equation}

The additional new series $\Phi^{(\s_0)}(\tau) \in \IQ[[\tau]]$
corresponds to the zero volume ($V(\s_0)=0$) trivial flat connection.
As exlained in Section~\ref{sub.413Phi}, it can be computed using the
colored Jones polynomial or the Kashaev invariant.  The first few
terms are
\begin{equation}
  \Phi^{(\s_0)}(\tfrac{\tau}{2\pi\ri})
  =\varphi^{(\s_0)}(\tfrac{\tau}{2\pi\ri}) 
  =  1 + 2\tau^2 + 6\tau^3+\frac{157}{6}\tau^4 + \ldots
\end{equation}

\begin{figure}[htpb!]
  \centering
  \subfloat[$\Lambda^{(\s_0)}$]{\includegraphics[height=6cm]{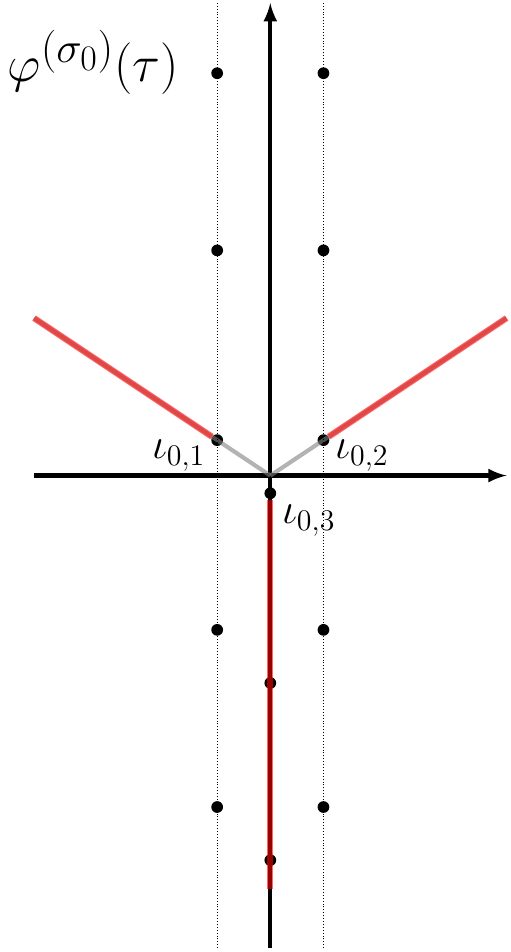}}%
  \hspace{3ex}
  \subfloat[$\Lambda^{(\s_1)}$]{\includegraphics[height=6cm]{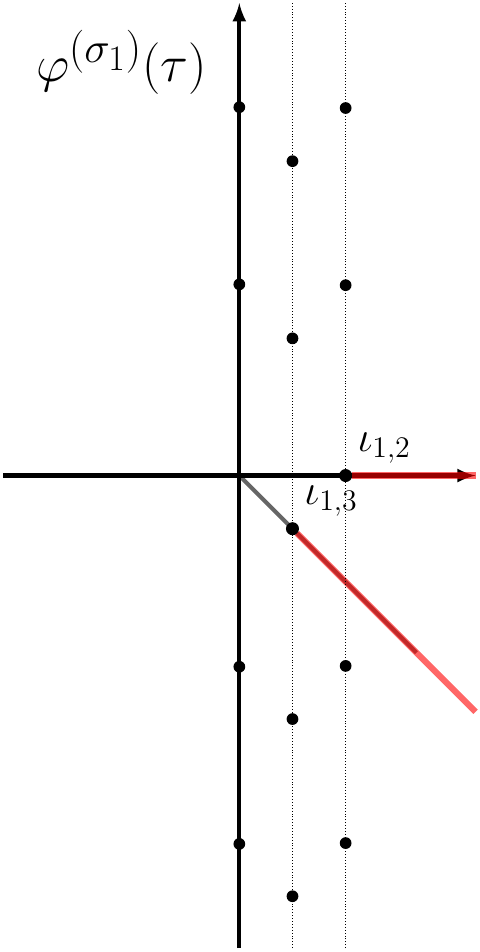}}%
  \hspace{3ex}
  \subfloat[$\Lambda^{(\s_2)}$]{\includegraphics[height=6cm]{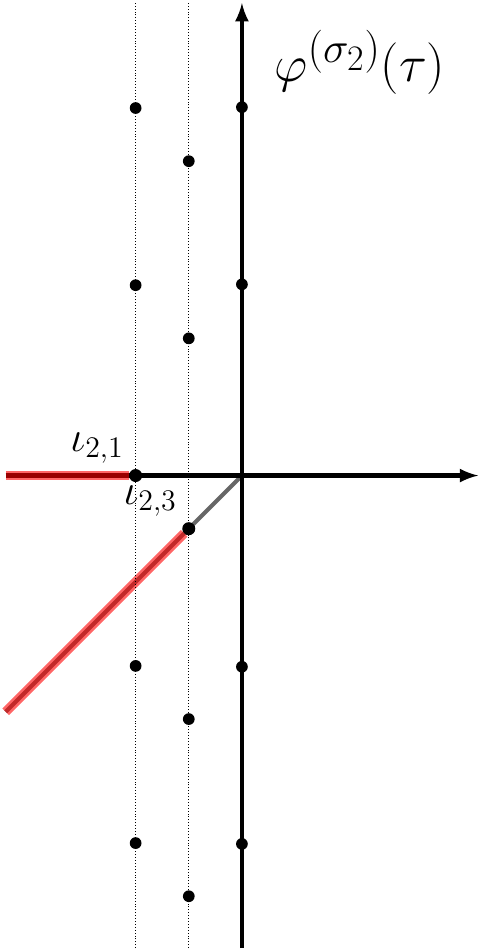}}%
  \hspace{3ex}
  \subfloat[$\Lambda^{(\s_3)}$]{\includegraphics[height=6cm]{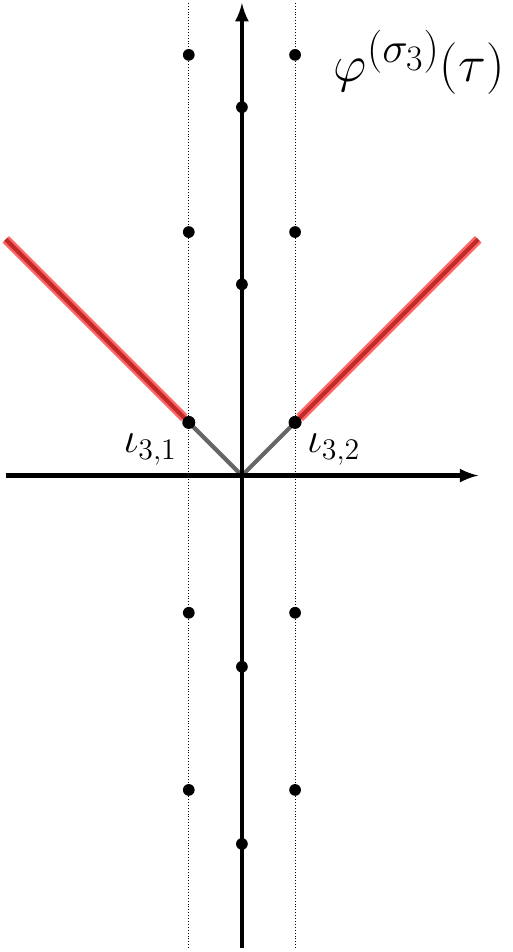}}%
  \caption{Singularities of Borel transforms of
    $\varphi^{(\s_j)}(\tau)$ for $j=0,1,2,3$ of the knot
    $\knot{5}_2$.}
  \label{fg:52_brplane}
\end{figure}

We are interested in the Stokes automorphism of the Borel resummation
of the 4-vector $\Phi(\tau)$ of asymptotic series
\begin{equation}
  \Phi(\tau) =
  \begin{pmatrix}
    \Phi^{(\s_0)}(\tau)\\
    \Phi^{(\s_1)}(\tau)\\
    \Phi^{(\s_2)}(\tau)\\
    \Phi^{(\s_3)}(\tau)
  \end{pmatrix}.
\end{equation}
First of all, the Borel transform of each asymptotic series
$\Phi^{(\s_j)}(\tau)$ ($j=0,1,2,3$) has rich patterns of
singularities.  Similar to the case of $\knot{4}_1$ knot discussed in
Section~\ref{sub.41borel}, the Borel transforms of
$\Phi^{(\s_j)}(\tau)$, $j=1,2,3$ have singularities located at
\begin{equation}
  \Lambda^{(\s_j)} = \{\iota_{j,i} + 2\pi\ri k
  \,|\, i=1,2,3,i\neq j,\,k\in\IZ \} \cup
  \{ 2\pi\ri k \,|\, k\in\IZ_{\neq 0}
  \},\quad j=1,2,3
\end{equation}
as shown in the right three panels of Fig.~\ref{fg:52_brplane},
while the Borel transform of $\Phi^{(\s_0)}(\tau)$ have singularities
located at (some of these singular points are actually missing as we
will comment in the end of the section.)
\begin{equation}
  \Lambda^{(\s_0)} = \{\iota_{0,i}+2\pi\ri k
  \,|\, i=1,2,3,\,k\in\IZ\},
\end{equation}
as shown in the left most panel of Fig.~\ref{fg:52_brplane}, where
\begin{equation}
  \iota_{j,i} = \frac{V_j-V_i}{2\pi\ri},\quad i,j=0,1,2,3.
\end{equation}

To any singularity located at $\iota_{i,j}^{(k)}:= \iota_{i,j}+2\pi\ri
k$ in the union 
\begin{equation}
  \Lambda = \cup_{j=0,1,2,3} \Lambda^{(\s_j)},
  \label{eq:Lmb52}
\end{equation}
we can associate a local Stokes matrix
\begin{equation}
  \mf{S}_{\iota_{i,j}^{(k)}}(\tq) = I + \mc{S}_{i,j}^{(k)}\tq^k
  E_{i,j},\quad \mc{S}_{i,j}^{(k)}\in\IZ,
\end{equation}
where $E_{i,j}$ is the $4\times 4$ elementary matrix with
$(i,j)$-entry 1 ($i,j=0,1,2,3$) and all other entries zero, and
$\mf{S}_{i,j}^{(k)}$ is the Stokes constant.
Then the Borel resummation along the rays $\rho_{\theta_\pm}$ raised
slight above and below the angle $\theta$ are related by the Stokes
automorphism
\begin{equation}
  \Delta(\tau)s_{\theta_+}(\Phi)(\tau) =
  \mf{S}_{\theta}(\tq)\Delta(\tau) s_{\theta_-}(\tau),
\end{equation}
where
\begin{equation}
  \mf{S}_\theta(\tq) = \prod_{\arg\iota =
    \theta}\mf{S}_{\iota}(\tq),\quad
  \Delta(\tau) = \diag(\tau^{3/2},1,1,1),
\end{equation}
and the locality condition is assumed.

\begin{figure}[htpb!]
  \centering
  \includegraphics[height=8cm]{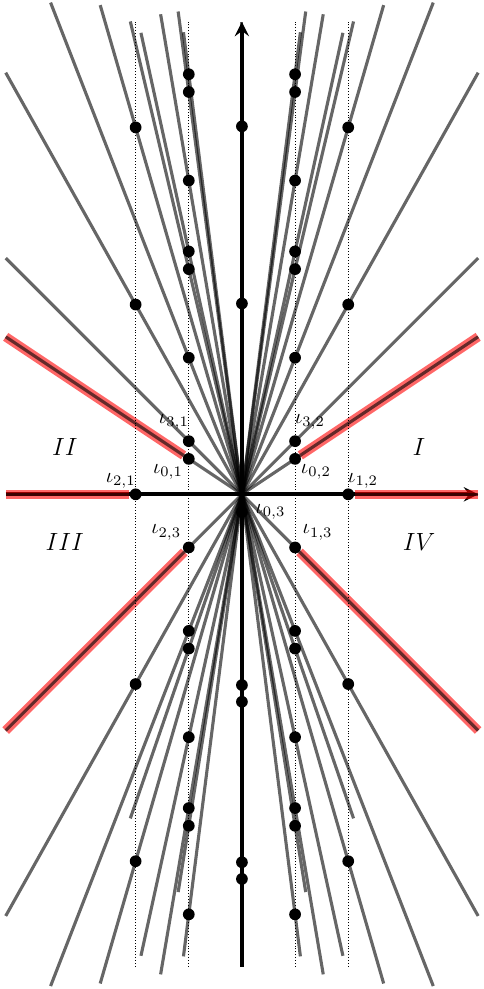}
  \caption{Stokes rays and cones in the $\tau$-plane for the 4-vector
    $\Phi(\tau)$ of asymptotic series of the knot $\knot{5}_2$.}
  \label{fg:52_rays}
\end{figure}

More generally, for two rays $\rho_{\theta^+}$ and $\rho_{\theta^-}$
whose arguments satisfy $0<\theta^+-\theta^-\leq\pi$, we can define
the global Stokes matrix $\mf{S}_{\theta^-\rightarrow \theta^+}$ as in
\eqref{eq:S-glob}, and it also satisfies the factorisation property
\eqref{eq:S-fac}.  Since the factorisation is unique
\cite{GGM:resurgent,GGM:peacock}, we only need to compute finitely
many global Stokes matrices in order to extract all the local Stokes
matrices associated to the infinitely many singularities in $\Lambda$
and thus the corresponding Stokes constants.  In particular, we can
choose four cones $I,II,III,IV$ slightly above the positive and the
negative real axes as shown in Fig.~\ref{fg:52_rays}, and compute the
four global Stokes matrices
\begin{equation}
  \mf{S}_{I\rightarrow II},\;\mf{S}_{II\rightarrow III},\;
  \mf{S}_{III\rightarrow IV},\;\mf{S}_{IV\rightarrow I},
  \label{eq:SRR}
\end{equation}
where a cone $R$ in the subscript means any ray inside the cone.

On the other hand, each of the global Stokes matrices in
\eqref{eq:SRR} has the block upper triangular form
\begin{equation}
  \begin{pmatrix}
    1&*&*&*\\
    0&*&*&*\\
    0&*&*&*\\
    0&*&*&*
  \end{pmatrix}.
\end{equation}
The $3\times 3$ sub-matrices $\mf{S}^{\text{red}}_{R\rightarrow R'}$
in the right bottom have been worked out in \cite{GGM:resurgent}.  For
later convenience, we write down two of the four reduced global Stokes
matrices,

\begin{subequations}
\begin{align}
\mf{S}^{\text{red}}_{I\rightarrow II}(\tq)
&=\frac{1}{2}
\begin{pmatrix}
0&1&0\\0&1&1\\-1&0&0
\end{pmatrix}\bJ^\RED_{-1}(\tq^{-1})^T
\begin{pmatrix}
0&0&1\\0&-2&0\\1&0&0
\end{pmatrix}\bJ^\RED_{-1}(\tq)
\begin{pmatrix}
0&0&-1\\1&-3&0\\0&1&0
\end{pmatrix},
\quad
|\tq|<1,\label{eq:SRed12}
\\
\mf{S}^{\text{red}}_{III\rightarrow IV}(\tq)
&=\frac{1}{2}
\begin{pmatrix}
      1&-3&0\\0&1&0\\0&0&-1
\end{pmatrix}\bJ^\RED_{-1}(\tq^{-1})^T
\begin{pmatrix}
  0&0&1\\0&-2&0\\1&0&0
\end{pmatrix}\bJ^\RED_{-1}(\tq)
\begin{pmatrix}
  1&0&0\\1&1&0\\0&0&-1
\end{pmatrix},
\quad |\tq| > 1.
\label{eq:SRed34}
\end{align}
\end{subequations}

In addition, as seen from Fig.~\ref{fg:52_brplane}, there are no
singularities along the positive and negative real axes in
$\Lambda^{(\s_0)}$ relevant for $\Phi^{(\s_0)}(\tau)$; all the
singular points in $\Lambda^{(\s_0)}$ are either in the upper half
plane beyond the cones $I,II$ or in the lower half plane beneath the
cones $III,IV$. Consequently we only need to compute the first row of
two Stokes matrices $\mf{S}_{I\rightarrow II}$ and
$\mf{S}_{III\rightarrow IV}$.  For this purpose, we find the following.

\begin{conjecture}
  For every cone $R \subset \BC\setminus\Lambda$ and every $\tau \in R$,
  we have
  \begin{equation}
    Q_0^{(2,0)}(q) = s_R(\Phi^{(\s_0)})(\tau) +
    \tau^{-3/2}\sum_{j=1}^3 M_{R,j}(\tq)s_R(\Phi^{(\s_j)})(\tau),
    \label{eq:QMPhi}
  \end{equation}
  where $M_{R,j}(\tq)$ ($j=1,2,3$) are $\tq$ (resp., $\tq^{-1}$)-series 
  if $\Im\tau>0$ (resp., $\Im\tau<0$) with integer coefficients that depend on
  $R$.
\end{conjecture}

A more elegant way to present $M_{R,j}(\tq)$ is by the row vector
$M_R(\tq) := (M_{R,1},M_{R,2},M_{R,3})(\tq)$, and it can be expressed
in terms of a $3\times 3$ matrix $M^{(\s_0)}_R(\tq)$
\begin{equation}
  M_R(\tq)
  = \left(\tq Q^{(2,0)}_0(\tq),\tq^2Q_1^{(2,0)}(\tq),\tq^3Q_2^{(2,0)}(\tq)\right)
  M^{(\s_0)}_R(\tq).
  \label{eq:MMs052}
\end{equation}
\begin{conjecture}
  Equation \eqref{eq:QMPhi} holds in the cones $R=I,II,III,IV$ where
  the $\tq$,$\tq^{-1}$-series $M_{R,j}(\tq)$ are given in terms of
  $M_R^{(0)}(\tq)$ through \eqref{eq:MMs0} which are as follows
  \begin{subequations}
  \begin{align}
    M^{(\s_0)}_{I}(\tq) =
    &\begin{pmatrix}
       1&-1&-3\tq\\0&-1&-1+3\tq\\0&0&-\tq
     \end{pmatrix},\\
    M^{(\s_0)}_{II}(\tq) =
    &\begin{pmatrix}
       -1&1&-3\tq\\-1&0&-1+3\tq\\0&0&-\tq
     \end{pmatrix},\\
    M^{(\s_0)}_{III}(\tq) =
    &\begin{pmatrix}
        3&1&-3\tq\\-1&0&-1+3\tq\\0&0&-\tq
     \end{pmatrix},\\
    M^{(\s_0)}_{IV}(\tq) =
    &\begin{pmatrix}
     1&3&-3\tq\\0&-1&-1+3\tq\\0&0&-\tq   
     \end{pmatrix}.
  \end{align}
\end{subequations}
\end{conjecture}


Eqs.~\eqref{eq:QMPhi}, together with the reduced Stokes matrices
$\mf{S}^{\text{red}}_{R\rightarrow R'}(\tq)$ for $\Phi^{(\s_j)}(\tau)$
($j=1,2,3$), allow us to calculate entries in the first row of
$\mf{S}_{I\rightarrow II}(\tq)$ and $\mf{S}_{III\rightarrow IV}(\tq)$
by
\begin{equation}
  \mf{S}_{R\rightarrow R'}(\tq)_{0,j} =
  M_{R,j}(\tq) - \sum_{k=1}^3M_{R',k}(\tq)
  \mf{S}^{\text{red}}_{R\rightarrow R'}(\tq)_{k,j},\quad j=1,2,3.
\end{equation}
In the following we list the first few terms of these
$\tq$ and $\tq^{-1}$-series.  In the upper half plane
\begin{subequations}
  \begin{align}
    \mf{S}_{I\rightarrow II}(\tq)_{0,1}=
    &-1+13\tq-12\tq^2-82\tq^3-29\tq^4+85\tq^5+\cO(\tq^6) ,\\
    \mf{S}_{I\rightarrow II}(\tq)_{0,2}=
    &1-16\tq+42\tq^2+135\tq^3-54\tq^4-346\tq^5+\cO(\tq^6) ,\\
    \mf{S}_{I\rightarrow II}(\tq)_{0,3}=
    &-\tq+10\tq^2+18\tq^3-31\tq^4-92\tq^5+\cO(\tq^6).
  \end{align}
\end{subequations}
In the lower half plane
\begin{subequations}
  \begin{align}
    \mf{S}_{III\rightarrow IV}(\tq)_{0,1}=
    &4\tq^{-1}-4\tq^{-2}-51\tq^{-3}-62\tq^{-4}-27\tq^{-5}+\cO(\tq^{-6}) ,\\
    \mf{S}_{III\rightarrow IV}(\tq)_{0,2}=
    &3\tq^{-1}+2\tq^{-2}-26\tq^{-3}-47\tq^{-4}-64\tq^{-5}+\cO(\tq^{-6}) ,\\
    \mf{S}_{III\rightarrow IV}(\tq)_{0,3}=
    &-1+\tq^{-2}+18\tq^{-3}+39\tq^{-4}+73\tq^{-5}+\cO(\tq^{-6}).
  \end{align}
\end{subequations}

Finally, we can factorise the global Stokes matrices
$\mf{S}_{I\rightarrow II}(\tq), \mf{S}_{III\rightarrow IV}(\tq)$ to
obtain local Stokes matrices associated to individual singular points
in $\Lambda$ and extract the associated Stokes constants.
The Stokes constants for $\Phi^{(\s_j)}(\tau)$ ($j=1,2,3$) are already given
in \cite{GGM:resurgent,GGM:peacock}.
We collect the Stokes contants for $\Phi^{(\s_0)}(\tau)$ in the
generating series
\begin{equation}
  \ms{S}^+_{0,j}(\tq) = \sum_{k\geq 0} \mc{S}^{(k)}_{0,j}\tq^k,\quad
  \ms{S}^-_{0,j}(\tq) = \sum_{k\leq 0} \mc{S}^{(k)}_{0,j}\tq^k,\quad j=1,2,3.
\end{equation}
And we find that in the upper half plane
\begin{subequations}
  \begin{align}
    \ms{S}^+_{0,1}(\tq) =
    &-1+\tq+3\tq^2+25\tq^3+278\tq^4+3067\tq^5+\cO(\tq^6),\\
    \ms{S}^+_{0,2}(\tq) =
    &1-\tq-3\tq^2-25\tq^3-278\tq^4-3067\tq^5+\cO(\tq^6),\\
    \ms{S}^+_{0,3}(\tq) =
    &0,
  \end{align}
\end{subequations}
while in the lower half plane
\begin{subequations}
  \begin{align}
    \ms{S}^-_{0,1}(\tq) =
    &3\tq^{-1}-34\tq^{-2}+391\tq^{-3}-4622\tq^{-4}
      +54388\tq^{-5}+\cO(\tq^{-6}),\\
    \ms{S}^-_{0,2}(\tq) =
    &3\tq^{-1}-34\tq^{-2}+391\tq^{-3}-4622\tq^{-4}
      +54388\tq^{-5}+\cO(\tq^{-6}),\\
    \ms{S}^-_{0,3}(\tq) =
    &-1+9\tq^{-1}-56\tq^{-2}+705\tq^{-3}-8378\tq^{-4}
      +98379\tq^{-5}+\cO(\tq^{-6}).
  \end{align}
\end{subequations}
We comment that the results of $\ms{S}^+_{0,3}(\tq)$ and
$\ms{S}^-_{0,3}(\tq)$ indicate that there are actually no singular
points of the type $\iota_{0,3}^{(k)}$ in the upper half plane, but
they exist in the lower half plane.  Also note that the constant terms
in $\ms{S}^+_{0,1}(\tq),\ms{S}^+_{0,2}(\tq)$ and $\ms{S}^-_{0,3}(\tq)$
are Stokes constants associated to the singular points $\iota_{0,j}$
($j=1,2,3$).  The Stokes constants associated to $\iota_{i,j}$
$(i,j=1,2,3, i\neq j)$ have already been computed in
\cite{GGM:resurgent,GGM:peacock}. We can assemble all these Stokes
constants in a matrix
\begin{equation}
  \begin{pmatrix}
    0&-1&1&-1\\
    0& 0&4& 3\\
    0&-4&0&-3\\
    0&-3&3& 0
  \end{pmatrix}
\end{equation}
which matches (after some changes of signs) the one appearing
in~\cite[Eq.(40)]{GZ:kashaev}.

\subsection{$(x,q)$-series}
\label{sub.52qx}

In this section we extend the results of Section~\ref{sub.52q3} by including the
Jacobi variable $x$. Recall that the matrix $\bJ^{\RED}_{m}(x,q)$\footnote{The matrices
  $\bJ^\RED_m(x,q)$ are related to the Wronskians $W_m(x,q)$
  in~\cite{GGM:resurgent,GGM:peacock} by
  \begin{equation}
    \bJ^\RED_m(x,q) = W_m(x,q)^T.
  \end{equation}}
is a fundamental solution to the linear $q$-difference equation
\be
  f_m(x,q)-(1+x+x^{-1})f_{m+1}(x,q)+(1+x+x^{-1}-q^{2+m})f_{m+2}(x,q) - f_{m+3}(x,q) = 0
\ee
and it is defined by
\begin{equation}
  \label{Jred52x}
  \bJ^\RED_m(x,q) =
  \begin{pmatrix}
    A_m(x,q)&A_{m+1}(x,q)&A_{m+2}(x,q)\\
    B_m(x,q)&B_{m+1}(x,q)&B_{m+2}(x,q)\\
    C_m(x,q)&C_{m+1}(x,q)&C_{m+2}(x,q)\\
  \end{pmatrix},\quad |q|\neq 1,
\end{equation}
where the holomorphic blocks are given by
\begin{subequations}
  \begin{align}
    &A_m(x,q) = H(x,x^{-1},q^m;q),\\
    &B_m(x,q) = \theta(-q^{1/2}x;q)^{-2}x^m H(x,x^2,q^mx^2;q),\\
    &C_m(x,q) = \theta(-q^{-1/2}x;q)^{-2}x^{-m} H(x^{-1},x^{-2},q^mx^{-2};q),
  \end{align}
\end{subequations}
where $H(x,y,z;q^\varepsilon):=H^\varepsilon(x,y,z;q)$ for $|q|<1$ and
$\varepsilon=\pm$ and
\begin{subequations}
  \begin{align}
    H^+(x,y,z;q) &=
      (qx;q)_\infty(qy;q)_\infty\sum_{n=0}^\infty\frac{q^{n(n+1)}z^n}
      {(q;q)_n(qx;q)_n(qy;q)_n},\\
    H^-(x,y,z;q) &= \frac{1}{(x;q)_\infty(y;q)_\infty}\sum_{n=0}^\infty
      (-1)^n\frac{q^{\frac{1}{2}n(n+1)}x^{-n}y^{-n}z^n}
      {(q;q)_n(qx^{-1};q)_n(qy^{-1};q)_n},\\
    \theta(x;q)&=(-q^{\frac{1}{2}}x;q)_{\infty}(-q^{\frac{1}{2}}x^{-1};q)_{\infty}.
  \end{align}
\end{subequations}
To these series we wish to add an additional series which satisfies the inhomogenous
$q$-difference equations of the descendant coloured Jones polynomial~\eqref{recJm52x}.
This can be easily constructed using the deformations of the Habiro
polynomials~\eqref{defHab51}. We find a solution
\begin{equation}
  D_m(x,q)
  =
  -\sum_{n=-\infty}^{-1}q^{mn}
  H_{n}(q)x^{-n}(qx;q)_{n}(q^{-1}x;q^{-1})_{n}.
\end{equation}
(compare with Equation~\eqref{DJdef}) 
where $|q|<1$ and $m\in\BZ$ or $|q|>1$ and $m\in\BZ_{\geq0}$, and $H_{n}(q)$ is the
coefficient of $\ve^0\delta^0$ in the expansion of $H_{n}(\epsilon,\delta;q)$. In
particular, for $|q|<1$ we have
\be
  D_m(x,q)
  =
  -\sum_{n,k=0}^\infty(-1)^kq^{n(n+1)+k(k+1)/2-nk-(m+1)(n+1)}
  \frac{(q;q)_{n+k}}{(q;q)_k(q;q)_n(x^{-1};q)_{n+1}(x;q)_{n+1}}
\ee
and we see the $(x,q)$-series $D_0(x,q)$ coincides with
$f_{\knot{5}_2}(x,q)$ in \cite{Park:2020edg,park-inverted}.

This series can be included as the first row of a $6\times6$ matrix of $(x,q)$-series.
The latter might be related to the factorisation of the state integral proposed
in Section~\ref{sec:newstateint}.

However, we find that the matrices above and below the reals have different
quantum modular co-cycles related by inversion. This implies that to do a full
discussion on resurgence one needs to understand the monodromy of this $q$-holonomic
system. Both these issue will be explored in later publications. For now, we give a
description of the Stokes matrices restricted to $\tau$ in the upper half plane.

\subsection{$x$-version of Borel resummation and Stokes constants}
\label{sub.52borelx}

In this section we discuss the $x$-deformation version of Section~\ref{sub.52borel}. 
The asymptotic series $\Phi^{(\s_j)}(\tau)$ for $j=0,1,2,3$ are extended to
series $\Phi^{(\s_j)}(x;\tau)$ with coefficients in $\IZ(x^{\pm 1})$. The
series $\Phi^{(\s_j)}(x;\tau)$ for $j=1,2,3$ are defined in terms of
perturbation theory of a deformed state-integral \cite{AK} and they
have been computed with about 200 terms for many values of $x$ in
\cite{GGM:peacock}.  Let $\xi_j$ ($j=1,2,3$) be three roots to the
equation
\begin{equation}
  (1-\xi)(1-x\xi)(1-x^{-1}\xi) = \xi^2,
\end{equation}
ordered such that they reduce to \eqref{eq:xi} in the limit
$x\rightarrow 1$. The series $\Phi^{(\s_j)}(\tau)$ ($j=1,2,3$) can be
uniformly written as\footnote{The series $\Phi^{(\s_j)}(x;\tau)$
  ($j=1,2,3$) are related to the series in
  \cite{GGM:resurgent,GGM:peacock}, which we will denote by
  $\Phi^{(\s_j)}_{\text{GGM}}(x;\tau)$, by a common prefactor
  \begin{equation}
    \Phi^{(\s_j)}(x;\tau) = \ri
    \re^{-\frac{\pi\ri}{12}(\tau+\tau^{-1})-2\pi\ri\tau}
    \Phi^{(\s_j)}_{\text{GGM}}(x;\tau),\quad j=1,2,3.
  \end{equation}
  The Stokes constants associated to the Borel resummation of
  $\Phi^{(\s_j)}_{\text{GGM}}(\tau)$ are not changed.}  
\begin{equation}
  \Phi^{(\s_j)}(x;\tau) = \frac{\re^{\frac{3\pi\ri}{4}}}{\sqrt{\delta_j(x)}}
    \re^{\frac{V_j(x)}{2\pi\ri\tau}} \varphi^{(\s_j)}(x;\tau)
\end{equation}
where $\delta_j(x) = \xi_j-s \xi_j^{-1}+2 \xi_j^{-2}$
and
\begin{equation}
\label{Vx52}
  \begin{aligned}
    V_1(x) &=
      -\Li_2(\xi_1^{-1})-\Li_2(x\xi_1^{-1})-\Li_2(x^{-1}\xi_1^{-1})
      +\frac{1}{2}\log^2 x -\frac{1}{2}\log^2 \xi_1
      +\pi\ri\log\xi_1+\frac{2\pi^2}{3},\\
    V_2(x) &=
      -\Li_2(\xi_2^{-1})-\Li_2(x\xi_2^{-1})-\Li_2(x^{-1}\xi_2^{-1})
      +\frac{1}{2}\log^2 x -\frac{1}{2}\log^2 \xi_2
      -\pi\ri\log\xi_2+\frac{2\pi^2}{3},\\
    V_3(x) &=
      -\Li_2(\xi_3^{-1})-\Li_2(x\xi_3^{-1})-\Li_2(x^{-1}\xi_3^{-1})
      +\frac{1}{2}\log^2 x -\frac{1}{2}\log^2 \xi_3
      +3\pi\ri\log\xi_3+\frac{2\pi^2}{3}.
  \end{aligned}
\end{equation}
The power series $\varphi^{(\s_j)}(x;\tau)$ are 
\begin{align}
  \varphi^{(\s_j)}(x;\tfrac{\tau}{2\pi\ri}) =
  &1 + \frac{\tau}{12\delta_j(x)}\big((-397 - 94 s - 114 s^2 + 390 s^3 - 278 s^4 + 81 s^5 - 10 s^6)\nn
  &+(-381 + 623 s - 124 s^2 - 328 s^3 + 268 s^4 - 81 s^5 + 10 s^6)
    \xi_j\nn
  &+(-270 + 137 s + 182 s^2 - 207 s^3 + 71 s^4 - 10 s^5) \xi_j^2
    \big)+\ldots
\end{align}
with
\begin{equation}
  s = s(x) = x^{-1}+1+x.
\end{equation}

The additional series $\Phi^{(\s_0)}(x;\tau)$, as in
Section~\ref{sub.41Phi0xq}, can be computed either from the colored
Jones polynomial or by using Habiro's expansion of the colored Jones
polynomials. We find
\begin{equation}
  \Phi^{(\s_0)}(x;\tau) = \varphi^{(\s_0)}(x;\tau),
\end{equation}
where the power series $\varphi^{(\s_0)}(x;\tau)$ reads
\begin{equation}
  \phi^{(\s_0)}(x;\tfrac{\tau}{2\pi\ri}) = \frac{1}{2x+2x^{-1}-3}
  -\frac{(x^{1/2}-x^{-1/2})^2(5x+5x^{-1}-4)}{(2x+2x^{-1}-3)^3}\tau+\ldots,
\end{equation}

\begin{figure}[htpb!]
  \centering
  \subfloat[$\Lambda^{(\s_0)}$]%
  {\includegraphics[height=5cm]{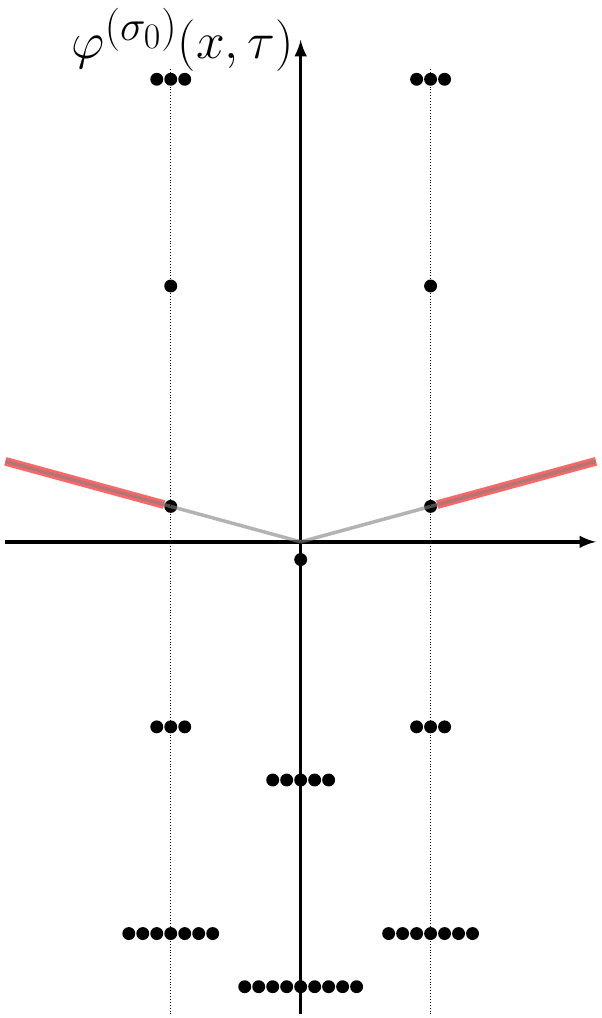}}\hspace{2ex}%
  \subfloat[$\Lambda^{(\s_1)}$]%
  {\includegraphics[height=5cm]{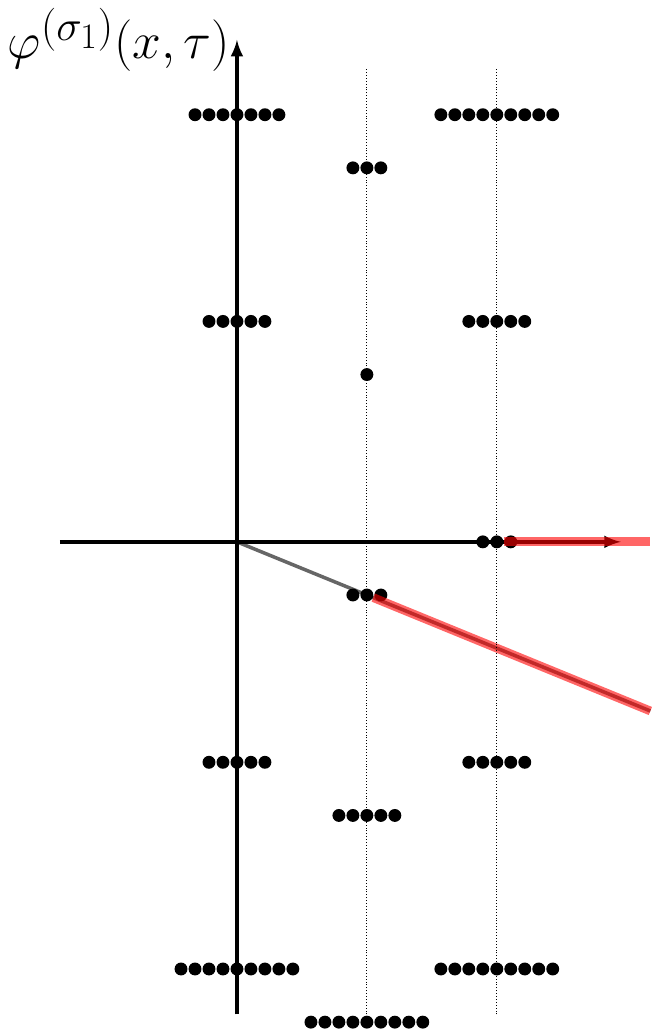}}\hspace{2ex}%
  \subfloat[$\Lambda^{(\s_2)}$]%
  {\includegraphics[height=5cm]{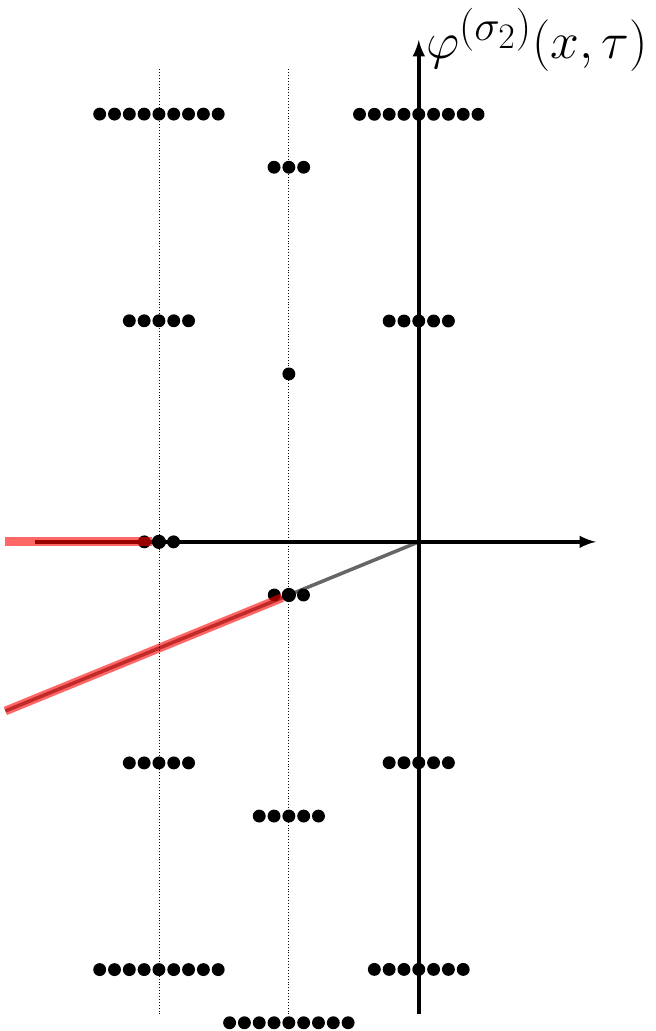}}\hspace{2ex}%
  \subfloat[$\Lambda^{(\s_3)}$]%
  {\includegraphics[height=5cm]{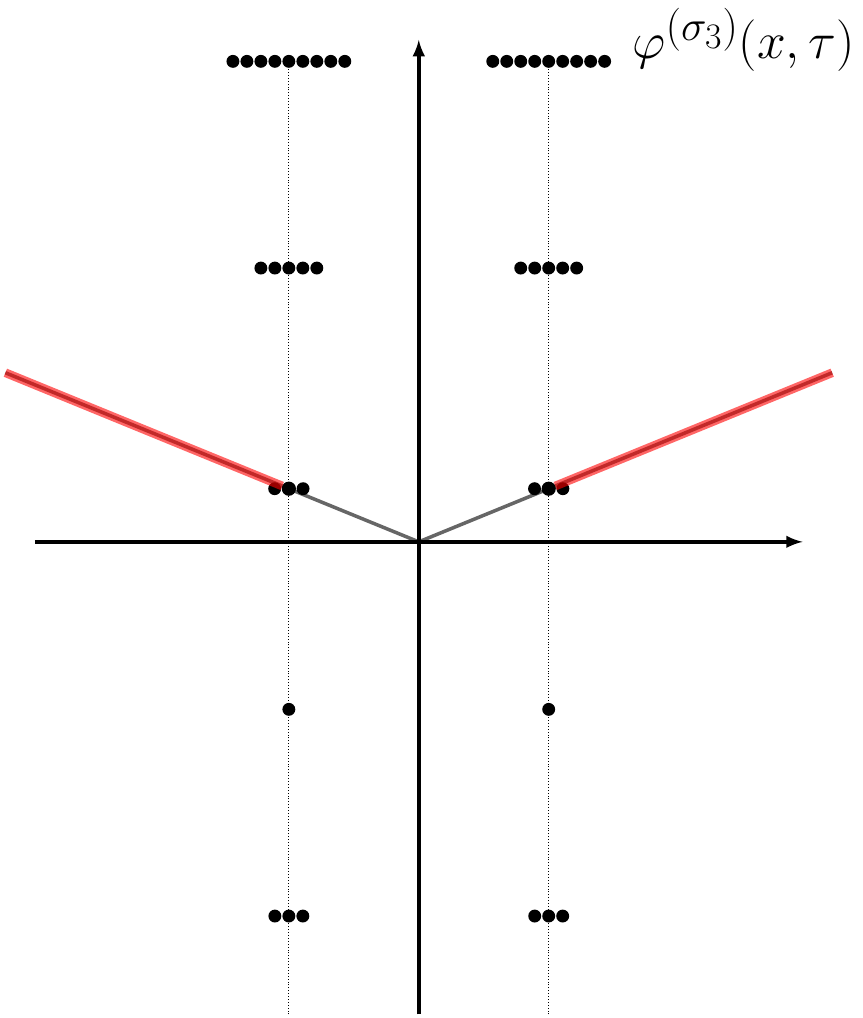}}\hspace{2ex}%
  \caption{Singularities of Borel transforms of
    $\varphi^{(\s_j)}(x,\tau)$ for $j=0,1,2,3$ of the knot
    $\knot{5}_2$.}
  \label{fg:52x_brplane}
\end{figure}

We are interested in the Stokes automorphisms in the upper half plane
of the Borel resummation of the 4-vector $\Phi(x;\tau)$ of asymptotic
series
\begin{equation}
  \Phi(x;\tau) =
  \begin{pmatrix}
    \Phi^{(\s_0)}(x;\tau)\\
    \Phi^{(\s_1)}(x;\tau)\\
    \Phi^{(\s_2)}(x;\tau)\\
    \Phi^{(\s_3)}(x;\tau)
  \end{pmatrix},
\end{equation}
when $x$ is close to 1.  The singular points of the Borel transform of
$\Phi(x;\tau)$, collectively denoted as $\Lambda(x)$, are smooth
functions of $x$ and they are equal to $\Lambda$ in \eqref{eq:Lmb52}
in the limit $x\rightarrow 1$.  When $x$ is slightly away from 1, each
singular point $\iota_{i,j}^{(k)}$ in $\Lambda$ splits to a finite set
of points located at
$\iota_{i,j}^{(k,\ell)}:=\iota_{i,j}^{(k)}+\ell \log(x)$,
$\ell\in\IZ$.  We illustrate this schematically in
Fig.~\ref{fg:52x_brplane}.  The complex plane of $\tau$ is divided by
rays passing through these singular points into infinitely many cones.
We will then pick the cones $I$ and $II$ located slightly above the
positive and negative real axes, and compute the global Stokes matrix
from cone $I$ to cone $II$ defined by
\begin{equation}
  \Delta(x,\tau) s_{II}(x,\tau) = \mf{S}_{I\rightarrow II}(\tx,\tq)
  \Delta(x,\tau) s_{I}(x,\tau),
  \label{eq:SRRx}
\end{equation}
where
\begin{equation}
  \Delta(x,\tau) =
  \diag\left(\tau^{1/2}\frac{x^{1/2}-x^{-1/2}}{\tx^{1/2}-\tx^{-1/2}},1,1,1\right).
\end{equation}
The global Stokes matrix $\mf{S}_{I\rightarrow II}(\tx,\tq)$
factorises uniquely into a product of local Stokes automorphisms
associated to each of the singular points in the upper half plane, from
which the individual Stokes constants can be read off.

The global Stokes matrix $\mf{S}_{I\rightarrow II}(\tx,\tq)$ in
\eqref{eq:SRRx} also has the block upper triangular form
\begin{equation}
  \begin{pmatrix}
    1&*&*&*\\
    0&*&*&*\\
    0&*&*&*\\
    0&*&*&*
  \end{pmatrix}.
\end{equation}
The $3\times 3$ sub-matrices $\mf{S}^{\text{red}}_{I\rightarrow II}$
in the right bottom have been worked out in \cite{GGM:peacock}, and
they are given by
\begin{subequations}
\begin{align}
\mf{S}^{\text{red}}_{I\rightarrow II}(\tx,\tq)
&=\frac{1}{2}
\begin{pmatrix}
0&1&0\\0&1&1\\-1&0&0
\end{pmatrix} \bJ^\RED_{-1}(\tx;\tq^{-1})^T
\begin{pmatrix}
1&0&0\\0&0&1\\0&1&0
\end{pmatrix} \bJ^\RED_{-1}(\tx;\tq)
\begin{pmatrix}
 0&0&-1\\1&-\ts&0\\0&1&0
\end{pmatrix},
\quad
|\tq|<1,\label{eq:SRed12x}
\end{align}
\end{subequations}
where
\begin{equation}
  \ts = s(\tx),
\end{equation}
and $\bJ^{\RED}(x,q)$ is given by~\eqref{Jred52x}. To calculate the first row of
$\mf{S}_{I\rightarrow II}(\tx,\tq)$, we use the additional holomorphic block $D_m(x,q)$.
\begin{conjecture}
  For every cone $R\subset \Lambda(x)$ and every $\tau \in R$, we have
  \begin{equation}
    D_0(x,q) = s_R(\Phi^{(\s_0)})(x;\tau) +
    \tau^{-1/2}\frac{\tx^{1/2}-\tx^{-1/2}}{x^{1/2}-x^{-1/2}}
    \sum_{j=1}^3 M_{R,j}(\tx,\tq)s_R(\Phi^{(\s_j)})(x;\tau),
    \label{eq:QMPhix}
  \end{equation}
  where $M_{R,j}(\tx,\tq)$ ($j=1,2,3$) are $\tq$-series 
  with coefficients in $\IZ(\tx^{\pm 1})$ depending on the
  cone $R$.
\end{conjecture}

We present $M_{R,j}(\tx,\tq)$ in terms of the row vector
$M_R(\tx,\tq) := (M_{R,1},M_{R,2},M_{R,3})(\tx,\tq)$, and it can be expressed
in terms of a $3\times 3$ matrix $M^{(\s_0)}_R(\tx,\tq)$
\begin{equation}
  M_R(\tx,\tq)
  = \left(\tq D_0(\tx,\tq),\tq^2D_1(\tx,\tq),\tq^3D_2(\tx,\tq)\right)
  M^{(\s_0)}_R(\tx,\tq).
  \label{eq:MMs0}
\end{equation}
\begin{conjecture}
  Equation \eqref{eq:QMPhi} holds in the cones $R=I,II$ where the
  $\tq$-
  series $M_{R,j}(\tq)$ are given in terms of $M_R^{(0)}(\tx,\tq)$ through
  \eqref{eq:MMs0} which are as follows
  \begin{subequations}
  \begin{align}
    M^{(\s_0)}_{I}(\tx,\tq) =
    &\begin{pmatrix}
       1&-1&-\ts\,\tq\\0&-1&-1+\ts\,\tq\\0&0&-\tq
     \end{pmatrix},\\
    M^{(\s_0)}_{II}(\tx,\tq) =
    &\begin{pmatrix}
       -1&1&-\ts\,\tq\\-1&0&-1+\ts\,\tq\\0&0&-\tq
     \end{pmatrix}.
  \end{align}
\end{subequations}
\end{conjecture}


Eqs.~\eqref{eq:QMPhix}, together with the reduced Stokes matrices
$\mf{S}^{\text{red}}_{I\rightarrow II}(\tx,\tq)$ for
$\Phi^{(\s_j)}(x;\tau)$ ($j=1,2,3$), allow us to calculate entries in
the first row of $\mf{S}_{I\rightarrow
  II}(\tx,\tq)$ 
by
\begin{equation}
  \mf{S}_{I\rightarrow II}(\tx,\tq)_{0,j} =
  M_{I,j}(\tx,\tq) - \sum_{k=1}^3M_{II,k}(\tx,\tq)
  \mf{S}^{\text{red}}_{I\rightarrow II}(\tx,\tq)_{k,j},\quad j=1,2,3.
\end{equation}
In the following we list the first few terms of these
$\tq$
-series.  
\begin{equation}
  \begin{aligned}
    \mf{S}_{I\rightarrow II}(\tx,\tq)_{0,1}=
    &-1+(1+\ts+\ts^2)\tq -(-2\ts-\ts^2+\ts^3)\tq^2-(1+\ts^4)\tq^3
      +\cO(\tq^4) ,\\
    \mf{S}_{I\rightarrow II}(\tx,\tq)_{0,2}=
    &1-(1+2\ts+\ts^2)\tq+(-\ts-\ts^2+2\ts^3)\tq^2+(3\ts^2+\ts^3+\ts^4)\ts^3
      +\cO(\tq^4) ,\\
    \mf{S}_{I\rightarrow II}(\tx,\tq)_{0,3}=
    &-\tq+(1+\ts^2)\tq^2+(3\ts+\ts^2)\tq^3+\cO(\tq^4).
  \end{aligned}
\end{equation}

Finally, we can factorise the global Stokes matrices
$\mf{S}_{I\rightarrow II}(\tx,\tq)$
to obtain local Stokes matrices associated to individual singular
points in $\Lambda$ and extract the associated Stokes constants.  The
Stokes constants for $\Phi^{(\s_j)}(x;\tau)$ ($j=1,2,3$) are already
given in \cite{GGM:resurgent,GGM:peacock}.  We collect the Stokes
contants for $\Phi^{(\s_0)}(x;\tau)$ in the generating series
\begin{equation}
  \ms{S}^+_{0,j}(\tx,\tq) = \sum_{k\geq 0}\sum_{\ell}
  \mc{S}^{(k,\ell)}_{0,j}\tx^\ell\tq^k,
  \quad j=1,2,3.
\end{equation}
And we find that 
\begin{equation}
  \begin{aligned}
    \ms{S}^+_{0,1}(\tx,\tq) =
    &-1+\tq+\ts\tq^2+(-2+3\ts^2)\tq^3+(2-\ts-2\ts^2+5\ts^3+2\ts^4)\tq^4
      +\cO(\tq^5),\\
    \ms{S}^+_{0,2}(\tx,\tq) =
    &1-\tq-\ts\tq^2-(-2+3\ts^2)\tq^3+(2-\ts-2\ts^2+5\ts^3+2\ts^4)\tq^4
      +\cO(\tq^5),\\
    \ms{S}^+_{0,3}(\tx,\tq) =
    &0.
  \end{aligned}
\end{equation}

\subsection{An analytic extension of the Kashaev invariant and the
  colored Jones polynomial}

In this section we discuss an analytic extension of the Kashaev invariant
and of the colored Jones polynomial of the $\knot{5}_2$ knot, illustrating
Conjectures~\ref{conj.ekashaev} and~\ref{conj.ejones}. 

Recall that the colored Jones polynomial of the $\knot{5}_2$ is given by
\begin{equation}
  J^{\knot{5}_2}_N(q) = \sum_{k=0}^{N-1}(-1)^k
  q^{-k(k+1)/2}(q^{1+N};q)_k(q^{1-N};q)_k H_k(q),\quad
  q = \re^{2\pi\ri \tau},
\end{equation}
where
\begin{equation}
  H_k(q) = (-1)^kq^{k(k+3)/2}\sum_{k=0}^\ell q^{\ell(\ell+1)}
  \frac{(q;q)_k}{(q;q)_\ell(q;q)_{k-\ell}}.
\end{equation}
Let $u$ be in a small neighborhood of the origin. It is related to $x = q^N$
and $\tau$ by
\begin{equation}
  x = \re^{u+2\pi\ri} = \re^{u},\quad \tau = \frac{u+2\pi\ri}{2\pi\ri N}.
  \label{eq:x-n-tau52}
\end{equation}
Then $x$ is close to $1$ and $\tau$ is close to $1/N$.
Note that
\begin{equation}
  N\tau = 1+\frac{u}{2\pi\ri}
\end{equation}
is the analogue of $n/k$ in \cite{gukov}, and here we are
considering a deformation from the case of $n/k=1$.  We also have
\begin{equation}
  \tx = \re^{\log(x)/\tau} =
  \exp\left(\frac{2\pi\ri N u}{u+2\pi\ri}\right).
\end{equation}
When $x$ is positive real, $\Phi^{(\s_1)}(x;\tau)$ are not Borel
summable along the positive real axis.  Depending on whether $\tau$ is
in the first or the fourth quadrant, we have 
\begin{subequations}
  \begin{align}
    J^{\knot{5}_2}_N(q) =
    s_{I}(\Phi^{(\s_0)})(x;\tau)
    +\tau^{-1/2}\frac{\tx^{1/2}-\tx^{-1/2}}{x^{1/2}-x^{-1/2}} 
    \big(
    &s_{I}(\Phi^{(\s_1)})(x;\tau)
      -(1+\tx)s_{I}(\Phi^{(\s_2)})(x;\tau)\nn
    &-(1+\tx)s_{I}(\Phi^{(\s_3)})(x;\tau)
      \big)
      \label{eq:Jn-s1-52}
    \\  =
    s_{IV}(\Phi^{(\s_0)})(x;\tau)  +\tau^{-1/2}
    \frac{\tx^{1/2}-\tx^{-1/2}}{x^{1/2}-x^{-1/2}} 
    \big(
    &s_{IV}(\Phi^{(\s_1)})(x;\tau)
      +(1+\tx^{-1})s_{IV}(\Phi^{(\s_2)})(x;\tau)\nn
    &-(1+\tx)s_{IV}(\Phi^{(\s_3)})(x;\tau)
      \big).
      \label{eq:Jn-s2-52}
    \end{align}
\end{subequations}    
  The two equations~\eqref{eq:Jn-s1-52}, \eqref{eq:Jn-s2-52} are related by
  the Stokes discontinuity formula
  \begin{equation}
    \text{disc}_0 \Phi^{(\s_1)}(x;\tau) =
    s_I(\Phi^{(\s_1)})(x;\tau) - s_{IV}(\Phi^{(\s_1)})(x;\tau)= 
    (2+\tx+\tx^{-1}) s(\Phi^{(\s_2)})(x;\tau).
  \end{equation}
Combined, they imply 
\begin{align}
\label{52ejones}
  J_N^{\knot{5}_2}(q) =
  s_{\text{med}}(\Phi^{(\s_0)})(x;\tau)+\tau^{-1/2}
  \frac{\tx^{1/2}-\tx^{-1/2}}{x^{1/2}-x^{-1/2}} 
  \big(
  &s_{\text{med}}(\Phi^{(\s_1)})(x;\tau)
    -(1+\tx)s_{\text{med}}(\Phi^{(\s_3)})(x;\tau)\nn
  &-\frac{\tx-\tx^{-1}}{2}s_{\text{med}}(\Phi^{(\s_2)})(x;\tau)
    \big) 
\end{align}
which is the assertion of Conjecture~\ref{conj.ejones}.

\subsection{A new state-integral for the ${\knot{5}_2}$ knot?}
\label{sec:newstateint}

In the case of the figure eight knot, the new state-integral was obtained by first 
writing down an integral formula for its colored Jones polynomial, in Habiro form,
and then changing the integration contour to pick the contribution from the poles
in the lower half plane. This led in particular to the ``inverted'' Habiro series
$\mc{C}_0(x,q)$ in (\ref{eq:Ixq}). Although we do not have a similar complete
theory for the ${\knot{5}_2}$ knot, we can however write down an integral formula
for its colored Jones polynomial which lead, after a change of contour, to the
corresponding inverted Habiro series. In fact, it is possible to write such an
integral for all twist knots $K_p$ (the ${\knot{5}_2}$ knot corresponds to $p=2$). 

Let us then consider the colored Jones polynomial of the twist knot $K_p$ in
Habiro's form \cite{Mas:HabPol}:
\be
\label{gen-habiro}
J^{K_p}_N (q;x)= \sum_{n=0}^{N-1} C^{K_p}_n (q) (qx;q)_n (q x^{-1};q)_n, 
\ee
where
\be
C^{K_p}_n(q)= -q^n \sum_{k=0}^n (-1)^k
q^{(p+1/2)k(k+1) +k } (q^{2k+1}-1) {(q;q)_n \over (q;q)_{n+k+1} (q;q)_{n-k}}. 
\ee
It is easy to see that (\ref{gen-habiro}) can be written as a double contour integral 
\be
\int_{\mc{A}_z} \int_{\mc{A}_w} \mc{I}_{K_p}(z,w) {\rm d} z {\rm d} w, 
\ee
where 
\be
\begin{aligned}
  &\mc{I}_{K_p} (z, w)=-\fad^{-1} \left(z-{\ri \over 2 \mb} + u \right)
  \fad^{-1} \left(z-{\ri \over 2 \mb} - u\right) \fad ^{-1}
  \left(z-{\ri \over 2 \mb} \right) \fad
  \left(z-w+ {\ri \mb \over 2} -{\ri \over  2 \mb} \right)  \\
& \times 
\fad \left(z+w+ {\ri \mb \over 2} -{\ri \over  2\mb}  \right)
\re^{-2 \pi \ri (p+1/2)(w+ {\ri \over 2 \mb})^2} \left( \re^{ 2 \pi \mb (z+w)}
  - \re^{ 2 \pi \mb (z-w)} \right) \tanh\left( {\pi z \over \mb} \right)
\tanh\left( {\pi w \over \mb} \right),
\end{aligned}
\ee
and the contours $\mc{A}_{z,w}$ encircle the poles of the form (\ref{jones-poles}) in
the upper complex planes of the $z$ and the $w$ variables, respectively. We can now
deform the contour to pick the poles in the lower half planes of $z$, $w$. The
contribution from the simple poles of the $\tanh$ functions in those half planes can
be easily computed, and one finds in this way the inverted Habiro series, 
\be
\label{CKp}
\begin{aligned}
  \mc{C}_{K_p} (q, x)&={1 \over (x^{1\over 2}- x^{-{1\over 2}})^2}
  \sum_{n \ge 0} { q^{n (n+1)/2} \over (q x;q)_n (q x^{-1};q)_n} \\
  & \times \sum_{k \ge n}  q^{n (n+1)/2 + (p+1/2) k(k+1)- (n+k)(n+k+1)/2-n}
  (q^k - q^{-k-1}) {(q;q)_{n+k} \over (q;q)_n (q;q)_{k-n}}. 
\end{aligned}
\ee
This gives a general formula for all twist knots which agrees with the results
of \cite{park-inverted} for $p=2$ (the ${\knot{5}_2}$ knot) and $p=3$ 
(the ${\knot{7}_2}$ knot). 

It might be possible to find appropriate integration contours so that the integral
of $\mc{I}_{K_p} (z, w)$ converges and provides the sought-for new state-integral
which sees the series $\Phi^{(\sigma_0)} (x, \tau)$, as it happened in the case of the
${\knot{4}_1}$ knot. In the case of the ${\knot{5}_2}$ knot, these contours do exist
and lead to a well-defined integral. We expect that an evaluation of such an integral
by summing over the appropriate set of residues will give the inverted Habiro
series \eqref{CKp}, together with additional contributions, as in \eqref{eq:Ixq}.
However, the fact that the integrals are two-dimensional makes them more difficult
to analyze. We expect to come back to this problem in the near future. 


\appendix

\section{$q$-series identities}
\label{Appqserids}

In this appendix we will sketch the proofs of some $q$-series identities
that appear in Conjecture~\ref{DS2hid}. Since this not one of the main
themes of the paper, our presentation will be rather brief. Our proofs of the
$q$-hypergeometric identities will use the algorithmic approach of the Wilf-Zeilberger
theory (see \cite{WZ,PWZ}) and the computer implementation by
Koutschan~\cite{Koutschan:holofunctions}). 


We outline part of the proof of Conjecture~\ref{DS2hid} for $|q|<1$, namely
\be
\label{3Q}
\begin{aligned}
  q\DS^{(0,0)}_{0}(q)
  &=
  (q;q)_{\infty}\sum_{n=0}^{\infty}\frac{q^{n(n+1)}}{(q;q)_{n}^{3}}\\
  q\DS_{1}^{(0,0)}(q)
  &=
  (q;q)_{\infty}\sum_{n=0}^{\infty}(2-q^{n})\frac{q^{n(n+1)}}{(q;q)_{n}^{3}}\\
  q\DS_{2}^{(0,0)}(q)
  &=
  (q;q)_{\infty}\sum_{n=0}^{\infty}
  \left((3+q^{-1})-(2+2q^{-1})q^{n}+q^{2n-1}\right)\frac{q^{n(n+1)}}{(q;q)_{n}^{3}}
\end{aligned}
\ee
\begin{proof}
The definition of $\DS_{m}^{(0,0)}(q)$ gives that
\be
q\DS_{m}^{(0,0)}(q)= f_{-m-1,0}(q)
\ee
where
\be
\label{fmp}
\begin{aligned}
  f_{m,p}(q)
  & =
  \sum_{n,k=0}^{\infty}(-1)^{k}q^{n(n+1)+k(k+1)/2-nk+mn+pk}
  \frac{(q;q)_{n+k}}{(q;q)_{n}^{3}(q;q)_{k}}
  \\
  & =\sum_{k=0}^{\infty}f_{m,p,k}(q)
\end{aligned}
\ee
with
\be
\label{fmpk}
  f_{m,p,k}(q)
  =
  \sum_{n=0}^{\infty}(-1)^{k}q^{n(n+1)+k(k+1)/2-nk+mn+pk}
  \frac{(q;q)_{n+k}}{(q;q)_{n}^{3}(q;q)_{k}} \,.
\ee
Likewise, we define
\be
\label{hmp}
\begin{aligned}
  h_{m,p}(q)
  &=
  (q;q)_{\infty}\sum_{n=0}^{\infty}\frac{q^{n(n+1)+pn+mn+mp}}{(q;q)_{n+p}^{2}(q;q)_{n}}
  \\
  &=
  \frac{1}{(q;q)_{\infty}}\sum_{n,k,\ell=0}^{\infty}(-1)^{k+\ell}\frac{q^{n(n+1)+k(k+1)/2
      +\ell(\ell+1)/2+pn+pm+pk+p\ell+nk+mn}}{(q;q)_{n}(q;q)_{k}(q;q)_{\ell}}
  \\
  &=\sum_{k=0}^{\infty}h_{m,p,k}(q)
\end{aligned}
\ee
where
\be
  h_{m,p,k}(q)
  =
  \frac{1}{(q;q)_{\infty}}\sum_{n+j+\ell+m=k}^{\infty}(-1)^{j+\ell}
  \frac{q^{n(n+1)+j(j+1)/2
  +\ell(\ell+1)/2+pn+pm+pj+p\ell+nj+mn}}{(q;q)_{n}(q;q)_{j}(q;q)_{\ell}} \,.
\ee
Therefore, we have
\be
  f_{-1,0}(q)
  =
  q\DS^{(0,0)}_{0}(q)
  \quad
  \text{and}
  \quad
  h_{0,0}(q)
  =
  (q;q)_{\infty}\sum_{n=0}^{\infty}\frac{q^{n(n+1)}}{(q;q)_{n}^{3}} \,.
\ee
This implies that the first equality in~\eqref{3Q} follows from the $p=0$
specialization of 
\be
  f_{-1,p}(q)=h_{0,p}(q), \qquad (p \in \BZ) \,.
\ee
This in turn follows (using Equations~\eqref{fmp} and~\eqref{hmp}) from
the following
\be
\label{f=h}
  f_{-1,p,k}(q)=h_{0,p,k}(q), \qquad (p \in \BZ, k \in \BN) \,.
\ee
Equation~\eqref{fmpk} expresses the two-variable $q$-holonomic function
$f_{-1,p,k}(q)$ as a one dimensional sum of a three variable proper $q$-hypergeometric
function. It follows from~\cite{Koutschan:holofunctions} that the annihilator ideal
of $F_{k,p}(q):=f_{-1,p,k}(q)$ is generated by the recursion relations
{\small
  \be
-q^k F_{p, k}(q) + F_{1 + p, k}(q)=0,
\ee
\be
q^{2 + k + 2 p} (-1 + q^{1 + k})^2 F_{p, k}(q) + 
q^{2 + k + p} (-3 + q^{1 + k} + q^{2 + k}) F_{p, 1 + k}(q)
+ (-1 + q^{2 + k}) F_{p, 2 + k}(q)=0
\ee
}
This coincides with the annihilator ideal of $h_{0,p,k}(q)$. Thus, the
equality~\eqref{f=h} for $p,k \in \BZ$ with $k \geq 0$ follows from the
two special cases $(p,k)=(0,0)$ and $(p,k)=(0,1)$, that is from the identities
\be
\label{euler}
\begin{aligned}
  \sum_{n=0}^{\infty}\frac{q^{n^{2}}}{(q;q)_{n}^{2}}
  & =
  \frac{1}{(q;q)_{\infty}}
  \\
  \frac{1}{1-q}\sum_{n=0}^{\infty}\frac{q^{n^{2}-n+1}(q^{n+1}-1)}{(q;q)_{n}^{2}}
  & =
  \frac{q^{2}-2q}{(q;q)_{\infty}(1-q)}
\end{aligned}  
\ee
The first one of the above identities is due to Euler and can be derived using
generating functions of partitions. The second one follows from the $q$-holonomic
system
\be
  g_{m}(q)=\sum_{n=0}^{\infty}\frac{q^{n^{2}+mn}}{(q;q)_{n}^{2}}
  \qquad\text{with}\qquad
  g_{m}(q)-2g_{m+1}(q)+(1-q^{m+1})g_{m+2}(q)=0.
\ee
This concludes the proof of the first identity
in~\eqref{3Q}. The remaining two identities follow (using the above steps)
from the following ones
\be
\label{Q23}
\begin{aligned}
  f_{-2,p,k}(q)
  &=
  2h_{0,p,k}(q)-h_{1,p,k}(q),\\
  f_{-3,p,k}(q)
  &=
  (3+q^{-1})h_{0,p,k}(q)-(2+2q^{-1})h_{1,p,k}(q)+q^{-1}h_{2,p,k}(q).
\end{aligned}
\ee
This concludes the sketch of the proof of~\eqref{3Q}. 
\end{proof}
In the course of the proof, we came up with the following conjecture which
expresses $f_{m,p}(q)$ as $\BZ[q^{\pm 1}]$-linear combinations of $h_{m,p}(q)$.

\begin{conjecture}
  For $m \geq 0$ we have:
  \be
\begin{aligned}
  f_{m,p}(q)
  &=
  \sum_{k,i=0}^{\infty}(-1)^iq^{i(i+1)/2+k}
  \frac{(q;q)_{m+k+i}}{(q;q)_{m}(q;q)_{i}(q;q)_{k}}h_{k,p}(q),\\
  f_{-1-m,p}(q)
  &=
  \sum_{k=0}^{m}\sum_{i=0}^{m-k}(-1)^iq^{i(i+1)/2+k}
  \frac{(q^{-1};q^{-1})_{m}}{(q^{-1};q^{-1})_{m-i-k}(q;q)_{i}(q;q)_{k}}h_{k,p}(q)
  \,.
\end{aligned}
\ee
\end{conjecture}


\bibliographystyle{hamsalpha}
\bibliography{biblio}
\end{document}